\documentclass[12pt]{article}

\usepackage{latexsym,mathrsfs,bm}
\usepackage[english]{babel}
\usepackage{paralist}
\usepackage{ulem} 
\usepackage{amsfonts,amssymb,amsmath}
\usepackage{amsthm}
\usepackage[latin1]{inputenc}
\usepackage{bm}

\newcommand{\eps}{\varepsilon}
\newcommand{\R}{\mathbb{R}}
\newcommand{\Q}{\mathbb{Q}}
\newcommand{\C}{\mathbb{C}}

\newcommand{\Z}{\mathbb{Z}}

\newcommand{\es}[1]{\begin{equation}\begin{split}#1\end{split}\end{equation}}
\newcommand{\est}[1]{\begin{equation*}\begin{split}#1\end{split}\end{equation*}}
\newcommand{\as}[1]{\begin{align}#1\end{align}}
\newcommand{\astr}[1]{\begin{align*}#1\end{align*}}

\newcommand{\B}{\mathrm{B}}

\newcommand{\tn}[1]{\textnormal{#1}}

\renewcommand{\mod}[1]{~\pr{\textnormal{mod}~#1}}
\newtheorem*{theo*}{Theorem}
\newtheorem{theo}{Theorem}

\newtheorem{lemma}{Lemma}
\newtheorem{corol}[lemma]{Corollary}
\newtheorem{remark}{Remark}

\newtheorem*{rem*}{Remark}
\def\sumstar{\operatornamewithlimits{\sum\nolimits^*}}
\def\sumdagger{\operatornamewithlimits{\sum\nolimits^\dagger}}

\newcommand{\pr}[1]{\left( #1\right)}
\newcommand{\prst}[1]{( #1)}
\newcommand{\prbig}[1]{\big( #1\big)}
\newcommand{\prBig}[1]{\Big( #1\Big)}
\newcommand{\prbigg}[1]{\bigg( #1\bigg)}

\newcommand{\pg}[1]{\left\{ #1\right\}}
\newcommand{\pmd}[1]{\left| #1\right|}

\newcommand{\sgn}{\operatorname{sgn}}
\newcommand{\e}[1]{\operatorname{e}\pr{ #1}}
\newcommand{\cc}{\operatorname{c}}
\newcommand{\dks}{\operatorname{s}}

\newcommand{\balpha}{\boldsymbol\alpha}
\newcommand{\bbeta}{\boldsymbol\beta}

\newcommand{\bv}{\boldsymbol{v}}

\newcommand{\bdu}{\boldsymbol{du}}
\newcommand{\bcu}{\boldsymbol{c_u}}
\newcommand{\bdv}{\boldsymbol{dv}}

\newcommand{\df}{\mathrm{d}}

\newcommand{\comment}[1]{}

\let\originalleft\left
\let\originalright\right
\renewcommand{\left}{\mathopen{}\mathclose\bgroup\originalleft}
\renewcommand{\right}{\aftergroup\egroup\originalright}

\numberwithin{equation}{section}

\makeatletter
\newcommand{\subjclass}[2][2010]{%
  \let\@oldtitle\@title%
  \gdef\@title{\@oldtitle\footnotetext{#1 \emph{Mathematics subject classification.} #2}}%
}
\newcommand{\keywords}[1]{%
  \let\@@oldtitle\@title%
  \gdef\@title{\@@oldtitle\footnotetext{\emph{Key words and phrases.} #1.}}%
}
\makeatother
\newcommand{\addresses}{{
  \bigskip
  \footnotesize

  S.~Bettin, \textsc{Dipartimento di Matematica, Universit\`a di Genova; via Dodecaneso 35; 16146 Genova; Italy. 
}\par\nopagebreak
  \textit{E-mail address}: \texttt{bettin@dima.unige.it}

 }}

\begin{document}

\author{Sandro Bettin}
\title{High moments of the Estermann function}
\date{}

\maketitle

\begin{abstract}
For $a/q\in\Q$ the Estermann function is defined as $D(s,a/q):=\sum_{n\geq1}d(n)n^{-s}\operatorname{e}(n\frac aq)$ if $\Re(s)>1$ and by meromorphic continuation otherwise. For $q$ prime, we compute the moments of $D(s,a/q)$ at the central point $s=1/2$, when averaging over $1\leq a<q$.

As a consequence we deduce the asymptotic for the iterated moment of Dirichlet $L$-functions $\sum_{\chi_1,\dots,\chi_k\mod q}|L(\frac12,\chi_1)|^2\cdots |L(\frac12,\chi_k)|^2|L(\frac12,\chi_1\cdots \chi_k)|^2$, obtaining a power saving error term.

Also, we compute the moments of certain functions defined in terms of continued fractions. For example, writing $f_{\pm}(a/q):=\sum_{j=0}^r (\pm1)^jb_j$ where $[0;b_0,\dots,b_r]$ is the continued fraction expansion of $a/q$ we prove that for $k\geq2$  and $q$ primes one has $\sum_{a=1}^{q-1}f_{\pm}(a/q)^k\sim2 \frac{\zeta(k)^2}{\zeta(2k)}  q^k$  as $q\to\infty$.

\end{abstract}

\section{Introduction} 

Since the pioneering work of Hardy and Littlewood~\cite{HL}, the study of moments of families of L-functions has gained a central role in number theory. This is mostly due their numerous applications on, e.g., non-vanishing (see~\cite{IS},~\cite{Sou00}) and sub-convexity estimates (cf.~\cite{CI}). Moreover, moments are also important as they highlight clearly the symmetry of each family.

In this paper we consider the moments of the Estermann function at the central point and, as a consequence, we obtain new results for moments of Dirichlet $L$-functions. We will describe the Estermann function in Section~\ref{mest}, we now focus on the the family of Dirichlet $L$-functions. For this family only the second and fourth moments have been computed. The asymptotic for the second moment was obtained by Paley~\cite{Pal}, whereas Heath-Brown~\cite{H-B81} considered the fourth moment and showed
\es{\label{vasc}
\frac1{\varphi^*(q)}\sumstar_{\chi\mod q}|L(1/2,\chi)|^{4}\sim
\frac{1}{2\pi^2}\prod_{p|q}\frac{(1-1/p)^3}{1+1/p}(\log q)^{ 4},
}
provided that $q$ doesn't have ``too many prime divisors'', a restriction that was later removed by Soundararajan~\cite{Sou07}. As usual, $\sum^*$ indicates that the sum is restricted to primitive characters and $\varphi^*(q)$ denotes the number of such characters. The problem of computing the full asymptotic expansion for the fourth moment was later solved by Young~\cite{You11b} in the case when $q$ is prime. He proved
\es{\label{aafy}
\frac1{\varphi^*(q)}\sumstar_{\chi\mod q}|L(1/2,\chi)|^{4}=
\sum_{i=0}^4c_i(\log q)^i+O\pr{q^{-\frac{5}{512}+\eps}}
}
for some absolute constants $c_i$ with $c_4=(2\pi^2)^{-1}$. Recently, Blomer, Fouvry, Kowalski, Michel and Mili\'cevi\'c~\cite{BFKMM} introduced several improvements in Young's work improving the error term in~\eqref{aafy} to $O(q^{-\frac{1}{32}+\eps})$.

In this paper, we consider a variation of this problem and compute the asymptotic of 
\es{\label{dfmkk}
M_{k}(q)
& =\frac1{\varphi^*(q)^{k-1}}\sumstar_{\substack{\chi_1,\dots,\chi_{k-1}\mod q}}|L(1/2,\chi_1)|^2\cdots |L(1/2,\chi_{k-1})|^2|L(1/2,\chi_1\cdots\chi_{k-1})|^2,
}
where the sum has the extra restriction that $\chi_1\cdots\chi_{k-1}$ is primitive. If $k=2$, this coincides with the usual $4$-th moment of Dirichlet $L$-functions as computed by Young, whereas if $k>2$ then $M_k(q)$ should be thought of as an iterated $4$-th moment, since each character appears $4$ times in the above expression. We shall prove the following theorem.
\begin{theo}\label{1tp}
Let $k\geq3$ and let $q$ be prime. Then, there exists an absolute constant $A>0$ such that 
\est{
M_{k}(q)=\sum_{n=1}^{\infty}\frac{2^{\nu(n)}}{n^{\frac k2}} \pr{(\log\tfrac q{8n\pi })^{k}+(-\pi)^{k}}+O_\eps\pr{k^{Ak}q^{-\delta_k+\eps}},
}
where $\nu(n)$ is the number of different prime factors of $n$, $\delta_k:=\frac{k-2-3 \vartheta}{2 k+5}$ 
with $\vartheta=\frac 7{64}$ being the best bound towards Selberg's eigenvalue conjecture. Also, the implicit constant depends on $\eps$ only. 
\end{theo}
\begin{remark}
Notice that $\delta_k$ is a decreasing sequence such that $\delta_k\rightarrow\frac12$ as $k\rightarrow\infty$. For $\vartheta=\frac 7{64}$ the first few values of $\delta_k$ are $\delta_3=\frac {43}{704}$, $\delta_4=\frac{107}{832}$, $\delta_5=\frac{57}{320}$.
\end{remark}
Theorem~\ref{1tp} yields an asymptotic formula for $M_k(q)$ for $k<\eta \frac {\log q}{\log \log q}$ with $\eta>0$ sufficiently small. Larger values of $k$ are easier to deal with and one obtains the following corollary.
\begin{corol}\label{cqa}
Let $q$ be prime. Then as $q\rightarrow\infty$ we have
\es{\label{3c4}
M_{k}(q)&  \sim \frac{ \zeta(\frac k2)^2}{\zeta(k)}(\log (q/8\pi)+\gamma)^k,\\
}
uniformly in $3\leq k=o(q^\frac12\log q)$. Moreover this range is optimal, meaning that~\eqref{3c4} is false if $k\gg q^\frac12\log q$.

\end{corol}

A moment somewhat similar to~\eqref{dfmkk} was  previously considered by Chinta~\cite{Chi} who used a multiple Dirichlet series approach to compute the asymptotic of the first moment of (roughly) 
\es{\label{vaa}
L(1/2,\chi_{d_1})L(1/2,\chi_{d_2})L(1/2,\chi_{d_1}\chi_{d_2}),
}
where $\chi_{d}$ denotes the quadratic character associated to the extension $\Q(\sqrt{d})$ of $\Q$. We remark there is a big difference between~\eqref{dfmkk} and this case. Indeed, if $\chi_1,\chi_2$ are characters modulo $q$ then so is $\chi_1\chi_2$
, whereas if $d_1,d_2\approx X$ then $\chi_{d_1}\chi_{d_2}$ is typically a character with conductor $\approx X^2$. This means that~\eqref{dfmkk} roughly correspond to an iterated $4$-th moment, whereas the second moment of~\eqref{vaa} roughly correspond to an iterated $6$-th moment of quadratic Dirichlet $L$-functions, and thus it doesn't seem to be attackable with the current technology. (As a comparison, the first moment computed by Chinta roughly correspond to an iterated $3$-rd moment).

\subsection{Twisted moments, the Estermann function, and continued fractions}

A nice feature of Theorem~\ref{1tp} is that it can be essentially rephrased in terms of high moments of other functions appearing naturally in Number Theory. Indeed, the same computations give also the asymptotic for moments of twisted moments of Dirichlet $L$-functions, of the Estermann function, and of certain functions defined in terms of continued fractions. We now briefly describe each of this objects and give the corresponding version of Theorem~\ref{1tp}.

\subsubsection{Moments of twisted moments}\label{mtm}

Several classical methods to investigate the central values of Dirichlet $L$-functions pass through the study of the second moment of $L(s,\chi)$ times a Dirichlet polynomial $P_{\theta}(s,\chi):=\sum_{n\leq q^{\theta}}a_n\cdot\chi(n)n^{-s}$:
\es{\label{adcf}
\frac1{\varphi^*(q)}\sumstar_{\chi \mod q}|L(1/2,\chi)P_{\theta}(1/2,\chi)|^{2}.
} 
For example, Iwaniec and Sarnak proved that $1/3$ of the Dirichlet $L$-functions do not vanish at the central point via proving the asymptotic for such average for $\theta<1/2$ (and choosing $P_{\theta}$ to be a mollifier). Moreover, it is easy to see that if one could extend such asymptotic to all polynomials of length $\theta<1$, then the Lindel\"of hypothesis would follow.

Expanding the square, using the multiplicativity of Dirichlet characters, and renormalizing, one immediately sees that~\eqref{adcf} can be reduced to an average of twisted moments of the form
\est{
M\pr{a,q}:=\frac{q^{\frac12}}{\varphi^*(q)}\sumstar_{\chi \mod q}\pmd{L\pr{\tfrac12,\chi}}^2\chi(a),
}
for $(a,q)=1$. By the orthogonality of Dirichlet characters one can immediately rewrite Theorem~\ref{1tp} (and~\eqref{vasc}) in terms of $M\pr{a,q}$. In particular, one has
\es{\label{tfm}
\sum_{a=1}^qM\pr{a,q}^k&=q^{\frac k2}M_k(q)\sim \begin{cases}
 \frac{ \zeta(\frac k2)^2}{\zeta(k)}(\log (q/8\pi)+\gamma)^k & \tn{if } 3\leq k=o(q^\frac12\log q),\\
 \frac1{2\pi^2}(\log q)^4 &\tn{if } k=2,
\end{cases}
}
as $q\rightarrow\infty$ with $q$ prime.

\subsubsection{Moments of the Estermann function}\label{mest}
For $(a,q)=1$, $q>0$, $\alpha,\beta\in\C$ and $\Re(s)>1-\min(\Re(\alpha),\Re(\beta))$, the Estermann function is defined as 
\es{\label{dfnes}
D_{\alpha,\beta}(s,\alpha, a/q):=\sum_{n=1}^\infty \e{n{a}/q}\frac{\tau_{\alpha,\beta}(n)}{n^s}=D_{\cos;\alpha,\beta}(s, a/q)+iD_{\sin;\alpha,\beta}(s, a/q),
}
where $D_{\cos}$ and $D_{\sin}$ have the same definition of $D$, but with $\e{na/q}$ replaced by $\cos(2\pi na/q)$ and $\sin(2\pi na/q)$ respectively. As usual, $\e{x}:=e^{2\pi i x}$ and $\tau_{\alpha,\beta}(n):=\sum_{d_1d_2=n}d_1^{-\alpha}d_2^{-\beta}$.

$D_{\alpha,\beta}(s, a/q)$ was first introduced (with $\alpha=\beta=0$) by Estermann who proved that it extends to a meromorphic function on $\C$ satisfying a functional equation relating $D_{\alpha,\beta}(s,a/q)$ with $D_{-\alpha,-\beta}(1-s,\pm\overline a/q)$, where $\overline a$ denotes the inverse of $a$ modulo $q$ (and similarly for $D_{\sin}$ and $D_{\cos}$ which satisfy a more symmetric functional equation given by~\eqref{afce} below).  

Since the work of Estermann~\cite{Est30,Est32} on the number of representations of an integer as a sum of two or more products, the Estermann function has proved itself as a valuable tool when studying additive problems of similar flavor (see e.g.~\cite{Mot80,Mot94}) and in problems related to moments of $L$-functions (see e.g.~\cite{H-B79,You11b} and~\cite{CGG}).
These applications mainly use the functional equation for $D$ as it encodes Voronoi's summation in an analytic fashion, allowing for a simpler computation of the main terms. However, the Estermann function is an interesting object by its own right, due to its surprising symmetries (see~\cite{Bet16}) and to the connections with some interesting objects in analytic number theory. For example, by the work of Ishibashi~\cite{Ish} (see also~\cite{BC13b}) one has
\est{
D_{\sin;1,0}(0,a/q)
=\pi \dks(a,q),\qquad D_{\sin;0,0}(0, a/q)&=\tfrac12\cc_0\pr{ a/q},
}
 where $\dks(a,q)$ is the classical Dedekind sum and $\cc_0\pr{ a/q}$ is a cotangent sum, related to the Nymann-Beurling criterion for the Riemann hypothesis, which has been object of intensive studies in recent years (see, for example,~\cite{BC13a,MR,Bet15}). Ishibashi obtained similar identities also for other values of $\alpha,\beta$, and in particular if $\alpha$ is a positive odd integer one obtains that $D_{\sin;\alpha,0}(0,a/q)$ is related to certain Dedekind cotangent sums studied by Beck~\cite{Bec}. All these functions satisfy certain reciprocity relations and provide examples of ``quantum modular forms'' (cf.~\cite{Zag}).

Moreover, one can also obtain formulae relating the Estermann function to twisted moments of Dirichlet $L$-function (see~\cite{Bet16} and \cite{CG}) and in particular for $q$ prime and $(a,q)=1$, one has
\es{\label{axe}
D_{\cos;0,0}(\tfrac12,a/q)+D_{\sin;0,0}(\tfrac 12,a/q)=M(a,q)+\frac{2(q^{\frac12}-1)}{\varphi(q)}\zeta(\tfrac12)^2.
}
By this formula and~\eqref{tfm}, it is clear that Theorem~\ref{1tp} gives an asymptotic formula for the high moments of $D_{\cos;0,0}(1/2,a/q)+D_{\sin;0,0}(1/2,a/q)$. The method however allows one to obtain the asymptotic for the joint moments of $D_{\cos;0,0}(1/2,a/q)$ and $D_{\sin;0,0}(1/2,a/q)$. We shall state this in~Theorem~\ref{mtws} below where shifts are also included (all our results will be derived from this theorem). Here we content ourselves with giving the asymptotic for the high moments of the Estermann function:
\begin{theo}\label{caee}
Let $q$ be prime. Then,  
\est{
\frac1{\varphi(q)}\sum_{a=1}^{q-1}D_{0,0}(\tfrac12,\tfrac aq)^k&\sim q^{\frac k2-1}2^{1-\frac k2}\frac{\zeta(\frac k2)^2}{\zeta(k)}\Re \pr{\pr{e^{\frac {\pi i}4}(\log\tfrac q{8\pi }+\gamma)-e^{-\frac {\pi i}4}\tfrac \pi2}^{k}}\\
}
as $q\rightarrow \infty$, uniformly in $3\leq k=o(q^\frac12\log q)$. In particular, if $3\leq k\ll1$ then
\est{
\frac1{\varphi(q)}\sum_{a=1}^{q-1}D_{0,0}(\tfrac12,\tfrac aq)^k&\sim q^{\frac k2-1}2^{1-\frac k2}\frac{\zeta(\frac k2)^2}{\zeta(k)}\pr{\cos\pr{\tfrac {k \pi }4}( \log q)^k-\frac \pi2 \sin\pr{\tfrac {k \pi }4}( \log q)^{k-1}}
}
as $q\rightarrow \infty$.
\end{theo}

Theorem~\ref{caee} should be compared with Theorem~1.2 of~\cite{FGKM} which gives the asymptotic for the moments of 
\es{\label{adae}
\sum_{\substack{n\geq1,\\n\equiv a\mod q}}d(n)f(n/X)-\mathcal M_1(X,q),
} where $\mathcal M_1(X,q)$ is a certain main term, $f\in \mathcal C_0^{\infty}(\R^{+})$, and $X=q^2/\Phi(q)$ with $\Phi(x)\rightarrow\infty$ and $\Phi(x)\ll x^\eps$ as $x\rightarrow\infty$ (see also~\cite{LY} the case of sharp cut-offs). For comparison, Theorem~\ref{caee} should be thought of roughly computing the moments of 
\es{\label{adae2}
\sum_{\substack{n\geq1}}d(n)\e{n a/q}f(n/X)-\mathcal M_2(X,q)
}with $X\approx q$ and another main term $\mathcal M_2(X,q)$. Despite some superficial similarities, the two problems are very different in nature (and have very different proofs). For example, a first important difference is that the range of the summation in~\eqref{adae} is twice as long as that of~\eqref{adae2} in the logarithmic scale.

\subsubsection{Moments of certain functions defined in terms of continued fractions}

Finally, we discuss the relation with continued fractions. In~\cite{Bet16} (see also~\cite{You11a}), it was observed that $M(a,q)$, and more generally, $D_{\cos}$ and $D_{\sin}$, can be written in terms of the continued fraction expansion of $a/q$. Indeed, if  $a,q\in\Z_{>0}$ and $[b_0;b_1,\cdots,b_{\kappa},1]$ is the continued fraction expansion of $\frac aq$, then for $q$ prime one has
\es{\label{prr}
M(a,q)=\sum_{\substack{j=1,\\j\tn{ odd}}}^{\kappa}b_j^{\frac12}(\log \tfrac {b_j }{8\pi}+\gamma)-\frac\pi 2\sum_{\substack{j=1,\\j\tn{ even}}}^{\kappa}b_j^{\frac12}+O(\log q).\\
}
It is therefore not surprising that Theorem~\ref{1tp} have an incarnation also in terms of moments for functions of the rationals defined as
\est{
 f_{r,\pm}(a/q):=\sum_{j=1}^\kappa (\pm1)^jb_j^{\frac r2},
}
where $r\in\Z_{\geq1}$.

\begin{theo}\label{tinc}
Le $q$ be prime and let $k,r\in\Z_{\geq1}$ with $3\leq kr =o(\frac{\log q}{\log\log q})$. Then
\est{
\sum_{a=1}^qf_{r,\pm}(a/q)^{k}\sim 2 \frac{\zeta(\frac {kr}2)^2}{\zeta(kr)}  q^{\frac {kr}2}
}
as $q\rightarrow\infty$.
\end{theo}

Starting with the work of Heilbronn~\cite{Hei}, who considered the average value of $f_{0,+}$, there have been a very large number of papers computing the mean values of functions defined in terms of the continued fraction expansion. In particular, we cite the works~\cite{Por, Ton} on $f_{0,+}$ and~\cite{YK} where the asymptotic for the first moment of $f_{2,+}$ was given. However, to the knowledge of the author, Theorem~\ref{tinc} is the first result giving asymptotic formulae for $k$-th moments with $k\geq 3$ without exploiting an extra average over $q$ (as in~\cite{Hen,BV}). For $k=2$ the only cases previously known where obtained by Bykovski\u i~\cite{Byk} (considering the second moment of $f_{0,+}$) and by the author~\cite{Bet16} (considering the second moment of a variation of $f_{2,+}$). By combining the techniques employed in~\cite{Bet16} and in this paper it seems possible to extend Theorem~\ref{tinc} to more general functions of similar shape.


%

\subsection{Brief outline of the proof}

The approximate functional equation allows one to express $M_k(q)$ roughly in the form 
\es{\label{tcp}
\sum_{\substack{ \pm n_1\pm_2 \cdots \pm n_k\equiv 0\mod q,\\ n_1\cdots n_k\ll q^{k}}}\frac{d(n_1)\cdots d(n_k)}{n_1^\frac12\cdots n_k^\frac12},
}
so that the problem of estimating $M_k(q)$ reduces to that of computing the asymptotic for this quadratic divisor problem. The diagonal terms (i.e. the terms with $\pm n_1\pm_2 \cdots \pm n_k=0$) are a bit easier to study and give a main term; the main difficulties then lie in obtaining an asymptotic for the off-diagonal terms and in assembling the various main terms. In his proof of~\eqref{aafy}, which corresponds to~\eqref{tcp} with $k=2$, Young used a combination of several techniques each effective for some range of the variables $n_1,n_2$. In particular, when $n_1\approx n_2$ (in the logarithmic scale) he followed an approach \`a la Motohashi~\cite{Mot97} using Kuznetsov formula, whereas when one variable is much larger than the other one, he used (new) estimates for the average value of the divisor function in arithmetic progressions.

Our approach is similar to that of Young, however there are several substantial differences which we will now discuss in some detail. First, the larger number of variables gives us the advantage of having to deal with more ``flexible'' sums enlarging the ranges where the various estimates are effective. For this reason, we can afford to use slightly weaker bounds employing the spectral theory only indirectly, through the bounds of Deshouilliers  and Iwaniec~\cite{DI} (together with Kim and Sarnak's bound for the exceptional eigenvalues~\cite{Kim}). It seems likely that one could use spectral methods in a more direct and efficient way, however the generalization of the methods in~\cite{You11b} (or~\cite{BFKMM}) to the $k\geq3$ case is not straightforward and so we choose a simpler route as this is still sufficient for our purposes. 

The larger number of variables has also a cost. Indeed, it introduces several new complications in the extraction and in the combination of the main terms, a process that requires a rather careful analysis and constitutes the central part of this paper. One of the causes of the complicated shape of the main terms (cf.~\eqref{gcfa}-\eqref{mtmt}) is that with more than two variables the dichotomy ``either one variable is much bigger than the other or the variables have the same size'' doesn't hold for $k>2$ and one has to (implicitly) deal also with cases such as $n_1\approx \cdots\approx n_{k-1}\approx q^{1+1/k}$ and $n_k\approx 1$.

Another difference with Young's work arises when studying the diagonal terms. If $k=2$, then one can handle these terms easily thanks to Ramanujan's formula $\sum_{n\geq1}d(n)^2/n^s=\zeta(s)^4/\zeta(2s)$. If $k\geq3$, we don't have such a nice exact formula, and we are left with the problem of showing that the series
\est{
\sum_{\pm n_1\pm\cdots\pm n_k=0}\frac{d(n_1)\cdots d(n_k)}{(n_1\cdots n_k)^s}
}
can be meromorphically continued past the line $\Re(s)=1-1/k$ which is the boundary of convergence. We shall leave this problem to a different paper,~\cite{Bet}, where with similar (but a bit simpler) techniques we prove that this series admits meromorphic continuation to the region $\Re(s)>1-2/(k+1)$.

Last, we mention a more technical problem. One of the steps in Young's proof requires separating $n_1,n_2$ in expressions of the form $(n_1\pm n_2)^{-z}$ when $\Re(z)\approx 0$. This can be easily obtained by using some classical Mellin formulae; however, whereas the Mellin integral corresponding to $(1+x)^{-z}$ converges absolutely, the Mellin integral corresponding to $(1-x)^{-z}$ converges only conditionally so that the terms containing $(n_1- n_2)^{-z}$ demand some caution. In our case this problem becomes rather more subtle as we need to apply these formulae iteratively in order to handle expressions such as $(n_1\pm \cdots\pm n_k)^{-s}$. We overcome this difficulties by using a modification of the resulting ``iterated'' Mellin formula allowing us to write such expressions in terms of absolutely convergent integrals (see Section~\ref{amell} for the details). 

\subsection{The structure of the paper}\label{aedxw}
The paper is organized as follow. In Section~\ref{dddc} we state Theorem~\ref{mtws}, a more general version of Theorem~\ref{caee} providing the asymptotic for the mixed moments of $D_{\cos}$ and $D_{\sin}$ (as well as allowing for some small shifts). We then use this result to deduce Theorem~\ref{1tp},~\ref{caee} and~\ref{tinc}. In Section~\ref{ek} we give some lemmas on the Estermann function which we shall need later on. It is in this lemmas that the spectral theory comes (indirectly) into play. The proof of Theorem~\ref{mtws} is carried out in Sections~\ref{bmp}-\ref{ttftd}, after introducing some notation in Section~\ref{assumptions}, and constitutes the main body of the paper. Finally, in Section~\ref{amell} we will prove the Mellin formula mentioned at the end of the previous section as well as some technical Lemmas needed in order to use this formula effectively.

\subsubsection*{Acknowledgments} 

The author would like to thank Sary Drappeau, Dimitris Koukoulopoulos and Maksym Radziwi\l\l\, for helpful discussions.

This paper was started while the author was a postdoc at the Centre de Recherche Math\'ematiques (CRM) in Montr\'eal, and was completed during another visit at the same institution. The author is grateful to the CRM for the hospitality and for providing a stimulating working environment.

The work of the author is partially supported by FRA 2015 ``Teoria dei Numeri" and by PRIN ``Number Theory and Arithmetic Geometry". 

\section{Mixed moments of $D_{\cos}$ and $D_{\sin}$ and the deduction of the main theorems}\label{dddc}

Let $k\geq1$, $q$ be a prime and let $\alpha_1,\dots,\alpha_k$, $\beta_1,\dots,\beta_k\in\C$. Then, for any subset $\Upsilon\subseteq\{1,\dots,k\}$ let $M_{\Upsilon,k}$ be the mixed shifted moment
\est{
M_{\Upsilon,k}&:=\frac1{\varphi(q)}\sum_{a=1}^{q-1}\prod_{i=1}^k D_{i;\alpha_i,\beta_i}\pr{\tfrac12,\tfrac aq},
}
where $D_{i;\alpha_i,\beta_i}:=D_{\sin;\alpha_i,\beta_i}$ if $i\in\Upsilon$ and $D_i:=D_{\cos;\alpha_i,\beta_i}$ otherwise. Also, let
\as{
\Gamma_i(s)&:=\begin{cases}
\Gamma(\frac12+s) &\tn{if $i\in\Upsilon$,}\\
\Gamma(s)&\tn{otherwise}.
\end{cases}\label{dfgi}
}
Since $D_{\sin;\alpha_i,\beta_i}(s,-\frac aq)=-D_{\sin;\alpha_i,\beta_i}(s,\frac aq)$, then $M_{\Upsilon,k}$ is identically zero if $|\Upsilon|$ is odd. If $|\Upsilon|$ is even the asymptotic for $M_{\Upsilon,k}$ is given by the following theorem, provided that $k\geq3$ (the corresponding theorem for $k=2$ is essentially implicit in~\cite{You11b}, whereas the case $k=1$ is trivial).
\begin{theo}\label{mtws}
Let $\Upsilon\subseteq\{1,\dots,k\}$ with $|\Upsilon|$ even.
Let $k\geq 3$ and let $q$ be a prime. Let $\balpha=(\alpha_1,\dots,\alpha_k)$, $\bbeta:=(\beta_1,\dots,\beta_k)\in\C^k$ with $|\alpha_i|,|\beta_i|\ll  \frac1{\log q}$, $|\alpha_i|,|\beta_i|\leq 1/10$ and $\alpha_i\neq\beta_i$ for all $i=1,\dots,k$. Then, there exists an absolute constant $A>0$ such that for any $\eps>0$ we have
\es{\label{asvfd}
M_{\Upsilon,k}&= \sum_{\{\alpha_i',\beta_i'\}=\{\alpha_i,\beta_i\}}\mathscr M_{\balpha',\bbeta'}+O_\eps\pr{k^{Ak}q^{\tfrac k2-1-\delta_k+\eps}},\\
}
where $\delta_k:=\frac{k-2-3 \vartheta}{2 k+5}$,
\es{\label{gabcd}
\mathscr M_{\balpha,\bbeta}:=\frac{q^{\frac k2-1}}{2^{k-1}}\frac{ \zeta\Big(\frac k2-\sum_{i=1}^k\alpha_i\Big)\zeta\Big(\frac k2+\sum_{i=1}^k\beta_i\Big)}{\zeta\Big(k-\sum_{i=1}^k(\alpha_i-\beta_i)\Big)}\prod_{i=1}^k\frac{\Gamma_i(\frac14-\frac{\alpha_i}2)}{\Gamma_i(\frac14+\frac{\alpha_i}2)}\pr{\frac q{\pi}}^{-\alpha_i} \zeta(1-\alpha_i+\beta_i)
}
and where the implicit constant in the error term depends on $\varepsilon$ only.
\end{theo}
\begin{remark}
If $\alpha_i=\beta_i$ for some $i=1,\dots k$, then $\mathcal M_{\balpha,\bbeta}$ has to be interpreted as the limit for $\alpha_i\rightarrow\beta_i$ (cf.~\eqref{slm} below).
\end{remark}

As mentioned in Section~\ref{aedxw}, we will prove Theorem~\ref{mtws} in Sections~\ref{bmp}-\ref{ttftd}. We will now deduce Theorem~\ref{1tp},~\ref{caee},~\ref{tinc} from Theorem~\ref{mtws}.

\subsection{Proof of Theorem~\ref{1tp},~\ref{caee} and~\ref{tinc} and of Corollary~\ref{cqa}}
We start by observing that if $|\Upsilon|$ is even then from Theorem~\ref{mtws} one has
\es{\label{slm}
\sum_{a=1}^{q-1}\prod_{i=1}^k D_{i;0,0}\pr{\tfrac12,\tfrac aq}= \frac{q^{\frac k2}}{2^{k-1}}\sum_{n=1}^{\infty}\frac{2^{\nu(n)}}{n^{\frac k2}}\prod_{i=1}^k (\log\tfrac q{8n\pi }+\gamma-a_i\pi)
+O_\eps\pr{k^{Ak}q^{\tfrac k2-\delta_k+\eps}},
}
where $a_i=-\frac12$ if $i\in \Upsilon$ and $a_i=\frac12$ otherwise. Indeed, if $\balpha$ and $\bbeta$ satisfy the hypothesis of Theorem~\ref{mtws}, then by contour integration the main term on the right hand side of~\eqref{asvfd} can be rewritten as
\es{\label{mtma}
\sum_{\{\alpha_i',\beta_i'\}=\{\alpha_i,\beta_i\}}\hspace{-0.9em}\mathscr M_{\balpha^*,\bbeta^*}&=\frac{q^{\frac k2-1}}{2^{k-1}}\frac1{(2\pi i)^k}\oint_{|s_1|=\frac14}\hspace{-0.6cm}\cdots\oint_{|s_k|=\frac14}  \hspace{-0.6cm}\frac{ \zeta\Big(\frac k2+\sum_{i=1}^k(s_i-\alpha_i-\beta_i)\Big)\zeta\Big(\frac k2+\sum_{i=1}^ks_i\Big)}{\zeta\Big(k+\sum_{i=1}^k(2s_i-\alpha_i-\beta_i)\Big)}\\
&\quad\times\prod_{i=1}^k\frac{\Gamma_i(\frac14+\frac{s-\alpha_i-\beta_i}2)}{\Gamma_i(\frac14-\frac{s-\alpha_i-\beta_i}2)}\pr{\frac{q}\pi}^{s_i-\alpha_i-\beta_i} \zeta(1+s_i-\alpha_i)\zeta(1+s_i-\beta_i)ds_i,
}
where the circles are integrated counter-clockwise . Thus, taking the limit for $\balpha,\bbeta\rightarrow0$ and expanding $\zeta(s)^2/\zeta(2s)$ as a Dirichlet series (see~\cite{Tit}, (1.2.8)), we obtain
\est{
\sum_{\{\alpha_i',\beta_i'\}=\{\alpha_i,\beta_i\}}\hspace{-0.9em}\mathscr M_{\balpha',\bbeta'}&= \frac{q^{\frac k2-1}}{2^{k-1}}\sum_{n=1}^{\infty}\frac{2^{\nu(n)}}{n^{\frac k2}}\prod_{i=1}^k \frac1{2\pi i}\oint_{|s_i|=\frac14}\frac{\Gamma_i(\frac14+\frac{s_i}2)}{\Gamma_i(\frac14-\frac{s_i}2)}\pr{\frac{q}{n\pi}}^{s_i} \zeta(1+s_i)^2ds_i\\
&=\frac{q^{\frac k2-1}}{2^{k-1}} \sum_{n=1}^{\infty}\frac{2^{\nu(n)}}{n^{\frac k2}}\prod_{i=1}^k (\log\tfrac q{8n\pi }+\gamma-a_i\pi),
}
by the residue theorem. Equation~\eqref{slm} then follows. 

To prove Theorem~\ref{1tp} we observe that by~\eqref{slm} we have (remember that if $|\Upsilon|$ is odd then $M_{\Upsilon,k}=0$)
\est{
&\sum_{a=1}^{q-1}(D_{\cos;0,0}(\tfrac12,\tfrac aq)+D_{\sin;0,0}(\tfrac12,\tfrac aq))^k= \sum_{r=0}^k\frac{k!}{r!(k-r)!}\frac1{\varphi(q)}\sum_{a=1}^{q-1}D_{\cos;0,0}(\tfrac12,\tfrac aq)^{k-r}D_{\sin;0,0}(\tfrac12,\tfrac aq)^{r}\\
&\hspace{6em}=\frac{q^{\frac k2}}{2^{k-1}}\sum_{n=1}^{\infty}\frac{2^{\nu(n)}}{n^{\frac k2}}\sum_{\substack{r=0,\\ r\tn{ even}}}\frac{k!}{r!(k-r)!} (\log\tfrac q{8n\pi }+\gamma-\tfrac \pi2)^{k-r} (\log\tfrac q{8n\pi }+\gamma +\tfrac \pi2)^{r}+\mathcal E_2\\
&\hspace{6em}=\frac{q^{\frac k2}}{2^{k}}\sum_{n=1}^{\infty}\frac{2^{\nu(n)}}{n^{\frac k2}} ((2\log\tfrac q{8n\pi }+2\gamma)^{k} +(-\pi)^{k})+\mathcal E_2\\
}
for some $\mathcal E_2\ll_\epsilon k^{Ak} q^{\tfrac k2-\delta_k+\eps}$. Thus, using~\eqref{axe} one obtains Theorem~\ref{1tp}.
One easily verifies that as $q\to\infty$
\est{
\sum_{n=1}^{\infty}\frac{2^{\nu(n)}}{n^{\frac k2}} \pr{(\log\tfrac q{8n\pi }+\gamma)^{k}+(-\pi)^{k}}\sim \frac{\zeta( \frac k2)^2}{\zeta(k)}(\log\tfrac q{8\pi }+\gamma)^{k}
}
uniformly in $k\geq3$. 
If $k=o(\frac {\log q}{\log \log q})$ the error term $\mathcal E_2$ is smaller than the above main term and so Corollary~\ref{cqa} follows on this range. 

Now assume $k\gg \frac {\log q}{\log \log q}$. First, we observe that by~\eqref{prr} for $a\neq1$ we have
\es{\label{gbf}
|M(a,q)|\leq (q/\eta)^{\frac12}\log q
}
for any fixed $1<\eta<2$ and $q$ sufficiently large. Indeed, this is obvious if $a=-1$ whereas if $a\neq\pm1$ then $\max_j b_j\leq  (q-1)/2$ and so the above bound follows since  $b_1\cdots b_\kappa\leq q$. 
Furthermore, from the second moment estimate $\sum_{a=1}^q|M(a,q)|^2\ll (\log q)^4$ it follows that for every $C>0$ there are at most $O(q(\log q)^4 /C^2)$ values of $a$ in  $1<a\leq q$ such that $|M(a,q)|\geq C$. Thus, by~\eqref{gbf} we have
\est{
\sum_{a=2}^q M(a,q)^k&\leq \sum_{\substack{2\leq a\leq q,\\ |M(a,q)|<C}}^q M(a,q)^k+\sum_{\substack{2\leq a\leq q,\\ |M(a,q)|\geq C}}^q M(a,q)^k\\
&\ll C^k q+\frac{q^{\frac k2+1} (\log q)^{k+4} }{\eta^{\frac k2}C^2}\ll \frac {q^{\frac k2+\frac 2{k+2}}}{\eta^{\frac k2}}(\log q)^{k+2}
}
for $C= \eta^{-\frac12}q^{\frac12-\frac {1}{k+2}}\log q$.
Note that if $k\gg \frac {\log q}{\log \log q}$, then the error term is $\ll q^{\frac k2}\eta^{-\frac k4}(\log q)^{k}=o(q^{\frac k2}(\log (q/8\pi)+\gamma)^{k})$ as $q\rightarrow \infty$, uniformly in $k$. Finally, we have (cf.~\cite{H-B81})
\est{
M(1,q)=q^{\frac 12}(\log (q/8\pi) +\gamma)+2\zeta(\tfrac12)^2+O(q^{-\frac12})
}
so that
\est{
M(1,q)^k= q^{\frac k2}(\log (q/8\pi) +\gamma)^k\exp\Big(2\zeta(\tfrac12)^2\frac{k}{q^{1/2}\log q}(1+O(1/\log q))\Big)\\
}
for $q$ large enough. Thus, if $\frac {\log q}{\log \log q}\ll k=o(q^{1/2}\log q)$ we have
\est{
M_k(q)&=q^{-k/2}\sum_{a=1}^q M(a,q)^k\sim (\log (q/8\pi) +\gamma)^k\sim\frac{\zeta(\frac k2)^2}{\zeta(k)}(\log (q/8\pi) +\gamma)^k
}
as $q\rightarrow\infty$, whereas this asymptotic is false if $k\gg q^{1/2}\log q$. This concludes the proof of Corollary~\ref{cqa}.

The proof of Theorem~\ref{caee} is analogous to those of Theorem~\ref{1tp} and Corollary~\ref{cqa}, with the difference that in this case we use~\eqref{dfnes} rather than~\eqref{axe}. Indeed for some $\mathcal E_1\ll_\epsilon k^{Ak} q^{\tfrac k2-\delta_k+\eps}$ we have
\est{
&\sum_{a=1}^{q-1}D_{0,0}(\tfrac12,\tfrac aq)^k=\sum_{r=0}^k\frac{k!}{r!(k-r)!}\sum_{a=1}^{q-1}D_{\cos;0,0}(\tfrac12,\tfrac aq)^{k-r}i^rD_{\sin;0,0}(\tfrac12,\tfrac aq)^{r}\\
&=\frac{q^{\frac k2}}{2^{k-1}}\sum_{n=1}^{\infty}\frac{2^{\nu(n)}}{n^{\frac k2}}  \sum_{\substack{r=0,\\ r\tn{ even}}}^k\frac{k!}{r!(k-r)!}(\log\tfrac q{8n\pi }+\gamma-\tfrac \pi2)^{k-r}i^r(\log\tfrac q{8n\pi }+\gamma+\tfrac\pi2)^r+\mathcal E_1\\
&=\frac{q^{\frac k2}}{2^{k}} \sum_{n=1}^{\infty}\frac{2^{\nu(n)}}{n^{\frac k2}}  \pr{((1+i)(\log\tfrac q{8n\pi }+\gamma)-(1-i)\tfrac \pi2)^{k}+((1-i)(\log\tfrac q{8n\pi }+\gamma)-(1+i)\tfrac \pi2)^{k}}+\mathcal E_1\\
&=(q/2)^{\frac k2}\,2\,\Re \pr{\sum_{n=1}^{\infty}\frac{2^{\nu(n)}}{n^{\frac k2}}  \pr{e^{\frac {\pi i}4}(\log\tfrac q{8n\pi }+\gamma)-e^{-\frac {\pi i}4}\tfrac \pi2}^{k}}+\mathcal E_1,\\
&\sim q^{\frac k2}2^{1-\frac k2}\frac{\zeta(k/2)^2}{\zeta(k)}\Re \pr{\pr{e^{\frac {\pi i}4}(\log\tfrac q{8\pi }+\gamma)-e^{-\frac {\pi i}4}\tfrac \pi2}^{k}}
}
as $q\rightarrow\infty$ with $3\leq k=o(\frac{\log q}{\log \log q})$. One then obtains Theorem~\ref{caee} on the range $3\leq k=o(q^\frac 12\log q)$ by proceeding as in the proof of Corollary~\ref{cqa}.

\subsection{Proof of Theorem~\ref{tinc}}
We compute the moments of $f_{r,+}$ only, the case of $f_{r,-}$ being analogous (using~\eqref{fffbs} instead of~\eqref{fffb}).

We start by noticing that Corollary~11 of~\cite{Bet16} gives
\begin{align}
D_{\cos;0,0}(\tfrac12,\tfrac aq)&=\frac12\sum_{j=1}^\kappa b_j^\frac12(\log \tfrac {b_j}{8\pi}+\gamma-\tfrac\pi2)+O(\log q),\label{fffb}\\
D_{\sin;0,0}(\tfrac12,\tfrac aq)&=\frac12\sum_{j=1}^\kappa (-1)^jb_j^\frac12(\log \tfrac{b_j}{8\pi}+\gamma+\tfrac\pi2)+O(\log q),\label{fffbs}
\end{align}
where $[0;b_1,\dots b_\kappa,1]$ is the continued fraction expansion of $a/q$. Moreover, since $b_1\cdots b_j\asymp q$, then if one among $b_1,\dots,b_j$, say $b_{j^*}$, satisfies $b_{j^*}> q/(\log q)^{100} $ (and thus in particular $\log b_{j*}=\log q+O(\log \log q)$), then $b_j\ll (\log q)^{100}$ for $j\neq j^*$. In particular, if $\max_ jb_j>q/(\log q)^{100}$ and $1\leq r=o( \log q/\log\log q)$, then
\begin{align}
f_{r,+}(\tfrac aq)&=\sum_{j=1}^\kappa b_j^\frac{r}2= \max_{j=1,\dots \kappa} b_j^{\frac r2}+O((\log q)^{50r+1})=\frac1{(\log q)^r} \big(\max_{j=1,\dots \kappa} b_j^{\frac 12}\log q\big)^r+O((\log q)^{50r+1})\notag\\
&=\frac1{(\log q)^r} \Big( \Big(\max_{j=1,\dots \kappa} b_j^{\frac 12}(\log \tfrac{b_j}{8\pi}+\gamma-\tfrac\pi2)\Big)(1+O(\log \log q/\log q))\Big)^r+O((\log q)^{50r+1})\notag\\
&=\frac{2^r}{(\log q)^r} D_{\cos}(\tfrac12,\tfrac aq)^r(1+O(r\log \log q/\log q)).\label{cra}
\end{align}
Moreover, from~\eqref{fffb} it follows easily that
\est{
\sum_{j=1}^\kappa b_j^\frac12\leq D_{\cos;0,0}(\tfrac12,\tfrac aq)+B\log q\\
}
for all $a/q$ and some $B>0$. In particular, if $\max_ jb_j\leq q/(\log q)^{100}$ and $q$ large enough, then
\es{\label{arc}
f_{r,+}(\tfrac aq)^k&\leq
 \frac{q^{\frac k2(r-1)}}{(\log q)^{50k(r-1)}}\bigg(\sum_{j=1}^\kappa b_j^\frac12\bigg)^k\leq \frac{q^{\frac k2(r-1)}}{(\log q)^{50k(r-1)}} \big(D_{\cos;0,0}(\tfrac12,\tfrac aq)+B\log q\big)^k\\
 &\ll \frac{q^{\frac {kr}2-1}}{(\log q)^{50k(r-1)+47(k-2)}}  \big(D_{\cos;0,0}(\tfrac12,\tfrac aq)+B\log q\big)^2 
}
for $k\geq 2$, since $\max_ jb_j\leq q/(\log q)^{100}$ implies $|D_{\cos}(\tfrac12,\tfrac aq)|+B\log q\leq q^{\frac12}/(\log q)^{48}$ for $q$ large enough.
Now, we have
\es{\label{rfcrv}
\sum_{a=1}^qf_{r,+}(\tfrac aq)^k=\sum_{\substack{1\leq a<q,\\ \max_ jb_j>q/(\log q)^{100}}}f_{r,+}(\tfrac aq)^k+\sum_{\substack{1\leq a<q,\\ \max_ jb_j\leq q/(\log q)^{100}}}f_{r,+}(\tfrac aq)^k.
}
By~\eqref{arc} the second summand is bounded by
\est{
\sum_{\substack{1\leq a<q,\\ \max_ jb_j\leq q/(\log q)^{100}}}f_{r,+}(a/q)^k&
\ll \frac{q^{\frac {kr}2-1}}{(\log q)^{50k(r-1)+48(k-2)}}  \sum_{1\leq a<q}\big(D_{\cos}(\tfrac12,\tfrac aq)+B\log q\big)^2  \\[-0.5em]
&\ll \frac{q^{\frac {kr}2}}{(\log q)^{50k(r-1)+48(k-2)-4}} \ll \frac{q^{kr/2}}{\log q}
}
for $kr\geq3$ (if $k=1$ one needs to modify slightly the argument, but the final bound still holds).  By~\eqref{cra} the first summand of~\eqref{rfcrv} can be written as
\est{
&\sum_{\substack{1\leq a<q,\\ \max_ jb_j>q/(\log q)^{100}}}^q \frac{ (2D_{\cos;0,0}(\tfrac12,\tfrac aq))^{kr}}{(\log q)^{kr}}(1+O(kr\log \log q/\log q))\\
&\qquad=\sum_{\substack{1\leq a<q} }\frac{ (2D_{\cos;0,0}(\tfrac12,\tfrac aq))^{kr}}{(\log q)^{kr}}(1+O(kr\log \log q/\log q))+O({q^{kr/2}}/{\log q})\\
&\qquad= 2 \frac{\zeta(kr)^2}{\zeta(\frac{kr}2)}  q^{\frac {kr}2}(1+O(kr\log \log q/\log q))
}
by~\eqref{slm} for $3\leq rk=o(\frac {\log q}{\log\log q})$ and where one can complete the sum by proceeding as in the previous computation. Theorem~\ref{tinc} then follows.

\section{The Estermann function and bounds for sums of Kloosterman sums}\label{ek}
In this Section we give some results for the Estermann function and for the periodic zeta-function which will be needed in the proof of Theorem~\ref{mtws}. In particular, in Section~\ref{fes}  we give the functional equation for both these functions, whereas in Section~\ref{afes} we give a version of the approximate functional equation for the Estermann function. Finally, in Section~\ref{eef} we give some estimates for products of the Estermann function and the periodic zeta-function, using the bounds of~\cite{DI} for sums of Kloosterman sums.

\subsection{The functional equations}\label{fes}
We start by giving the functional equation for the Estermann function.
\begin{lemma}\label{properties of Estermann}
For $(a,q)=1$, $q>0$ and $\alpha \in\C$, $D_{\alpha,\beta}(s,\frac aq)-q^{1-\alpha-\beta-2s}\zeta(s+\alpha)\zeta(s+\beta)$
can be extended to an entire function of $s$. Moreover, $D_{\alpha,\beta}(s,\frac aq)$ satisfies the functional equation
\es{\label{feest}
D_{\alpha,\beta}\pr{s,\tfrac aq}&=-\tfrac{2}{q}\pr{\tfrac{q}{2\pi}}^{2-2s-\alpha-\beta}\Gamma\pr{1-s-\alpha}\Gamma\pr{1-s-\beta}\times\\
&\hspace{0.8em}\times\Big(\cos\pr{\tfrac\pi2\pr{2s+\alpha+\beta}}D_{-\alpha,-\beta}\pr{1-s,-\tfrac {\overline a}q}-\cos\tfrac{\pi (\alpha-\beta)}2\, D_{\alpha,\beta}\pr{1-s,\tfrac {\overline a}q}\Big),
}
where, here and in the following, $\overline a$ denotes the multiplicative inverse of $a$ modulo the denominator $q$.
\end{lemma}

\begin{proof}
The lemma follows easily from the analytic continuation and the functional equation for the Hurwitz zeta function $\zeta(s,x)$ (see~\cite{Apo}, Theorem 12.8) and from the decomposition
\astr{
D_{\alpha,\beta}\pr{s,\tfrac hk}=&q^{-\alpha-\beta-2s}\sum_{m,n=1}^{q}\e{\tfrac{mn a}q}\zeta\pr{s+\alpha,\tfrac mq}\zeta\pr{s+\beta,\tfrac nq}.\qedhere
}
\end{proof}

\begin{corol}
 Let 
\est{
\Lambda_{c;\alpha,\beta}\pr{s,\tfrac aq}&:=\Gamma\pr{\tfrac {s+\alpha}2}\Gamma\pr{\tfrac {s+\beta }2} \pr{\frac{q}{\pi}}^{s+\frac {\alpha+\beta} 2} D_{c;\alpha,\beta}\pr{s,\tfrac aq},\\
\Lambda_{s;\alpha,\beta}\pr{s,\tfrac aq}&:=\Gamma\pr{\tfrac {1+s+\alpha}2}\Gamma\pr{\tfrac {1+s+\beta }2} \pr{\frac{q}{\pi}}^{s+\frac {\alpha+\beta} 2} D_{s;\alpha,\beta}\pr{s,\tfrac aq}.
}
Then, we have the functional equations
\es{\label{afce}
\Lambda_{c;\alpha,\beta}\pr{s,\tfrac aq}&=\Lambda_{c;-\alpha,-\beta}\pr{1-s,\tfrac {\overline a }q},\qquad
\Lambda_{s;\alpha,\beta}\pr{s,\tfrac aq}=\Lambda_{s;-\alpha,-\beta}\pr{1-s ,\tfrac {\overline a }q}.\\
}
\end{corol}
\begin{proof}
These functional equations follow  from~\eqref{feest}, using the reflection and the duplication formulas for the $\Gamma$-function.
\end{proof}

We also need the basic properties of the periodic zeta-function which, for $x\in \R$ and $\Re(s)>1$, is defined as
\es{\label{dff}
F(s,x):=\sum_{n=1}^\infty\frac{\e{nx}}{n^s}.
}
Notice that if $x=1$, then $F(s,x)=\zeta(s)$.
\begin{lemma}
Let $h,l\in\Z$ with $(h,\ell)=1$ and $\ell>0$, then $F(s,h/\ell)$ extends to an entire function of $s$ with the exception of a simple pole at $s=1$ if $\ell=1$. Moreover, $F(s,x)$ satisfies the functional equation
\es{\label{fef}
F(1-s,h/\ell)=\ell^{s-1}\sum_{b=1}^\ell\e{{hb}/\ell}\frac{\Gamma(s)}{(2\pi)^s}\pr{e^{-\frac{\pi is}{2}}F(s,b/\ell)+e^{\frac{\pi is}{2}}F(s,-b/\ell)}.
}
Finally, for $\ell\nmid h$ we have
\es{\label{fef12}
F(0,h/\ell)=-\tfrac12+\tfrac i2\cot(\pi h/\ell).
}
\end{lemma}
\begin{proof}
If $\Re(s)<0$, then splitting the series defining $F(s,h/\ell)$ according to the congruence class of $n$ modulo $\ell$ we obtain
\est{
F(1-s,h/\ell)=\ell^{s-1}\sum_{b=1}^\ell\e{{hb}/\ell}\zeta(1-s,b/\ell),
}
where $\zeta(s,x)$ the Hurwitz zeta-function. Thus the analytic continuation of $F(s,h/\ell)$ follows from that of $\zeta(s,x)$ since $\zeta(s,x)$ is holomorphic on $\C$ with the exception of a simple pole of residue $1$ at $s=1$. The functional equation~\eqref{fef} 
follows immediately from the identity 
\es{\label{wew}
\zeta(1-s,x)=(2\pi)^{-s}\Gamma(s)(e^{-\frac{\pi is}2}F(s,x)+e^{\frac{\pi is}2}F(s,-x))
} 
(see~\cite{Apo}, Theorem 12.6), whereas~\eqref{fef12} follows from~\eqref{wew}, the Laurent expansion for $\zeta(s,x)$ (cf.
~\cite{WW}, p. 271)
\est{
\zeta(s,x)=\frac1{s-1}-\psi(x)+O(|s-1|),
}
where $\psi(x)$ is the digamma function, and the identity $\psi(1 - x) - \psi(x) = \pi\,\!\cot{ \left ( \pi x \right ) }$.
\end{proof}

\subsection{The approximate functional equation}\label{afes}

Next, we give an approximate functional equation allowing us to express a product of $k$ Estermann functions as a sum of total length about $q^{\frac k2}$. In this Lemma and in the rest of the paper we will often omit to indicate the dependencies on $\Upsilon$. 

\begin{lemma}\label{adwede}
Let $k\geq1$ and $\Upsilon\subseteq\{1,\dots,k\}$. Let $G_{\balpha,\bbeta}(s)$ be an entire function satisfying $G_{\balpha,\bbeta}(-s)=G_{-\balpha,-\bbeta}(s)$,  $G_{\balpha,\bbeta}(0)=1$ and $G_{\balpha,\bbeta}(\frac12-\alpha_i)=G_{\balpha,\bbeta}(\frac12-\beta_i)=0$ for $i=1,\dots,k$ and decaying faster than any power of $s$ on vertical strips. Let
\begin{align}
g_{\balpha,\bbeta}(s)&:=\pi^{-ks}\prod_{j=1}^{k}\frac{\Gamma_i\pr{\tfrac {\frac12+s+\alpha_i}2}\Gamma_i\pr{\tfrac {\frac12+s+\beta_i}2}}{\Gamma_i\pr{\tfrac {\frac12+\alpha_i}2}\Gamma_i\pr{\tfrac {\frac12+\beta_i}2}},\label{dfngs}\\
X_{\balpha,\bbeta}&:=\prod_{j=1}^k\frac{\Gamma_i\pr{\tfrac {\frac12-\alpha_i}2}\Gamma_i\pr{\tfrac {\frac12-\beta_i}2} }{\Gamma_i\pr{\tfrac {\frac12+\alpha_i}2}\Gamma_i\pr{\tfrac {\frac12+\beta_i}2} }\pr{\frac{q}{\pi}}^{- \alpha_i-\beta_i}\notag
\end{align}
and for any $c_s>0$ let
\est{
V_{\balpha,\bbeta}(x):=\frac1{2\pi i}\int_{(c_s)}G_{\balpha,\bbeta}(s)g_{\balpha,\bbeta}(s)x^{-s}\frac{ds}s,
}
where, as usual, $\int_{(c)}\cdot\ ds$ indicates that the integral is taken along the vertical line from $c-i\infty$ to $c+i\infty$.
Then for $a,q\in\Z,$ with $q>1$ and $(a,q)=1$ we have
\es{\label{feqw}
\prod_{i=1}^k D_{i;\alpha_i,\beta_i}(\tfrac12,\tfrac aq)=S_{\balpha,\bbeta}\pr{ a,q}+X_{\balpha,\bbeta}S_{-\balpha,-\bbeta}\pr{\overline a,q},
}
where $\overline a$  is the inverse of $a$ modulo $q$ and
\est{
S_{\balpha,\bbeta}\pr{ a,q}&:=\frac{i^{-|\Upsilon|}}{2^{k}}\sum_{\epsilon=(\pm_11,\dots,\pm_k1)\in\{\pm1\}^{k}}\sum_{n_1,\dots,n_k\geq1}\rho_{\Upsilon}(\epsilon)\frac{\tau_{\alpha_1,\beta_1}(n_1)\cdots \tau_{\alpha_k,\beta_k}(n_k)}{(n_1\cdots n_k)^{\frac12}}\times\\
&\quad\times\e{\frac{a(\pm_1n_1\pm_2\cdots\pm_kn_k)}{q}}V_{\balpha,\bbeta}\pr{\frac{n_1\cdots n_k}{q^k}},
}
with $\rho_{\Upsilon}(\epsilon):=\prod_{i\in\Upsilon} (\pm_i1)$.
\end{lemma}
\begin{proof}
By contour integration and the functional equation, we have
\est{
\prod_{i=1}^k \Lambda_{i;\alpha_i,\beta_i}\pr{\tfrac12,\tfrac aq}&=\frac1{2\pi i}\bigg(\int_{(2)}-\int_{(-2)}\bigg)\prod_{i=1}^k \Lambda_{i;\alpha_i,\beta_i}(\tfrac12+s,\tfrac aq) \cdot G_{\balpha,\bbeta}(s)\frac{ds}s\\
&=\frac1{2\pi i}\int_{(2)}\prod_{i=1}^k \Lambda_{i;\alpha_i,\beta_i}(\tfrac12+s,\tfrac aq) \cdot G_{\balpha,\bbeta}(s)\frac{ds}s\\
&\quad+\frac1{2\pi i}\int_{(2)}\prod_{i=1}^k \Lambda_{i;-\alpha_i,-\beta_i}(\tfrac12,\tfrac {\overline a}q) \cdot G_{-\balpha,-\bbeta}(s)\,\frac{ds}s.\\
}
Now, expanding the Estermann functions into their Dirichlet series, we see that
\est{
&\frac1{2\pi i}\int_{(2)}\prod_{i=1}^k \frac{\Lambda_{i;\alpha_i,\beta_i}\pr{\frac12+s,\frac aq}}{\Gamma_i\pr{\tfrac {\frac12+\alpha_i}2}\Gamma_i\pr{\tfrac {\frac12+\beta_i}2} \pr{\frac{q}{\pi}}^{\frac12+\frac {\alpha_i+\beta_i}2}}\cdot G_{\balpha,\bbeta}(s)\frac{ds}s\\
&\hspace{5em}=\frac{i^{-|\Upsilon|}}{2^{k}}\sum_{n_1,\dots,n_k\in\Z\setminus\{0\}}\sgn(\prod_{i\in\Upsilon}n_i)\frac{\tau_{\alpha_1,\beta_1}(|n_1|)\cdots\tau_{\alpha_k,\beta_k}(|n_k|)}{|n_1\cdots n_k|^{\frac12}}\e{\frac {a(n_1+\cdots+n_k)}q}\\
&\hspace{5em}\quad\times \frac1{2\pi i}\int_{(2)}G_{\balpha,\bbeta}(s)g_{\balpha,\bbeta}(s)\pr{\frac{|n_1\cdots n_k|}{q^{k}}}^{-s}\frac{ds}s
}
and the lemma follows.
\end{proof}

\subsection{Estimates for the Estermann function}\label{eef}
In this section we give two bounds for certain averages of products the Estermann function and the periodic zeta-function. Both bounds depend on estimates for Kloosterman sums, more specifically on Weil's bound and on (a minor modification of) a bound by Deshouilliers and Iwaniec~\cite{DI}. We recall that the classical Kloosterman sum is defined as
\est{
S(m,n;\ell):=\sumstar_{c\mod \ell}\e{\frac{mc+n\overline c}{\ell}}
}
for any $c,m,n\in\Z,c\geq1$, where $\sumstar$ indicates that the sum is over $c\mod \ell$ such that $(c,\ell)=1$. 
Also, we recall that Weil's bound gives
$S(m,n;\ell)\ll d(\ell)(m,n,\ell)^{\frac12}\ell^{\frac12}.$
Using this bound we obtain the following Lemma.

\begin{lemma}\label{DIbound2}
Let $r>0$, $0<\delta<1$, $C\geq2$, $\eta_0\neq0$ and $(\eta_1,\dots,\eta_r)\in\{\pm 1\}^r$. Let $|a|\leq 2C\delta$, $|b|\leq C\delta$ and $|a_j|,|b_j|<\delta$ for $j=1,\dots,r$.
Then, for some $A>0$ we have
\es{\label{shjgk}
&\sum_{\ell\geq1}\frac{1}{\ell^{C+a}}\sumstar_{h\mod \ell}F(1+C(s-\tfrac12)+b,\tfrac{\eta_0h}\ell)\prod_{j=1}^ r D_{a_j,b_j}(\tfrac12,\tfrac {\eta_jh}\ell)\\
&\hspace{21em}\ll_\delta (AC/\delta)^{A(r+C)}(1+|s|)^{A(r+C)}
}
in the strip 
\est{
-\frac12+\frac{r+\frac32}{C+r-\frac12}+8\delta <\Re(s)<\frac12-2\delta.
}
Moreover, the left hand side of~\eqref{shjgk} is meromorphic in the half plane $\Re(s)>-\frac12+\frac{r+\frac32}{C+r-\frac12}+8\delta$ with poles at $s=\tfrac12-a_j$ and $s=\tfrac12-b_j$ for $j=1,\dots,r$ and $s=\frac12-b/C$ and these poles are simple if $a_1,\dots,a_r,b_1,\dots,b_r$ and $b/C$ are all distinct.
\end{lemma}
\begin{proof}
For $L\geq 1$, let
\est{
H_L(s)&:=\sum_{L<\ell\leq 2L}\frac{1}{\ell^{C+a}}\sumstar_{h\mod \ell}F(1+C(s-\tfrac12)+b,\tfrac{\eta_0h}\ell)\prod_{j=1}^ r D_{a_j,b_j}(\tfrac12+s,\tfrac {\eta_jh}\ell),\\
K(s)&:=\prod_{j=1}^r(s-\tfrac12+a_j)(s-\tfrac12+b_j).
}
Notice that if $\ell\neq1$ and $(h,\ell)=1$ then $F(x,h/\ell)$ is entire and thus so is $H_L(s)K(s)$ for all $L\geq1$. Now, if $\Re(s)=\frac12+2\delta$, then a trivial bound gives
 \es{\label{pl1}
H_L(s)K(s)\ll (1+|s|^{2r})(A/\delta)^{2r+1}L^{-C+2+2\delta C},
}
where, here and in the following, $A$ denotes a sufficiently large positive constant, which might change from line to line.

Next, take $\Re(s)=-\frac12-2\delta$. Then, applying the functional equations~\eqref{feest} and~\eqref{fef} to $D$ and $F$, expanding $D$ and $F$ into their Dirichlet series, and using Stirling's formula in the crude form
\es{\label{cbg}
\Gamma(\sigma+it)&\ll c^{-1} (1+A|\sigma|)^{|\sigma|}(1+|t|)^{\sigma-\frac12}e^{-\frac{\pi}{2}|t|},\qquad \sigma\geq c>0,
}
we see that 
\est{
H_L(s)K(s)&\ll A^{r}C^{AC}(1+|s|)^{A(r+C)}L^{r-1+(5C+6r)\delta}\sum_{L<\ell\leq 2L}\sum_{u=1}^\ell |F(C(\tfrac12-s)-b,u/\ell)|\times\\
&\hspace{1em}\times  \sum_{n_1,\cdots,n_r\in\Z_{\neq 0}}\frac{|\tau{-a_1,-b_1}(|n_1|)\cdots \tau_{-a_r,-b_r}(|n_r|)|}{|n_1^{1+2\delta}\cdots n_r^{1+2\delta}|} |S(\eta_0u, n_1+\cdots+n_k;\ell)|,
}
and thus by Weil's bound we obtain
\es{\label{pl2}
H_L(s)K(s)&\ll (A/\delta)^{2r+5}C^{AC}(1+|s|)^{A(r+C)}L^{r+\frac32+6(r+C)\delta},
}
when $\Re(s)=-\frac12-2\delta$. Thus, by~\eqref{pl1},~\eqref{pl2} and the Phragm\'en-Lindel\"of principle, if $-\frac12-2\delta\leq\Re(s)\leq \frac12+2\delta$
we have
\est{
H_L(s)K(s)&\ll (A/\delta)^{2r+5}C^{AC}(1+|s|)^{A(r+C)}L^{r+\frac32-(C+r-\frac12)(\Re(s)+\frac12)+5\delta (r+C)}.
}
Moreover, if $|s-\frac12|>2\delta$ then $K(s)\gg \delta^{2r}$ and thus, if $-\frac12-2\delta\leq\Re(s)\leq \frac12-2\delta$, we have
\est{
H_L(s)&\ll (A/\delta)^{4r+5}C^{AC}(1+|s|)^{A(r+C)}L^{r+\frac32-(C+r-\frac12)(\Re(s)+\frac12)+5\delta (r+C)}.
}
It follows that if
\es{\label{cfs}
-\frac12+\frac{r+\frac32+6\delta (r+C)}{C+r-\frac12}\leq \Re(s)\leq \frac12-2\delta
}
then  
\est{
&\sum_{\ell>1}\frac{1}{\ell^{C+a}}\sumstar_{h\mod \ell}F(1+C(s-\tfrac12)+b,\tfrac{\eta_0h}\ell)\prod_{j=1}^ r D_{a_j,b_j}(\tfrac12+s,\tfrac {\eta_jh}\ell)\\[-0.3em]
&\hspace{23em}\ll_\delta (A/\delta)^{4r+6}C^{AC}(1+|s|)^{A(r+C)}.
}
Finally, the contribution of the $\ell=1$ term to the left hand side of~\eqref{shjgk} is
\est{
&\zeta(1+C(s-\tfrac12)+b)\prod_{j=1}^ r \zeta(\tfrac12+s+a_j)\zeta(\tfrac12+s+b_j)\ll (A/\delta)^{2r+1}C^{AC}(1+|s|)^{A(r+C)}
}
when $s$ satisfies~\eqref{cfs} and thus~\eqref{shjgk} follows. We conclude by remarking that the above computations also give the meromorphicity of the left hand side of~\eqref{shjgk} on $\Re(s)\geq-\frac12+\frac{r+\frac32+5\delta (r+C)}{C+r-\frac12}$.
\end{proof}

We now states a variation of a bound by Deshouilliers-Iwaniec for sums of Kloosterman sums (cf. Theorem~9 and (1.52) of~\cite{DI}), which is essentially implicit in~\cite{BHM} and~\cite{HWW} (cf. Theorem~1.4 of~\cite{Wat}). 
\begin{lemma}\label{llin}
Let $W$ be a smooth function supported in $[1,2]$ and satisfying $W^{(i)}(x)\ll C^i$ for $i=0,1,2$ and some $C>1$. Let $a_m,b_n\ll 1$ be sequences of complex numbers supported in $[M,2M]$ and $[N,2N]$ respectively. Then, for $q\geq1$ and $\eta\in\{\pm1\}$ we have
\es{\label{Linnik}
\sum_{m,n,\ell\geq1} W\pr{\ell/L}a_m b_nS(qm,\eta n;\ell)\ll_\eps q^{\vartheta+\eps} C^4 (L^{1+\eps}+q^\frac12)MN,
}
where $\vartheta=\frac7{64}$.
\end{lemma}
\begin{proof}
After splitting the sum over $m$ according to the common factor of $m$ with $q$, one can repeat the arguments used in the proof of Theorem~9 of~\cite{DI} but using the multiplicativity of Hecke-eigenvalues (which holds since we are in the case of level $1$, for which there are only new-forms) just before applying the spectral large sieve. Kim-Sarnak's bound~\cite{Kim} for Hecke eigenvalues then gives~\eqref{Linnik}.

The above argument was carried out in detail (sorting out also the delicate issues arising when considering a level different from $1$) in~\cite{BHM}, Theorem~4. We remark that the factor $(1+C/\sqrt{MN})^{2\vartheta}$ appearing there can be removed when the level is $D=1$ since Selberg's eigenvalue conjecture is known in the case of $\tn{SL}(2,\Z)$.
\end{proof}
\begin{remark}
Using the variation of the spectral large sieve given by Blomer and Mili\'cevi\'c in Theorem~8 of~\cite{BM}, one obtains a bound which improves upon~\eqref{Linnik} when the parameters are in certain ranges. It is likely that the use of such bound in combination with~\eqref{Linnik} would lead to a better bound for the error term in Theorem~\ref{mtws}. However for simplicity we choose to use~\eqref{Linnik} in all ranges, since this is sufficient for our purposes.
\end{remark}

Using Lemma~\ref{llin} we obtain the following result.
\begin{lemma}\label{DIbound}
For $r\geq 1$, let $ t_0,\cdots,t_r\in \R$, $(\eta_1,\dots,\eta_r)\in\{\pm1\}^{r}$, and let $\eta_0\neq0$. Furthermore, let $|a_j|,|b_j|<\delta$ for $j=1,\dots,r$ and some $0<\delta<1$. Finally, let $L>0$ and let $W(x)$ be a function supported on $[1,2]$ with $W^{(i)}(x)\ll 1$ for $i=0,1,2$. Then, if $w\in\C$ and $\sigma\geq 2\delta$, we have 
\est{
\mathfrak S:=\sum_{\ell\geq1}\frac{W(\ell/L)}{\ell^{1+w}}\sumstar_{h\mod \ell}F(1+\sigma+it_0,\tfrac{\eta_0h}\ell)\prod_{j=1}^rD_{a_j,b_j}(-\sigma+it_j,\tfrac {\eta_jh}\ell)\\[-0.5em]
}
is bounded by
\es{\label{fddca}
\mathfrak S\ll_\delta L^{r(3\sigma+1)-\Re(w)} \frac{A^{r(\sigma+1)}}{\delta^{2r}}K_{r}(\sigma,w,t_1,\dots,t_j)\times
\begin{cases}
 |\eta_0|^{ \vartheta +\delta} & \tn{if $L\geq |\eta_0|^\frac12$,}\\
 |\eta_0|^{\frac16+\frac \vartheta 3+\delta} & \tn{always,}
\end{cases}
}
for some absolute $A>0$ and where 
\est{
K_{r}(s,w,t_1,\dots,t_j):=(1+\sigma)^{2r(2\sigma+1)}(1+|w|)^{4}\prod_{j=1}^{r}(1+|s|+|t_j|)^{1+4\sigma}
}
\end{lemma}
\begin{proof}
Applying the functional equation~\eqref{feest}, expanding $D$ and $F$ into their Dirichlet series, and using~\eqref{cbg} we obtain 
\est{
&\mathfrak S \ll_\delta A^r (1+A\sigma)^{2r(\sigma+\delta+1)}\Big(\prod_{j=1}^r(1+|t_j|)^{2(\sigma+\delta)+1} \Big)\Big|\sum_{\ell\geq1}\sum_{m\geq1} \sum_{n\in\Z}\frac{W_0(\ell) f_n}{m^{1+\sigma+it_0}}S(\eta_0m, n;\ell)\Big|,
}
where 
\est{
W_0(x)&:=W(x/L)x^{r(2\sigma+1)-1-w-\sum_{j=1}^r(t_j+a_j+b_j)}\ll (2L)^{r(3\sigma+1)-1-\Re(w)},\\ 
f_n&:=\sum_{\substack{n_1,\cdots,n_r\in\Z_{\neq 0},\\ n_1+\cdots+n_r=n}}\frac{\tau_{-a_1,-b_1}(n_1)}{n_1^{1+s-it_1}}\cdots \frac{\tau_{-a_r,-b_r}(n_r)}{n_r^{1+s-it_r}}\ll_\delta \pr{ A/{\delta}}^{2r}\frac1{|n|^{1+\delta/2}}.\\
}
Splitting the sums over $n$ and $m$ into diadic blocks and applying~\eqref{Linnik} one easily gets the bound
\es{\label{3asc}
\mathfrak S\ll_\delta L^{r(3\sigma+1)-\Re(w)}\frac{A^{r(\sigma+1)}}{\delta^{2r}}K_{r}(\sigma,w,t_1,\dots,t_j)|\eta_0|^{ \vartheta +\delta} (1+|\eta_0|^\frac12/L),
}
which gives~\eqref{fddca} in the case $L\geq |\eta_0|^\frac12$. Applying Weil's bound rather than~\eqref{Linnik}, one obtains
\es{\label{3asc2}
\mathfrak S\ll_\delta L^{r(3\sigma+1)-\Re(w)} \frac{A^{r(\sigma+1)}}{\delta^{2r}} K_{r}(\sigma,w,t_1,\dots,t_j) L^\frac12,
}
and taking the minimum between~\eqref{3asc} and~\eqref{3asc2} one gets~\eqref{fddca} also in the case $L< |\eta_0|^\frac12$.
\end{proof}

\section{Some assumptions}\label{assumptions}
In this section we set up some notation and make some simplifying assumptions, which we will use throughout the rest of the paper.

First, $q$ will always denote a prime, $k$ an integer greater than $2$, and $\Upsilon$ a subset of $\{1,\dots,k\}$ with even cardinality. Moreover we shall use the convention that $A$ and $\eps$ denote a sufficiently large and arbitrarily small positive constants on which the implicit bounds are allowed to depend and whose values might change from line to line.

Also, we assume $\balpha=(\alpha_1,\dots,\alpha_k)\in \mathbb A _{2C}^k$, $\bbeta=(\beta_1,\dots,\beta_k)\in \mathbb A _{C/2}^k$ for some constant $C>0$ (with $4C/\log q\leq 1/10$), where $\mathbb A_{r}$ denotes the annulus   $\{s\in\C\mid  C/\log q \leq |s|\leq 2C/\log q\}$. This assumption can then be removed by analytic continuation and the maximum modulus principle, since both the left hand side and, by~\eqref{mtma}, the main term on the right hand side of~\eqref{asvfd}  are analytic functions of the shifts in $|\alpha_i|,|\beta_i|\leq 4C/\log q$. We remark in particular, that with the above assumption, we have $|\alpha_i|,|\beta_i|,|\alpha_i-\beta_i|\asymp 1/\log q$.

Moreover, for the rest of the paper we fix an entire function $G_{\balpha,\bbeta}(s)$ as follow:
\es{\label{dfnG}
G_{\balpha,\bbeta}(s):=\frac{Q_{\balpha,\bbeta}(s)}{Q_{\balpha,\bbeta}(0)}\frac{\xi(\frac 12+s)}{\xi(\frac12)},
}
where $\xi(s):=\frac12s(s-1)\pi^{- \pi/2}\Gamma(\frac12s)\zeta(s)$ is the Riemann $\xi$-function and
\est{
Q_{\balpha,\bbeta}(s):=\prod_{i=1}^k(s^2-(\alpha_i-\beta_i)^2).
}
By the functional equation for the Riemann zeta-function we have $G_{\balpha,\bbeta}(-s)=G_{-\balpha,-\bbeta}(s)$ and so $G_{\balpha,\bbeta}(s)$ satisfies the hypothesis of Lemma~\ref{adwede}. Moreover, using Stirling's formula~\eqref{cbg} we also obtain
\es{\label{boundforG}
G_{\balpha,\bbeta}(s)\ll   (\log q)^{2k} e^{-C_1|t|}(1+|\sigma|)^{A(|\sigma|+k)},
}
for all $s=\sigma+it\in\C$ and some $C_1>0$.

Finally, we notice that from the functional equations~\eqref{afce}, for $i=1,\dots,k$, we have the convexity bound
\est{
D_{i}\pr{\tfrac12+\alpha_i,\alpha_i-\beta_i,\tfrac aq}\ll q^{\frac 12}(\log q)^2
}
and so trivially $M_{\Upsilon,k}\ll (Aq^{\frac 12}(\log q)^2)^k$. Also, from~\eqref{mtma} it is easy to see that one also has 
\est{
\sum_{\{\alpha_i',\beta_i'\}=\{\alpha_i,\beta_i\}}\mathscr M_{\balpha',\bbeta'}\ll q^{\frac k2-1}(A\log q)^k.
}
It follows that Theorem~\ref{mtws} is trivial if $k\gg \log q/\log \log q$ since in this case $(Ak)^{Ak}\gg q^{A/2}$. Thus, we will assume $k=o(\log q/\log \log q)$. In particular, for $q$ large enough we have $|\alpha_i|,|\beta_i|<1/\log q<1/(k\log\log q)<\frac{\eps}{2k}$ and a fortiori
\est{
|\alpha_1|+\cdots+|\alpha_k|+|\beta_1|+\cdots+|\beta_k|< \eps.
}
Moreover, notice that under these assumptions we also have the inequality $(k/\eps)^{Ak}\ll (\log q)^{Ak}\ll {}q^{\eps}$ which we shall often use. 


\section{Dividing into diagonal and off-diagonal terms and structure of the proof}\label{bmp}

By the approximate functional equation~\eqref{feqw} and the orthogonality of additive characters, we can decompose $M_{\Upsilon,k}$ into diagonal and off-diagonal terms:
\est{
M_{\Upsilon,k}&:=\frac1{\varphi(q)}\sum_{a=1}^{q-1}\prod_{i=1}^k D_{i;\alpha_i,\beta_i}(\tfrac12,\tfrac aq)
=\mathcal D_{\balpha,\bbeta}+X_{\balpha,\bbeta}\mathcal D_{-\balpha,-\bbeta}+\mathcal {O}_{\balpha,\bbeta}+X_{\balpha,\bbeta}\mathcal O_{-\balpha,-\bbeta},
}
where
\est{
\mathcal D_{\balpha,\bbeta}&:=\frac{i^{|\Upsilon|}}{2^{k}}\sum_{\epsilon\in\{\pm1\}^k}\rho_{\Upsilon}(\epsilon)\hspace{-0.5em}\sum_{\substack{\pm_1n_1+\cdots +\pm_kn_k=0}} \hspace{-0.5em}\frac{\tau_{\alpha_1,\beta_1}(n_1)\cdots \tau_{\alpha_k,\beta_k}(n_k)}{(n_1\cdots n_k)^{\frac12}}V_{\balpha,\bbeta}\pr{\frac{n_1\cdots n_k}{q^k}},\\
\mathcal O_{\balpha,\bbeta}&:=\frac{i^{|\Upsilon|}}{2^{k}}\sum_{\epsilon\in\{\pm1\}^k}\rho_{\Upsilon}(\epsilon)\mathcal{O}'_{\epsilon,\balpha,\bbeta},\\
\mathcal{O}'_{\epsilon,\balpha,\bbeta}&:=\sum_{d|q}d\,\frac{\mu(\frac qd)}{\varphi(q)}
\sum_{\substack{ d|( \pm_1n_1+\cdots +\pm_kn_k),\\ \pm_1n_1+\cdots +\pm_kn_k\neq0}}\hspace{-0.4em}
\frac{\tau_{\alpha_1,\beta_1}(n_1)\cdots \tau_{\alpha_k,\beta_k}(n_k)}{(n_1\cdots n_k)^{\frac12}}V_{\balpha,\bbeta}\pr{\frac{n_1\cdots n_k}{q^k}}\\
}
and the sum over $\epsilon$ is a sum over $\epsilon=\{\pm_11,\dots,\pm_k1\}\in\{1,-1\}^k$.

The diagonal term $\mathcal D_{\balpha,\bbeta}$ will be treated in Section~\ref{dts}, using the results of~\cite{Bet}. The terms with $d=1$ in $\mathcal O_{\epsilon,\balpha,\bbeta}'$ could be easily dealt with in a simple way, however it is more convenient to keep them together with the other off-diagonal terms. 


\begin{lemma}\label{asaca}
We have
\es{\label{dteq}
\mathcal D_{\balpha,\bbeta}=\mathscr D_{\balpha,\bbeta}+O(q^{\frac k2-\frac{2k}{k+1}+\eps}),
}
where 
\est{
\mathscr D_{\balpha,\bbeta}:=\sum_{\substack{\mathcal I\cup \mathcal J=\{1,\dots,k\},\ \mathcal I\cap\mathcal J=\emptyset,\\|\mathcal I|>|\mathcal J|+1,\ |\mathcal I \cap\Upsilon|\tn{ even}}}\sum_{\substack{\{\alpha_i^\prime,\beta_i^\prime\}=\{\alpha_i,\beta_i\}\ \forall i\in\mathcal I,\\ (\alpha_j^\prime,\beta_j^\prime)=(\alpha_j,\beta_j)\ \forall j\in\mathcal J}}\mathscr D'_{\mathcal I;\balpha',\bbeta'}
}
and $\mathscr D'_{\mathcal I;\balpha,\bbeta}$ is as defined in~\eqref{mtmt}.
\end{lemma}

For the off-diagonal terms we introduce partitions of unity. We need a function $P:\R_{\geq0}\rightarrow\R_{\geq0}$, satisfying
\est{
\sumdagger_{N}P(x/N)=1,\qquad \forall x>0,
}
where by $\sumdagger$ we mean that the index runs through the elements of a certain (fixed) set of positive real numbers such that $\sumdagger_{X^{-1}\leq N\leq X}1\ll \log X$. 
%
%
Also, we require that $P(x)$ is supported on $1\leq x\leq 2$ and $P^{(j)}(x)\ll j^{Aj}$ for some $A>0$. It is not difficult to construct such a partition.\footnote{For example take the set of indexes in $\sumdagger$ to be $\{(\frac 32)^{n}\mid n\in\Z\}$ and $P(x)=\int_{1}^{\frac32}\eta(xy)\frac{dy}{y}$, where $\eta(x)=Ce^{-\frac1{1-(4x-7)^2}}$ for $|x-\frac{7}4|<\frac14$ and $\eta(x)=0$ otherwise, and where  $C$ is such that $\int_\R\eta(y)\frac{dy}{y}=1$.%
}
Notice that under these conditions, the Mellin transform of $P(x)$,
\est{
\tilde P(s):=\int_{0}^{\infty}P(x)x^{s-1}dx,
}
is entire and satisfies
\es{\label{fnn}
\tilde P(\sigma+it)\ll (1+j+|\sigma|)^{Aj}A^{|\sigma|}(1+|t|)^{-j}\qquad \forall j\geq0.
}
Using partitions of unity we can decompose $\mathcal O'_{\balpha,\bbeta}$ into
\est{
\mathcal O_{\epsilon,\balpha,\bbeta}:=\sumdagger_{N_1,\cdots,N_k}\mathcal O''_{\epsilon,\balpha,\bbeta}(N_1,\cdots,N_k),
}
where $\mathcal O''_{\epsilon,\balpha,\bbeta}(N_1,\cdots,N_k)$ is defined as $\mathcal {O}'_{\epsilon,\balpha,\bbeta}$, with the only difference that the summands are multiplied by $P(n_1/N_1)\cdots P(n_k/N_k)$.  In the following we will often omit to indicate the dependencies from $N_1,\dots N_k$ for ease of notation.

The following two Lemmas summarize our results on the off-diagonal terms. The first Lemma, which is effective when $N_1,\cdots, N_k$ are close together, uses the spectral theory of automorphic forms (via the bounds proven in Section~\ref{eef}) and is proven  in Section~\ref{ttctd}. The second lemma, which is effective when one of the $N_i$ is quite larger than the others, uses the bounds for sums of Kloosterman sums proven by Young in~\cite{You11b} and is proven in Section~\ref{ttftd}.
\begin{lemma}\label{ctd}
Let $N_{\tn{max}}$ be the maximum among $N_1,\dots,N_k$. Then
\est{
&\sum_{\epsilon\in\{\pm1\}^k}\rho_{\Upsilon}(\epsilon)\mathcal O''_{\epsilon,\balpha,\bbeta}(N_1,\cdots,N_k) = \mathcal{ M}_{\balpha,\bbeta}(N_1,\cdots,N_k)+ \mathcal{ E}_{1;\balpha,\bbeta}(N_1,\cdots,N_k),\\
}
where 
\es{\label{gdfsc}
\mathcal{ E}_{1;\balpha,\bbeta} \ll  \frac{N_{\max}^{\eps}}{q^{1-\eps} } \pr{\frac{q^\vartheta N_{\max}^{\frac {k}2+\frac12}}{(N_1\cdots N_k)^{\frac12}}+\frac{q^{\frac k2-\frac13+\frac\vartheta 3} N_{\max}^\frac12}{(N_1\cdots N_k)^{\frac12}}+\frac{q^{\frac16+\frac \vartheta3}(N_1\cdots N_k)^\frac12}{ N_{\max}}+\frac{(N_1\cdots N_k)^\frac12}{N_{\max}^{\frac12}}}
}
and $\mathcal M_{\balpha,\bbeta}(N_1,\cdots,N_k) $ is defined in~\eqref{dfmm}.
Moreover,
\es{\label{bmaf}
 \mathcal M_{\balpha,\bbeta}(N_1,\cdots,N_k) \ll q^{\eps}(N_1\cdots N_k)^{\frac12}N_{\tn{max}}^{-1+\eps}.
}
\end{lemma}

\begin{lemma}\label{pfofa}
Let $N_{\tn{max}}$ be the maximum among $N_1,\dots,N_k$. Then 
\est{
\mathcal O''_{\epsilon,\alpha,\beta}(N_1,\cdots,N_k) &\ll q^{\eps}\prbigg{ \pr{\frac{N_1\cdots N_k}{N_{\tn{max}}}}^{\frac12-\frac{1}{2(k-1)}}\prbigg{\frac{q^{\frac34}}{N_{\tn{max}}^{1/2}}+1}  +\frac{(N_1\cdots N_k)^{\frac12}}{ N_{\tn{max}}^{3/4}}}.
}
\end{lemma}

Notice that in the crucial case $N_1\cdots N_k\asymp q^{k}$ Lemma~\ref{ctd} is non-trivial for $N_{\max}\ll q^{2-2\frac{\vartheta+1}{k+1}-\delta}$ for any fixed $\delta>0$, whereas Lemma~\ref{pfofa} is non-trivial as long as $N_{\tn max}\gg_kq^{\frac43+\delta}$. In particular, in order to have a non-trivial bound for all ranges we need $\vartheta<\frac{k-2}{3}$ and so for $k=3$ we need $\theta<\frac13$.

The following lemma, which we shall prove in Section~\ref{mmtt}, allows us to combine the various main terms.
\begin{lemma}\label{smtmt}
We have 
\est{
&\frac{i^{|\Upsilon|}}{2^{k}} \sumdagger_{\substack{N_1,\dots,N_k,\\ N_1\cdots N_k\ll q^{k+\eps}}}( \mathcal M_{\balpha,\bbeta}(N_1,\cdots,N_k)+X_{\bbeta,\balpha}\mathcal M_{-\bbeta,-\balpha}(N_1,\cdots,N_k))\\
&=-(\mathscr D_{\balpha,\bbeta}+X_{\bbeta,\balpha}\mathscr D_{-\bbeta,-\balpha})+\sum_{\{\alpha_i',\beta_i'\}=\{\alpha_i,\beta_i\}}\mathscr M_{\balpha',\bbeta'}+O(q^{\frac k2-\frac32+\iota_k+\eps}),
}
where $\mathcal M_{\alpha,\beta}$ is as defined in~\eqref{gabcd} and $\iota_k=\frac3{14}$ if $k=4$ and $\iota_k=0$ otherwise. 
\end{lemma}

We conclude the section with the deduction of Theorem~\ref{mtws} from the above lemmas. 
\begin{proof}[Proof of Theorem~\ref{mtws}]
First we remark that by~\eqref{bmaf}, we can rewrite Lemma~\ref{pfofa} as
\est{
&\sum_{\epsilon\in\{\pm1\}^k}\rho_{\Upsilon}(\epsilon)\mathcal O''_{\epsilon,\alpha,\beta}(N_1,\cdots,N_k) = \mathcal{ M}_{\alpha,\beta}(N_1,\cdots,N_k)+\mathcal{ E}_{2;\alpha,\beta}(N_1,\cdots,N_k),
}
where
\est{
\mathcal{ E}_{2;\alpha,\beta}&\ll   q^{\eps}\prbigg{ \pr{\frac{N_1\cdots N_k}{N_{\tn{max}}}}^{\frac12-\frac{1}{2(k-1)}}\prbigg{\frac{q^{\frac34}}{N_{\tn{max}}^{1/2}}+1}  +\frac{(N_1\cdots N_k)^{\frac12}}{ N_{\tn{max}}^{3/4}}}.\\
}
Moreover, we also have 
\est{
\sum_{\epsilon\in\{\pm1\}^k}\rho_{\Upsilon}(\epsilon)\mathcal O''_{\epsilon,\alpha,\beta}(N_1,\cdots,N_k) = \mathcal{ M}_{\alpha,\beta}(N_1,\cdots,N_k)+\mathcal{ E}_{3;\alpha,\beta}(N_1,\cdots,N_k),
}
with $\mathcal{ E}_{3;\alpha,\beta}(N_1,\cdots,N_k)\ll q^{-1+\eps}(N_1\cdots N_k)^{\frac12}$, since both $O''_{\epsilon,\alpha,\beta}$ and $\mathcal{ M}_{\alpha,\beta}$ satisfy trivially such a bound.
 Thus, by Lemma~\ref{smtmt} (and adding the condition $N_1\cdots N_k\ll q^{k+\eps}$ at a negligible cost) we have 
\est{
\mathcal O_{\balpha,\bbeta}&=\frac{i^{|\Upsilon|}}{2^{k}}\sum_{\epsilon\in\{\pm1\}^k}\rho_{\Upsilon}(\epsilon) \sumdagger_{\substack{N_1,\dots,N_k,\\ N_1\cdots N_k\ll q^{k+\eps}}}( \mathcal O_{\alpha,\beta}(N_1,\cdots,N_k)+X_{\bbeta,\balpha}\mathcal O_{-\bbeta,-\balpha}(N_1,\cdots,N_k))+O(1)\\
&=-(\mathscr D_{\balpha,\bbeta}+X_{\bbeta,\balpha}\mathscr D_{-\bbeta,-\balpha})+\sum_{\{\alpha_i',\beta_i'\}=\{\alpha_i,\beta_i\}}\mathcal M_{\balpha',\bbeta'}+\mathscr{ E}_{\balpha,\bbeta}+O(q^{\frac k2-\frac32+\iota_k+\eps}),
}
where 
\est{
\mathscr{ E}_{\balpha,\bbeta}\ll\max_{\substack{N_1,\dots,N_k,\\ N_1\cdots N_k\ll q^{k+\eps}}}\pr{\min(\mathscr{ E}_{1},\mathscr{ E}_{2},\mathscr{ E}_{3})}.
}
Thus, since the term $-(\mathscr D_{\balpha,\bbeta}+X_{\bbeta,\balpha}\mathscr D_{-\bbeta,-\balpha})$ cancels out with the main term of the diagonal term given by~\eqref{dteq}, to conclude the proof of Theorem~\ref{mtws} we just need to show that $\mathscr{ E}_{\balpha,\bbeta}\ll q^{\frac k2-1-\delta_k+\eps}$. Writing $N_{\tn{max}}=q^{a}$ and $N_1\cdots N_k=q^{b}$  (and considering only the contribution from the first summand in~\eqref{gdfsc}, since it easy to see the other terms produce a contribution which is $O(q^{\frac k2-\frac32+\eps})$), we have that it is sufficient to show that
\est{
\max_{i=1,2,3}\max_{\substack{ a\leq b\leq k,\\ ka\geq b}}\min(\tfrac{k+1}2a-\tfrac b 2-1+\vartheta,L_i(a,b),\tfrac b2-1)=\tfrac k2-\tfrac 32+\tfrac {3(3+2\vartheta)}{2 (2 k+5)}=\tfrac k2-1-\delta_k,
}
for $k\geq3$, 
where
\est{
L_1(a,b):=\tfrac34-\tfrac a2+(b-a)(\tfrac12-\tfrac{1}{2(k-1)}),\quad L_2(a,b):=(b-a)(\tfrac12-\tfrac{1}{2(k-1)}),\quad L_3(a,b):=\tfrac b2-\tfrac 34a.
}
If the maximum is attained at the interior of $\{a\leq b\leq k, ka\geq b\}$, then it must occur when $\frac{k+1}2a-\frac b 2-1=\frac b2-1=L_i(a,b)$ for $i=1,2,$ or $3$ and so it would be $\frac {7k}{20}-\frac{13}{20}$, $\frac k3-\frac 23$ and $\frac k3-\frac 23$ respectively.
%
%
%
%
%
Along the lines $a=b$, $ka=b$ and  $b=k$ we have
\est{
&\max_{i=1,2,3}\max_{\substack{0\leq a \leq k}}\min(\tfrac{k}2a-1+\vartheta,L_i(a,a),\tfrac a2-1)=\max_{\substack{0\leq a \leq k}}\min(L_i(a,a),\tfrac a2-1)=0,\\
&\max_{i=1,2,3}\max_{\substack{ 0\leq a\leq 1}}\min(\tfrac a2-1+\vartheta,L_i(a,ka),\tfrac {ka}2-1)\leq-\tfrac12+\vartheta\leq0\\
&\max_{i=1,2,3}\max_{\substack{ 1\leq a\leq  k}}\min(\tfrac{k+1}2a-\tfrac k 2-1+\vartheta,L_i(a,k),\tfrac k2-1)\\
&\qquad=\max(\tfrac k2-\tfrac 74+\tfrac{8 k (2+\vartheta)-19-12 \vartheta}{4 (k^2 +2k-4)},\tfrac k2-\tfrac 32+\tfrac{2 (k+\vartheta)-5-4\vartheta}{2 (k^2+k-3)},\tfrac k2-\tfrac 32+\tfrac {3(3+2\vartheta)}{2 (2 k+5)})=\tfrac k2-\tfrac 32+\tfrac {3(3+2\vartheta)}{2 (2 k+5)},\\
}
for $k\geq3$ and $\vartheta\leq\frac13$.
\comment{
To compute the last one, observe that the first term and the $L_i$ are increasing and decreasing respectively and they reverse order on the borders. Thus for each of the $L_i$ the min occurs at the intersection provided that the value at such a point is less than $\frac k2-1$.

Also, notice that for $\vartheta>\frac13$ and $k=3$ the min for the min with $L_1$ and the one for $L_3$ are both $\frac k2-1$.

For the final step, for integers$\geq3$, third greater than second if $k>3$ and equal if $k=3$. Third always greater than second.
}%
Theorem~\ref{mtws} then follows. 
\end{proof}
\section{The diagonal terms}\label{dts}
In this section we prove Lemma~\ref{asaca} deducing it from the following Lemma in~\cite{Bet}. We recall that in Section~\ref{assumptions} we assumed $|\alpha_i|,|\beta_i|<\frac{\eps}{2k}$ for all $i=1,\dots,k$.
\begin{lemma}
For $\Re(s)>1-\frac1k-\frac1k\sum_{i=1}^k\min(\Re(\alpha_i),\Re(\beta_i))$, let 
\est{
\mathcal W_{\alpha,\beta}(s)&:=\frac{i^{|\Upsilon|}}{2^{k}}\sum_{\epsilon\in\{\pm1\}^k}\rho_\Upsilon(\epsilon)\sum_{\substack{\pm_1n_1+\cdots +\pm_kn_k=0}} \frac{\tau_{\alpha_1,\beta_1}(n_1)\cdots \tau_{\alpha_k,\beta_k}(n_k)}{(n_1\cdots n_k)^{s}}.\\
}
Also, let
\est{
\mathcal W^{\dagger}_{\alpha,\beta}(s)&:=\sum_{(\mathcal I,\balpha',\bbeta')\in S_{\balpha,\bbeta}}\frac{2^{|\mathcal J|+1}\pi^{\frac{|\mathcal I|}2-1}}{ |\mathcal  I|(s-1)+s_{\mathcal I;\balpha'}+1}\pr{ \prod_{i\in \mathcal  I}\zeta(1-\alpha_i'+\beta_i')\frac{\Gamma_i(-\frac{\alpha'_i}2+\frac{1+s_{\mathcal I;\balpha'}}{2|\mathcal I|})}{\Gamma_i(\frac12+\frac{\alpha'_i}2-\frac{1+s_{\mathcal I;\balpha'}}{2|\mathcal I|})}} 
\times\\
&\quad\times\sum_{\ell\geq1} \sumstar_{h\mod \ell} \frac1{\ell^{|\mathcal I'|-\sum_{i\in\mathcal I}(\alpha'_i-\beta'_i)}} \prod_{i\notin \mathcal I}D_i(1+\alpha'_i-\tfrac{1+s_{\mathcal I;\balpha'}}{|\mathcal I|},\alpha'_i-\beta'_i,\tfrac {h}{\ell }).\\
}
where $s_{\mathcal I;\balpha'}:=\sum_{i\in \mathcal I}\alpha_i'$ and 
\est{
S_{\balpha,\bbeta}:=\left\{(\mathcal I,\balpha',\bbeta')\left|
\begin{aligned}
&\mathcal I\subseteq\{1,\dots k\},\ |\mathcal I|>|\mathcal J|+1,\ |\mathcal I\cap \Upsilon|\tn{ even},\\ 
& \{\alpha_i',\beta_i'\}=\{\alpha_i,\beta_i\}\ \forall i\in \mathcal I,\ (\alpha_i',\beta_i')=(\alpha_i,\beta_i)\ \forall i\notin\mathcal I
\end{aligned}
\right. \right\}.
}
Then for any $\eps>0$ $\mathcal W_{\alpha,\beta}(s)-\mathcal W^{\dagger}_{\alpha,\beta}(s)$ extends to an holomorphic function on $\Re(s)\geq 1-\tfrac{2-4\eps}{k+1}$ and in such half plane it satisfies
$\mathcal W_{\alpha,\beta}(s)-\mathcal W^{\dagger}_{\alpha,\beta}(s) \ll (\frac k\eps(1+|s|))^{Ak}$.
\end{lemma}
\begin{proof}
Theorem~3 of~\cite{Bet} gives the meromorphic continuation and the bound for each $\epsilon$. Thus, one obtains the Lemma by summing over $\epsilon$ (for the simplification of the polar term one proceeds as in Lemma~\ref{gfar}; cf. also Remark~2 of~\cite{Bet}).
\end{proof}
\begin{proof}[Proof of Lemma~\ref{asaca}]
Writing $V_{\balpha,\bbeta}$ in terms of it's Mellin transform we have
\est{
\mathcal D_{\balpha,\bbeta}=\frac1{2\pi i}\int_{(2)}G_{\balpha,\bbeta}(s)g_{\balpha,\bbeta}(s)\mathcal W_{\alpha,\beta}(\tfrac12+s) q^{ks}\frac{ds} s.
}
We write $\mathcal W_{\alpha,\beta}(\tfrac12+s)$ as $\mathcal W^\dagger_{\alpha,\beta}(\tfrac12+s)+(\mathcal W_{\alpha,\beta}(\tfrac12+s)-\mathcal W^\dagger_{\alpha,\beta}(\tfrac12+s))$. For the second term we move the line of integration to $\Re(s)=\frac12-\tfrac{2-4\eps}{k+1}$ and bound trivially using~\eqref{boundforG} obtaining an error of size $O(k^{Ak}q^{\frac k2-\frac{2k}{k+1}+\eps})=O(q^{\frac k2-\frac{2k}{k+1}+\eps})$. For the first term we move the line of integration to $\Re(s)=-\frac12$ picking up the residues from the poles. We obtain
\es{\label{gcfa}
\mathcal D_{\balpha,\bbeta}=\sum_{\substack{\mathcal I\cup \mathcal J=\{1,\dots,k\},\ \mathcal I\cap\mathcal J=\emptyset,\\|\mathcal I|>|\mathcal J|+1,\ |\mathcal I \cap\Upsilon|\tn{ even}}}\sum_{\substack{\{\alpha_i^\prime,\beta_i^\prime\}=\{\alpha_i,\beta_i\}\ \forall i\in\mathcal I,\\ (\alpha_j^\prime,\beta_j^\prime)=(\alpha_j,\beta_j)\ \forall j\in\mathcal J}}\mathscr D_{\mathcal I;\balpha',\bbeta'}+O(q^{\frac k2-\frac{2k}{k+1}+\eps}),
}
where  
\es{\label{mtmt}
\mathscr D_{\mathcal I;\balpha,\bbeta}&=2^k\frac{G_{\balpha,\bbeta}(s_{\mathcal I;\balpha} )}{s_{\mathcal I;\balpha}\pi |\mathcal  I|}g_{\balpha,\bbeta}(s_{\mathcal I;\balpha})q^{ks_{\mathcal I;\balpha}}\pr{ \prod_{i\in \mathcal  I}\frac{\pi^{\frac12}\zeta(1-\alpha_i+\beta_i)}{2^{\frac12+\alpha_i+s_{\mathcal I,\balpha}}}\frac{\Gamma_i(\frac14-\frac{\alpha_i+s_{\mathcal I,\balpha}}{2})}{\Gamma_i(\frac14+\frac{\alpha_i+s_{\mathcal I,\balpha}}{2})}} \times\\
&\quad\times\sum_\ell\sumstar_{h\mod \ell} \frac1{\ell^{|\mathcal I|-\sum_{i\in\mathcal I}(\alpha_i-\beta_i)}} \pr{\prod_{j\in \mathcal J}D_j(\tfrac 12+\alpha_j+s_{\mathcal I,\balpha},\alpha_j-\beta_j,\pm_j\tfrac {h}{\ell })}\\
}
with $s_{\mathcal I;\balpha}:=\sum_{i\in \mathcal I}\alpha_i$.
\end{proof}

\section{The terms close to the diagonal}\label{ttctd}
In this section we prove Lemma~\ref{ctd}. First, we assume that $N_1$ is the maximum of $N_1,\dots,N_k$, as we can do since both the main term and the error terms in Lemma~\ref{ctd} are symmetric in the indexes. Moreover, since we assumed that $|\Upsilon|$ is even, then we have
\est{
\sum_{\epsilon\in\{\pm1\}^k}\rho_{\Upsilon}(\epsilon)\mathcal O''_{\epsilon,\alpha,\beta}=2\sum_{\substack{\epsilon\in\{\pm1\}^k,\\\pm_11=-1}}\rho_{\Upsilon}(\epsilon)\mathcal O''_{\epsilon,\balpha,\bbeta},
}
where here and in the following $\epsilon=(\pm_11,\pm_21,\dots,\pm_k1)$. 
We split $\mathcal O''_{\epsilon,\alpha,\beta}$ further, depending on the sign and the size of $\pm_*f:=-n_1\pm_2n_2\pm_3\cdots \pm_k n_k$ (with $f>0$), introducing another partition of unity controlling the size of $f$:
\es{\label{iml}
\sum_{\epsilon\in\{\pm1\}^k}\rho_{\Upsilon}(\epsilon)\mathcal O''_{\epsilon,\balpha,\bbeta}=2\sumdagger_{N_*\ll kN_1q^{\eps/k}}\sum_{\substack{\epsilon\in\{\pm1\}^{k},\\\pm_1=-1}}\rho_{\Upsilon}(\epsilon)\sum_{\pm_*1\in\{\pm1\}}\mathcal K_{\epsilon,\pm_*,\balpha,\bbeta},
}
where
\est{
\mathcal K_{\epsilon,\pm_*;\balpha,\bbeta}&:=\sum_{d|q}d\,\frac{\mu(\frac qd)}{\varphi(q)}\sum_{\substack{f\geq1,\\ f\equiv 0\mod d}}\,
\sum_{\substack{n_1,\dots,n_k\geq1,\\ n_1=\pm_2n_2\pm_3\cdots\pm_k n_k\pm_*f}}
\hspace{-0.5em}\frac{\tau_{\alpha_1,\beta_1}(n_1)\cdots \tau_{\alpha_k,\beta_k}(n_k)}{(n_1\cdots n_k)^{\frac12}}\times\\
&\quad\times V_{\balpha,\bbeta}\pr{\frac{n_1\cdots n_k}{q^k}}P\pr{\frac{n_1}{N_1}}\cdots P\pr{\frac{n_k}{N_k}}P\pr{\frac{f}{N_*}}.\\
}
Notice that in~\eqref{iml} we truncated the sum over $N_*$ at $N_*\ll kN_1q^{\eps/k}$, as we clearly could.

\subsection{Separating the variables arithmetically}
We wish to separate the variables in $\tau_{\alpha_1,\beta_1}(n_1)=\tau_{\alpha_1,\beta_1}(\pm_2n_2\pm_3\cdots\pm_kn_k\pm_*f)$. One can achieve this goal by using Ramanujan's identity 
\es{\label{hga}
\tau_{a,b}(n)=n^{-a}\tau_{0,b-a}(n)=n^{-a}\zeta(1-a+b)\sum_{\ell=1}^{\infty}\frac{c_\ell(n)}{\ell^{1-a+b}},
}
which holds for $n\neq0$ and $\Re(a-b)<0$. The coefficient ${c_\ell(n)}$ denotes the Ramanujan sum 
\est{
c_\ell(n):=\sumstar_{h\mod \ell}\e{\frac {nh}\ell}.
}
However, since~\eqref{hga} doesn't hold in a neighborhood of $a=b=0$, it is more convenient to follow Young's approach and use the following Lemma, which rephrase~\eqref{hga} as an approximate functional equation for $\tau_{a,b}(n)$.
\begin{lemma}
Let $n\in\Z_{>0}$ and let $a,b\in\C$. Then,
\es{\label{har}
\tau_{a,b}(n)=n^{-a}\sum_\ell\frac{c_\ell(n)}{\ell^{1-a+b}}\upsilon_{a-b}\pr{\frac{\ell^2}{n}}+n^{-b}\sum_\ell\frac{c_\ell(n)}{\ell^{1+a-b}}\upsilon_{b-a}\pr{\frac{\ell^2}{n}}
}
where
\est{
\upsilon_a\pr{x}=\int_{(c_w)}x^{-\frac w2}\zeta(1-a+w)\frac{G_{\balpha,\bbeta}(w)}{w}dw,
}
where $c_w>|\Re(a-b)|$ and $G_{\balpha,\bbeta}(w)$ is as defined in~\eqref{dfnG}.
\end{lemma}
\begin{proof}
See Lemma 5.4 of Young~\cite{You11b}.
\end{proof}

Applying~\eqref{har} and splitting the resulting sum over $\ell$ using another partition of unity (and adding the restriction $L\geq\frac12$ as we can do since $P$ is supported on $[1,2]$), we rewrite  $\mathcal K_{\epsilon,\pm_*;\balpha,\bbeta}$ as
\es{\label{hgfd}
 \mathcal K_{\epsilon,\pm_*;\balpha,\bbeta}=\sum_{\substack{\{\alpha_1^\prime,\beta_1^\prime\}=\{\alpha_1,\beta_1\},\\ (\alpha_j^\prime,\beta_j^\prime)=(\alpha_j,\beta_j)\ \forall j\neq 1}} \sumdagger_{L\geq\frac12} \mathcal L_{\balpha^\prime,\bbeta^\prime},
}
where $\balpha^\prime:=(\alpha_1^\prime,\dots,\alpha_k^\prime)$, $\bbeta^\prime:=(\beta_1^\prime,\dots,\beta_k^\prime)$ and 
\est{
\mathcal L_{\balpha,\bbeta}&:=\sum_{\substack{n_1,\dots,n_k,f\geq1,\\n_1=\pm_2n_2\pm_3\cdots\pm_kn_k\pm_*f}}\sum_\ell\sumstar_{h\mod \ell}\frac{c_\ell \pr{\pm_2n_2\pm_3\cdots\pm_kn_k\pm_*f}}{\ell^{1-\alpha+\beta}}\upsilon_{\alpha_1-\beta_1}\pr{\frac {\ell^2 }{n_1}}\times\\
&\quad\times\frac{\sigma_{\alpha_2-\beta_2}(n_2)\cdots \sigma_{\alpha_k-\beta_k}(n_k)}{n_1^{\frac12+\alpha_1+s}\cdots n_{k}^{\frac12+\alpha_k+s}}V_{\balpha,\bbeta}\pr{\frac{n_1\cdots n_k}{q^k}}P\pr{\frac{n_1}{N_1}}\cdots P\pr{\frac{n_k}{N_k}}P\pr{\frac{f}{N_*}}P\pr{\frac {\ell}L}.\\
}
Notice that we have omitted to indicate the dependency of $\mathcal L_{\balpha,\bbeta}$ from $\epsilon$ and $\pm_*$ in order to save notation.

Expressing $P$, $\upsilon_{\alpha_1-\beta_1}$ and $V$ in terms of their Mellin transform and making the change of variables $u_i\rightarrow u_i-s$, for $i=1,\dots,k$, we see that $\mathcal L_{\balpha,\bbeta}$ can be written as
\es{\label{jhk}
\mathcal L_{\balpha,\bbeta}&=\sum_{d|q}\frac{\mu(\frac qd)d}{\varphi(q)}\sum_{\substack{n_1,\dots,n_k,f\geq1,\ d|f,\\\pm_2n_2\pm_3\cdots\pm_kn_k\pm_*f>0}}\sum_\ell\sumstar_{h\mod \ell}\frac1{(2\pi i)^{k+3}}\int_{(\substack{c_s,c_w,\\\bcu,c_{u_*}})}\frac{N_*^{u_*}}{f^{u_*}} P\pr{\frac \ell L}\\
&\quad\times \frac{N_1^{u_1-s}\cdots N_k^{u_k-s}}{\ell^{1-\alpha_1+\beta_1+w}}\frac{\tau_{\alpha_2,\beta_2}(n_2)\cdots \tau_{\alpha_k,\beta_k}(n_k)c_\ell \pr{\pm_2n_2\pm_3\cdots\pm_kn_k\pm_*f}}{(\pm_2n_2\pm_3\cdots\pm_kn_k\pm_*f)^{\frac12+\alpha_1+u_1-\frac w2}n_2^{\frac12+u_2}\cdots n_{k}^{\frac12+u_k}}\\
&\quad \times\tilde P(u_*) \tilde P(u_1-s)\cdots  \tilde P(u_k-s)q^{ks}\frac{H_{\balpha,\bbeta}(w,s)}{ws} dwds\bdu du_*,
}
where $\bdu:=du_1\cdots du_k$, $\bcu$ denotes the lines of integration $c_{u_1},\dots, c_{u_k}$ and
\est{
H_{\balpha,\bbeta}(w,s):=\zeta(1+w-\alpha_1+\beta_1)G_{\balpha,\bbeta}(s)G_{\balpha,\bbeta}(w)g_{\balpha,\bbeta}(s).
}
Notice that, by the definitions~\eqref{dfnG} and~\eqref{dfngs} of $G_{\balpha,\bbeta}(s)$ and $g_{\balpha,\bbeta}(s)$, $H_{\balpha,\bbeta}(w,s)$ is entire and decays rapidly in both variables $w$ and $s$:
\es{\label{bfdsh}
H_{\balpha,\bbeta}(w,s)\ll  e^{-C_2(|\Im(s)|+|\Im(w)|}(1+|\Re(s)|+|\Re(w)|)^{A(|\Re(s)|+|\Re(w)|+k)},
}
for some $C_2>0$. As lines of integration, we take
\est{
c_s:=\eps/k,\quad c_{u_1} =-3k-\frac12-\alpha_1+7\eps, \quad c_{u_*}=c_{u_2}=\cdots=c_{u_k}=4k, \quad c_{w}=10\eps.
}
\subsection{Separating the variables analytically}
To complete the separation of the variables, we need also to deal with the factor $(\pm_2n_2\pm_3\cdots\pm_kn_k\pm_*f)^{\frac12+\alpha_1+u_1-\frac w2}$ in~\eqref{jhk}. In order to do so, we use Lemma~\ref{sml}, in Section~\ref{amell}.
We apply the lemma with $\kappa:=k+1$, $B:=3k$ and $v_1=\frac12-\alpha_1-u_1+\frac w2$, so that $\Re(v_1)=B+1-2\eps$. 
We get
\es{\label{ddrR}
\mathcal L_{\balpha,\bbeta}&=\sum_\nu
\frac{B!}{\nu_2!\cdots \nu_k!\nu_*!} \pr{\mathcal N_{\nu;\balpha,\bbeta}+\mathcal N'_{\nu;\balpha,\bbeta}}
}
where the sum is over $\nu=(\nu_2,\dots,\nu_k,\nu_*)\in\Z_{\geq0}^k$ satisfying
\est{
\nu_2+\cdots+\nu_k+\nu_*=B,\qquad \nu_{i}=0\text{ if $\pm_i=-1$},\qquad\nu_{*}=0\text{ if $\pm_*=-1$}
}
and $\mathcal N_{\nu;\balpha,\bbeta}$ is defined by
\es{\label{adfq}
\mathcal N_{\nu;\balpha,\bbeta}&:=\sum_{d|q}\frac{\mu(\frac qd)d}{\varphi(q)}\sum_{\substack{n_1,\dots,n_k,\ell\geq1,\\ f\geq1,\ d|f}}\frac{P(\ell/L)}{(2\pi i)^{2k+3}}\int_{(\substack{c_s,c_w,\bcu,\\c_{u_*},c_{v_*}})}\hspace{-0.9em}\frac{c_\ell \pr{\pm_2n_2\pm_3\cdots\pm_kn_k\pm_*f}}{f^{v_*+u_*-\nu_*}}N_*^{u_*} \\
&\hspace{-0.em}\quad\times \frac{q^{ks}N_1^{u_1-s}}{\ell^{1-\alpha_1+\beta_1+w}}\tilde P(u_*) \tilde P(u_1-s) \bigg(\prod_{i=2}^k\int_{(c_{v_i})}\frac{\tau_{\alpha_i,\beta_i}(n_i)N_i^{u_i-s}}{n_{i}^{\frac12+u_i+v_i-\nu_i}}\tilde P(u_i-s)\bigg)\\
&\hspace{-0.em}\quad\times \Psi_{\epsilon^*,B}\pr{\tfrac12-\alpha_1-u_1+\tfrac w2,\bv,v_*} \frac{H_{\balpha,\bbeta}(w,s)}{ws}dsdw\bdu \bdv du_*dv_*,
}
with $c_{v_2}=\cdots=c_{v_k}=c_{v_*}=\eps/k$, and $\mathcal N'_{\nu;\balpha,\bbeta}$ is defined in the same way with lines of integrations $c'_{v_2}=\cdots=c'_{v_k}=c'_{v_*}=\frac12$ in place of $c_{v_2},\dots,c_{v_k},c_{v_*}$. Also, in~\eqref{adfq} we used the notation $\bv:=(v_2,\dots,v_k)$, $\bdv:=dv_2\cdots dv_{k}$ and $\epsilon^*:=(\pm_11,\dots,\pm_k1,\pm_*1)$.

 The contribution of $\mathcal N'_{\nu;\balpha,\bbeta}$ can be bounded by moving the lines of integration $c_{u_i}$ to $c_{u_i}=2\eps+\nu_i$ for $i=2,\dots,k$ and $c_{u_*}$ to $c_{u_*}=\frac12+\nu_*+\eps$ and bounding trivially. We obtain
 \est{
\mathcal N'_{\nu;\balpha,\bbeta}&\ll q^{-1+\eps}N_1^{-B-\frac12+A\eps}N_*^{\frac12+\nu_*+2\eps}N_2^{\nu_2+2\eps}\cdots N_k^{\nu_k+2\eps}L^{-\eps}\vspace{10cm}
}
and thus
 \est{
\sum_\nu
\frac{B!}{\nu_2!\cdots \nu_k!\nu_*!}\mathcal N'_{\nu;\balpha,\bbeta}
&\ll  q^{-1+A\eps}N_1^{A\eps}L^{-\eps},
}
since $N_1$ is the maximum among $N_1,\dots,N_k$ and $N_*\ll kN_1q^{\eps/k}$.

Next, we open the Ramanujan sum in~\eqref{adfq} and we execute the sums over $n_2,\dots,n_k,f$ as we can do since the integrals and sums are absolutely convergent. We obtain 
\est{
\mathcal N_{\nu;\balpha,\bbeta}&=\sum_{d|q}\frac{\mu(\frac qd)}{\varphi(q)}\sum_\ell\frac{P(\ell/L)}{(2\pi i)^{2k+3}}\int_{(\substack{c_s,c_w,\bcu,\\ c_{u_*},c_{v_*}})} \sumstar_{h\mod \ell}\frac{d^{1-v_*-u_*+\nu_*}}{\ell^{1-\alpha_1+\beta_1+w}} F(v_*+u_*-\nu_*,\pm_*\tfrac{dh}\ell)q^{ks}\\
&\hspace{-0.em}\quad \times N_*^{u_*} \tilde P(u_*) \bigg(\prod_{i=2}^k\int_{(c_{v_i})}D_{\alpha_i,\beta_i}\pr{\tfrac12+u_i+v_i-\nu_i,\tfrac{\pm_i h}{\ell}}\tilde P(u_i-s)N_i^{u_i-s}\bigg)\\
&\hspace{-0.em}\quad\times N_1^{u_1-s}\tilde P(u_1-s) \Psi_{\epsilon^*,B}\pr{\tfrac12-\alpha_1-u_1+\tfrac w2,\bv,v_*} \frac{H_{\balpha,\bbeta}(w,s)}{ws}dsdw\bdu\bdv du_*dv_*,\\
}
where, after moving the lines of integration $c_{u_2},\dots,c_{u_k},c_{u_*}$, we have
\es{\label{fcoli}
&c_s:=\eps/k,\qquad c_{u_1} =-B-\tfrac12-\Re(\alpha_1)+7\eps, \qquad c_{w}=10\eps,\\
&c_{v_2}=\cdots=c_{v_k}=c_{v_*}=\eps/k,\qquad c_{u_*}=1+\nu_*+2\eps-\eps/k
}
and $c_{u_i}=\tfrac12+\nu_i+\eps/k$, for $i=2,\dots,k$.

\begin{remark}
Thanks to~\eqref{bfdsh} and to Lemma~\ref{bfp} in Section~\ref{amell}, the integrals in $\mathcal N_{\nu;\balpha,\bbeta}$ are all absolutely convergent when the line of integration are chosen so that $\Re(v_2)=\cdots=\Re(v_k)=\Re(v_*)=\eps/k$ and $\Re(v_1):=\Re(\frac12-\alpha_1-u_1+\frac w2)=B+1-2\eps$ (and even if an extra factor of $\prod_{i=2}^k(1+|u_i|+|v_i|)^{1+4\eps}$ is introduced inside the integrals, as will be relevant later on in the argument). In the following computations, until Remark~\ref{acep}, we will (almost) always arrange the lines of integration in a way such that $\Re(v_1)$ is kept equal to $B+1-2\eps$.\footnote{The only exception is in the proof of~\eqref{abfaw}, where we need to take $\Re(v_1)=B-2\eps$. One can easily verify however that the integrals are all absolutely convergent also in that case.} This ensures the absolute convergence of the integrals in all the bounds we give.
\end{remark}
\begin{remark}
We also observe that, by the definition~\eqref{dfgfs}, the poles of $\Psi_{\epsilon^*,B}\pr{v_1,\bv,v_*}$ are contained in the set 
\est{
\left\{(v_1,\bv,v_*)\in\C^{k+1}\mid v_i\in\Z_{\leq0}\tn{ for some $i\in\{1,\dots,k\}$ or }v_1+\dots+v_k+v_*= B+1\right\}.
}
\end{remark}

\subsection{Picking up the residues of the Estermann function}

For each $i=2,\dots,k$ we move the line of integration $c_{u_i}$ to $c_{u_i}=-\frac12+\nu_i-2\eps,$
passing through the poles of the Estermann function at $u_i=\frac12-\alpha_i-v_i+\nu_i$ and $u_i=\frac12-\beta_i-v_i+\nu_i$. By Lemma~\ref{properties of Estermann} and the residue theorem, we obtain
\es{\label{sdga}
\mathcal N_{\nu;\balpha,\bbeta}&=\sum_{\substack{I\cup J=\{2,\dots,k\},\\ I\cap J=\emptyset}}\sum_{\substack{\{\alpha_i^\prime,\beta_i^\prime\}=\{\alpha_i,\beta_i\}\ \forall i\in I,\\ (\alpha_j^\prime,\beta_j^\prime)=(\alpha_j,\beta_j)\ \forall j\in J\cup\{1\}}}\mathcal P_{I;\nu;\balpha^\prime,\bbeta^\prime},
}
where, for $I\cup J=\{2,\dots,k\}$, $I\cap J=\emptyset$,
\est{
\mathcal P_{I;\nu;\balpha,\bbeta}&:=\sum_{d|q}\frac{\mu(\frac qd)}{\varphi(q)}\sum_\ell \frac{P(\ell/L)}{(2\pi i)^{5}}\int_{(\substack{c_s,c_w,c_{u_1}\\ c_{u_*},c_{v_*}})} \frac{q^{ks}d^{1-v_*-u_*+\nu_*}}{\ell^{\sum_{i\in I\cup\{1\}}(1-\alpha_i+\beta_i)+w}} \sumstar_{h\mod \ell}F(v_*+u_*-\nu_*,\pm_*\tfrac{dh}\ell)\\
&\quad \times  \bigg(\prod_{j\in J}\frac1{(2\pi i)^2}\int_{(c_{u_j},c_{v_j})}D_{\alpha_j,\beta_j}\pr{\tfrac12+u_j+v_j-\nu_j,\tfrac{\pm_j h}{\ell}}\tilde P(u_j-s)N_j^{u_j-s}du_j\bigg)\\
&\quad\times \bigg(\prod_{i\in I}\frac1{2\pi i}\int_{(c_{v_i})}\tilde P(\tfrac 12-\alpha_i-v_i+\nu_i-s)N_i^{\frac 12-\alpha_i-v_i+\nu_i-s}\bigg)N_*^{u_*} \tilde P(u_*)N_1^{u_1-s}\tilde P(u_1-s)\\
&\quad\times  \Psi_{\epsilon^*,B}\pr{\tfrac12-\alpha_1-u_1+\tfrac w2,\bv,v_*} \frac{H^\prime_{I;\balpha,\bbeta}(w,s)}{ws}dsdwdu_1 \bdv du_*dv_{*}\\
}
and
\es{\label{hgfc2}
H^\prime_{I;\balpha,\bbeta}(w,s)&:=G_{\balpha,\bbeta}(s)G_{\balpha,\bbeta}(w)g_{\balpha,\bbeta}(s)\zeta(1+w-\alpha_1+\beta_1)\prod_{i\in I}\zeta(1-\alpha_i+\beta_i),
}
so that 
\es{\label{bfh2}
H^\prime_{I;\balpha,\bbeta}(w,s)\ll (A \log q)^{| I|} e^{-C_2(|\Im(s)|+|\Im(w)|}(1+|\Re(s)|+|\Re(w)|)^{A(|\Re(s)|+|\Re(w)|+k)}.
}
We remind also that the lines of integrations are given by~\eqref{fcoli} and
\es{\label{fcoli2}
c_{u_j}=-\frac12+\nu_j-2\eps,\ \forall j\in J.
}

\subsection{Applying the bounds on sums of Kloosterman sums}
In this section, we apply Lemma~\ref{DIbound} to give a bound for $\mathcal P_{I,\nu;\balpha,\bbeta}$ under certain conditions.
\begin{lemma}
Let  $I\subseteq\{2,\dots,k\}$ and let $J:=\{2,\dots,k\} \setminus I$. Then, if $|I|\leq |J|$ 
we have
 \es{\label{21w1}
 \mathcal P_{I,\nu;\balpha,\bbeta} &\ll q^{-1+A\eps}N_1^{A\eps} (q^\vartheta N_1^{\frac {k+1}2}+q^{\frac k2-\frac13+\frac\vartheta 3} N_1^{\frac12})(N_1\cdots N_k)^{-\frac12}L^{-\eps},
 }
whereas if $|I|>|J|$ and $\nu_j>0$ for some $j\in J$, then
 \es{\label{21w2}
 \mathcal P_{I,\nu;\balpha,\bbeta} &\ll  q^{-\frac56+\frac \vartheta3+A\eps} N_1^{-1+A\eps}(N_1\cdots N_k)^{\frac12}L^{-\eps}.\\
  }
\end{lemma}
\begin{proof}
First, we bound the sums over $h$ and $\ell$ by Lemma~\ref{DIbound} and we bound trivially the integrals which are all convergent by~\eqref{fnn},~\eqref{bfh2} and~\eqref{feg} when the lines of integrations are given by~\eqref{fcoli} and~\eqref{fcoli2}. Doing so, we obtain
\begin{align}
 \mathcal P_{I,\nu;\balpha,\bbeta}&\ll q^{-\frac56+\frac \vartheta3+A\eps}N_1^{-B-\frac12+A\eps}N_*^{1+\nu_*+\eps}\Big(\prod_{i\in I}N_{i}^{\frac12+\nu_i}\Big)\Big(\prod_{j\in J}N_{j}^{-\frac12+\nu_j}\Big)L^{|J|-|I|-\eps}\times\notag\\
 &\quad\times \int_{(\substack{c_s,c_w,c_{u_1},\\c_{u_*},c_{v_*}})}
 \bigg(\prod_{j\in J}\int_{(c_{u_j})}(1+|v_j|+|u_j|)^{1+4\eps}|\tilde P(u_j-s)||du_j|\bigg)\notag\\
&\quad\times \bigg(\prod_{i\in I}\int_{(c_{v_i})}|\tilde P(\tfrac 12-\alpha_i-v_i+\nu_i-s)|\bigg)|\tilde P(u_*)|
|\tilde P(u_1-s)|\notag\\
&\quad\times  \pmd{\Psi_{\epsilon^*,B}\pr{\tfrac12-\alpha_1-u_1+\tfrac w2,\bv,v_*} }\frac{|H^\prime_{I;\balpha,\bbeta}(w,s)|}{|ws|}|dsdwdu_1 \bdv du_*dv_{*}|\notag\\
&\ll q^{-\frac56+\frac \vartheta3+A\eps} N_1^{-B-\frac12+A\eps}N_*^{1+\nu_*}\Big(\prod_{i\in I}N_{i}^{\frac12+\nu_i}\Big)\Big(\prod_{j\in J}N_{j}^{-\frac12+\nu_j}\Big)L^{|J|-|I|-\eps}.\label{fasf}
\end{align}
If $|I|-|J|>0$ 
and at least one of the $\nu_j$ is greater than zero, then this is bounded by
\est{
 \mathcal P_{I,\nu;\balpha,\bbeta}
 &\ll q^{-\frac56+\frac \vartheta3+A\eps} N_1^{A\eps}\frac{N_*^{\nu_*+1}}{N_1^{\nu_*+1}}(N_1\cdots N_k)^{\frac12}\Big(\prod_{i\in I}\frac{N_{i}^{\nu_i}}{N_1^{\nu_i}}\Big)\Big(\prod_{j\in J}\frac{N_j^{\nu_j-1}}{N_1^{\nu_j}}\Big)L^{-\eps}\\
 &\ll q^{-\frac56+\frac \vartheta3+A\eps} N_1^{-1+A\eps}(N_1\cdots N_k)^{\frac12}L^{-\eps},\\
 }
since $B=\nu_1+\cdots+\nu_k$ and $N_2,\dots, N_k\leq N_1$, $N_*\ll kN_1q^{\eps/k}$. 

Now assume $|I|-|J|\leq 0$ and let $L\leq q^\frac12$. In this case~\eqref{fasf} gives
\begin{align}
  \mathcal P_{I,\nu;\balpha,\bbeta}
 &\ll q^{-\frac56+\frac \vartheta3+A\eps} N_1^{\frac12+|I|+A\eps}\frac{N_*^{\nu_*+1}}{N_1^{\nu_*+1}}(N_1\cdots N_k)^{-\frac12}\Big(\prod_{i\in I}\frac{N_{i}^{1+\nu_i}}{N_1^{1+\nu_i}}\Big)\Big(\prod_{j\in J}\frac{N_j^{\nu_j}}{N_1^{\nu_j}}\Big)q^{\frac 12(|J|-|I|)-\eps}\notag\\
 &\ll q^{-\frac56+\frac \vartheta3+\frac 12(|J|-|I|)+A\eps} N_1^{\frac12+|I|+A\eps}(N_1\cdots N_k)^{-\frac12}\notag\\
 &\ll q^{-\frac13+\frac\vartheta3+A\eps} (N_1^{\frac k2+A\eps}q^{-\frac12}+q^{\frac k2-1}N_1^\frac12)(N_1\cdots N_k)^{-\frac12},\label{dwcg}
\end{align}
since $|I|=k-1-|J|$ and $\frac {k-1}2\leq |J|\leq k-1$.

Finally, if $|I|-|J|\leq0$ 
and $L>q^\frac12$, then 
we move the lines of integration $c_{w}$ and $c_{u_1}$ to 
\est{
c_{w}=|J|-|I|+10\eps=k-1-2|I|+10\eps,\qquad
c_{u_1}
=-1-B+\tfrac k2-\Re(\alpha_1)-|I|+7\eps. 
} 
Then, we use Lemma~\ref{DIbound} and bound trivially the integrals (using~\eqref{fnn},~\eqref{bfh2} and~\eqref{feg}) and we obtain
\begin{align}
\mathcal P_{I,\nu;\balpha,\bbeta}&\ll q^{-1+\vartheta+A\eps} N_1^{-1-B+\frac k2-|I|+A\eps}N_*^{1+\nu_*}\Big(\prod_{i\in I}N_{i}^{\frac12+\nu_i}\Big)\Big(\prod_{j\in J}N_{j}^{-\frac12+\nu_j}\Big)L^{-\eps}\notag\\
 &\ll \frac{q^{-1+\vartheta+A\eps}}{L^\eps} N_1^{\frac k2+A\eps}\frac{N_*^{\nu_*+1}}{N_1^{\nu_*+1}}\Big(\prod_{i\in I}\frac{N_{i}}{N_1}\Big)\Big(\prod_{i=2}^k\frac{N_{i}^{-\frac12+\nu_i}}{N_1^{\nu_i}}\Big)\ll \frac{q^{-1+\vartheta+A\eps} N_1^{\frac k2+\frac12+A\eps}}{(N_1\cdots N_k)^{\frac12}L^{\eps}}.\label{frecaer}
\end{align}
Thus, since $ N_1^{\frac k2}q^{-\frac12}\ll q^{\frac k2-1} N_1^{\frac12}+q^{-1} N_1^{\frac k2+\frac12},$ we have that~\eqref{dwcg} and~\eqref{frecaer} imply~\eqref{21w1}.
\end{proof}
\subsection{Reassembling the sum over $\nu$ and further manipulations}

By the previous section, we only need to consider the $\mathcal P_{I;\nu;\balpha,\bbeta}$ with $|I|>|J|$ and $\nu_j=0$ for all $j\in J$ (and lines of integration given~\eqref{fcoli} and~\eqref{fcoli2}).
For each $j\in J$, we move $c_{u_{j}}$ to  $\frac12+\nu_j-2\eps$ and contextually $c_{v_{j}}$ to $c_{v_{j}}=-1+\eps/k$, passing through the pole of $\Psi_{\epsilon^*,B}$  at $v_j=0$. 
The contribution of the integral on the new line of integration can be bounded by
\es{\label{abfaw}
 &\ll q^{-\frac56+\frac \vartheta3+A\eps} N_1^{-1+A\eps}(N_1\cdots N_k)^{\frac12}L^{-\eps},
} 
as can be see by moving $c_{u_1}$ to $c_{u_1} =-B-\frac32-\alpha_1+7\eps$ and bounding the sums and integrals as in the proof of~\eqref{21w2}. Thus we only need to consider the residue at $v_{j}=0$ for all $j\in J$.

In the same way, we move the line of integration $c_{v_*}$ to $c_{v_*}=1+\eps/k$ and $c_{u_*}$ to $c_{u_*}=\nu_*+2\eps-\eps/k$, passing through the pole of $\Psi_{\epsilon^*,B}$ at $v_*=B+1-(\frac12-\alpha_1-u_1+\frac w2)-\sum_{i\in I}v_i$.  The contribution of the new line of integration can be bounded by~\eqref{abfaw} in a similar way,
so again we only need to consider the contribution of the residue. Thus, summarizing (and recalling~\eqref{ddrR} and~\eqref{sdga}), we arrive to
\es{\label{afml}
\mathcal L_{\balpha,\bbeta}&=\sum_{\substack{I\cup J=\{2,\dots,k\},\\ I\cap J=\emptyset,\   |I|>|J|}}\,\sum_{\substack{\{\alpha_i^\prime,\beta_i^\prime\}=\{\alpha_i,\beta_i\}\ \forall i\in I,\\ (\alpha_j^\prime,\beta_j^\prime)=(\alpha_j,\beta_j)\ \forall j\in J\cup\{1\}}}\sum_{\substack{\nu}}\frac{B!}{\nu_2!\cdots \nu_k!\nu_*!} \mathcal Q_{I;\nu;\balpha,\bbeta}\\
&\quad+O\pr{  \frac{q^{-1+A\eps} N_1^{A\eps}}{(N_1\cdots N_k)^{\frac12}} (q^\vartheta N_1^{\frac {k+1}2}+q^{\frac k2-\frac13+\frac\vartheta 3} N_1^{\frac12}+q^{\frac16+\frac \vartheta3} N_2\cdots N_k)L^{-\eps}},
}
where the sum over $\nu$ is now over $\nu=(\nu_2,\dots,\nu_k,\nu_*)\in\Z_{\geq0}^k$ satisfying
\est{
\nu_2+\cdots+\nu_k+\nu_*=B,\qquad\nu_{i}=0\text{ if $\pm_i1=-1$ or $i\in J$},\qquad\nu_{*}=0\text{ if $\pm_*1=-1$},
}
and where
\es{\label{sab}
\mathcal Q_{I;\nu;\balpha,\bbeta}&:=\sum_{d|q}\frac{\mu(\frac qd)}{\varphi(q)}\sum_\ell\sumstar_{h\mod \ell}\frac{P(\ell/L)}{(2\pi i)^{4+|I|}}\int_{(\substack{c_s,c_w,\\c_{u_1},c_{u_*}},\,c_{v_i}\forall {i\in I})}\frac{d^{1-(B+\frac12+\alpha_1+u_1-\frac w2)-u_*+\nu_*}}{\ell^{\sum_{i\in I\cup\{1\}}(1-\alpha_i+\beta_i)+w}} \\
&\quad \times q^{ks} F(B+\frac12+\alpha_1+u_1-\frac w2-\sum_{i\in I}v_i+u_*-\nu_*,\pm_*\tfrac{dh}\ell)\frac{H'_{I;\balpha,\bbeta}(w,s)}{ws}\\[-0.5em]
&\quad \times  \bigg(\prod_{j\in J}\frac1{2\pi i}\int_{(c_{u_j})}D_{\alpha_j,\beta_j}\pr{\tfrac12+u_j,\tfrac{\pm_j h}{\ell}}\tilde P(u_j-s)N_j^{u_j-s}du_j\bigg)\\
&\quad\times N_*^{u_*} \tilde P(u_*)N_1^{u_1-s}\tilde P(u_1-s) \Psi_{I;\epsilon^*_I,B}'\pr{\tfrac12-\alpha_1-u_1+\tfrac w2,\bv_I} \\
&\quad \times \bigg(\prod_{i\in I}\tilde P(\tfrac 12-\alpha_i-v_i+\nu_i-s)N_i^{\frac 12-\alpha_i-v_i+\nu_i-s}d^{v_i}dv_i\bigg)dwdu_1 dsdu_{*},\\
}
for $\bv_I:=(v_i)_{i\in I}$, $\epsilon_I^*=(\pm_i1)_{i\in I\cup\{*\}}$ and
\es{\label{gfddr}
\Psi_{I;\epsilon^*_I,B}'(v_1,\bv_I)&:=\frac{\Gamma(B+1-v_1-\sum_{i\in I}v_i)\prod_{i\in I\cup\{1\}}\Gamma(v_i)}{\Gamma(V_{\mp_*;\epsilon_I}(v_1,\bv_I))\Gamma(B+1-V_{\mp_*;\epsilon_I}(v_1,\bv_I))},\\
V_{\pm,\epsilon_I}(v_1,\bv_I)&:=\sum_{\substack{i\in I\cup\{1\},\\\pm_i1=\pm1}}v_{i}.
}
We also remind that the line of integrations are
\es{\label{ghw}
&c_s:=\eps/k,  \quad c_{w}=10\eps,\quad c_{u_*}=1+\nu_*+2\eps-\eps/k,\\ &c_{u_j}=-\tfrac12+\nu_j-2 \eps \  \forall j\in J,\quad c_{v_i}=\frac \eps k\  \forall i\in I
}
and $c_{u_1} =-B-\tfrac12-\Re(\alpha_1)+7\eps$.

\begin{remark}\label{acep}
Notice that the integrand in~\eqref{sab} decays rapidly along vertical strip in each of the variables of integration. In particular, in the following computations we will always be able to bound the integrals trivially.
\end{remark}

At this point, we wish to execute the sum over the partitions of unity $N_*$. However, first we need to remove the truncation $N_*\ll kN_1q^{\eps/k}$.  This can be done at a negligible cost, as it is shown in the following lemma.
\begin{lemma}
For $N_*\gg kN_1q^\eps$ we have
\es{\label{jha2}
\sumdagger_{N_*\ll kN_1q^{\eps/k}}\mathcal Q_{I;\nu;\balpha,\bbeta}=\sumdagger_{N_*}\mathcal Q_{I;\nu;\balpha,\bbeta}+O(q^{-2+A\eps} N_1^{A\eps}(N_1\cdots N_k)^{\frac12}L^{-\eps}).
}
\end{lemma}
\begin{proof}
Given a large positive integer $\Delta$, we move $c_{v_i}$ to $c_{v_i}=-\Delta k+\frac12$ for all $i\in I$. Doing so, we pass through the poles of $\Psi_{I;\epsilon^*_I,B}'\pr{\tfrac12-\alpha_1-u_1+\tfrac w2,\bv_I}$ at $v_i\in S_\Delta:=\{-r\mid r\in\Z,0\leq r<\Delta k\}$, so that we have to deal with a sum of $(\Delta k+1)^{|I|}$ terms coming from the contribution of the residues and of the integrals in the new line of integrations.\footnote{In total there are $(\Delta k+1)^{|I|}$ terms because for each $v_i$ we have the possibility of taking the residue at $v_i=-r_i\in S_\Delta$ or to take the integral at $c_{v_i}=-\Delta k+\frac12$.} Then, for each of these terms, we move the line of integration $c_{u_1}$ to $c_{u_1}=\Delta k$. This can be done without crossing any pole of $\Psi_{I;\epsilon^*_I,B}'$ if the term was coming from picking up a residue in each of the variables $v_i$ for all $i\in I$, since in this case the $\Gamma$ factor in the denominator of $\Psi_{I;\epsilon^*_I,B}'$ cancel the poles of $\Gamma(v_1)=\Gamma(\tfrac12-\alpha_1-u_1+\tfrac w2)$. Otherwise, we also have to consider the residues of $\Psi_{I;\epsilon^*_I,B}'$ at $\tfrac12-\alpha_1-u_1+\tfrac w2\in S_\Delta$. In all cases, however, all the terms will still have at least one integral left (besides the $w$ and $u_*$ integrals) with line of integration $c_{v_i}=-\Delta k+\frac12$ or $c_{u_1}=\Delta k$.
Finally, for each of these terms we place the line of integration $c_{u_*}$ so that the real part of the argument of the function $F$ in~\eqref{sab} is still $1+(4-|I|/k)\eps$ (we can do this without crossing poles). 

Thus, bounding these terms by using Lemma~\ref{DIbound} and the estimates~\eqref{fnn},~\eqref{bfh2} and~\eqref{feg}, we obtain
\est{
\mathcal Q_{I;\nu;\balpha,\bbeta}&\ll q^{-\frac56+\frac\vartheta3+A\eps} \sum_{\substack{ r_1,\dots r_k\in(\{0,1,\dots,\Delta-1\}\cup\{\Delta k-\frac12\}),\\ r_i=\Delta k-\frac12\tn{ for some $i\in(\{1\}\cup I$)},\\ r_i=0\tn{ if $i\in J$}}}\frac{N_1^{\frac12+r_1+A\eps}N_2^{r_2+\frac12+\nu_2}\cdots N_k^{r_k+\frac12+\nu_k}}{N_*^{B+\sum_{i=1}^kr_i-\nu_*+(4-|I|/k)\eps}{\displaystyle L^\eps}}\prod_{j\in J}N_j^{-1}\\
&\ll q^{-1+A\eps} \frac{N_1^{A\eps}(N_1\cdots N_k)^{\frac12}}{q^{\Delta\eps}N_*^{\eps}}\Big(\prod_{j\in J}N_j^{-1}\Big)L^{-\eps}\ll q^{-2+A\eps} N_1^{A\eps}(N_1\cdots N_k)^{\frac12}L^{-\eps},
}
if $\Delta$ is large enough with respect to $\eps$. Equation~\eqref{jha2} then follows.
\end{proof}

We now move the line of integration $c_{u_1}$ to $c_{u_1}=\tfrac12+3\eps$ and then execute the sum over $N_*$, which we do by using the following Lemma whose (simple) proof is left to the reader.
\begin{lemma}\label{sopu}
Let $K(s)$ be a function which is analytic and grows at most polynomially on a strip $|\Re(s)|<c$ for some $c>0$. Then, for any $-c<c_u<c$ we have
\est{
\sumdagger_N \frac1{2\pi i}\int_{(c)}K(u)\tilde P(u)N^{u} du=K(0).
}
\end{lemma}
We move the line of integration $c_{u_1}$ to $c_{u_1}=\tfrac12+3\eps$ without crossing poles 
and apply the above lemma obtaining 
\es{\label{afwae}
\sumdagger_{N_*}\mathcal Q_{I;\nu;\balpha,\bbeta}&=\sum_{d|q}\frac{\mu(\frac qd)}{\varphi(q)}\sum_\ell\sumstar_{h\mod \ell}\frac{P(\ell/L)}{(2\pi i)^{3+|I|}}\int_{(\substack{c_s,c_w,\\ c_{u_1},}\,c_{v_i}\forall {i\in I})} \frac{q^{ks}d^{1-(B+\frac12+\alpha_1+u_1-\frac w2)+\nu_*}}{\ell^{\sum_{i\in I\cup\{1\}}(1-\alpha_i+\beta_i)+w}} \\
&\quad \times F(B+\tfrac12+\alpha_1+u_1-\frac w2-\sum_{i\in I}v_i-\nu_*,\pm_*\tfrac{h}\ell)\\[-0.6em]
&\quad \times  \bigg(\prod_{j\in J}\frac1{2\pi i}\int_{(c_{u_j})}D_{\alpha_j,\beta_j}\pr{\tfrac12+u_j,\tfrac{\pm_j h}{\ell}}\tilde P(u_j-s)N_j^{u_j-s}du_j\bigg)\\
&\quad\times N_1^{u_1-s}\tilde P(u_1-s) \Psi_{I;\epsilon^*_I,B}'\pr{\tfrac12-\alpha_1-u_1+\tfrac w2,\bv_I} \frac{H'_{I;\balpha,\bbeta}(w,s)}{ws}\\
&\quad \times\prod_{i\in I} \bigg(\tilde P(\tfrac 12-\alpha_i-v_i+\nu_i-s)N_i^{\frac 12-\alpha_i-v_i+\nu_i-s}d^{v_i}\df v_i\bigg)\df w\df u_1 \df s,
}
with lines of integrations that we can take to be given by~\eqref{ghw} and $c_{u_1}=\tfrac12+3\eps$.

We are finally ready to execute the sum over $\nu$. We do this in the following Lemma, which also summarizes the previous computations.
\begin{lemma}
We have
\es{\label{tya}
\sum_{\epsilon\in\{\pm1\}^k}\hspace{-0.2em}\rho_{\Upsilon}(\epsilon)\mathcal O''_{\epsilon,\alpha,\beta}
&=2\sumdagger_L\hspace{-0.5em}\sum_{\substack{I\cup J=\{2,\dots,k\},\\  I\cap J=\emptyset,\   |I|>|J|}}\,\sum_{\substack{\{\alpha_i^\prime,\beta_i^\prime\}=\{\alpha_i,\beta_i\}\, \forall i\in I_1,\\ (\alpha_j^\prime,\beta_j^\prime)=(\alpha_j,\beta_j)\, \forall j\in J}}\ 
\sum_{\substack{\epsilon\in\{\pm1\}^{k},\\ \pm_11=-1}}\hspace{-0.15em}\rho_{\Upsilon}(\epsilon)
\hspace{-0.45em}\sum_{\pm_*1\in\{\pm1\}} \mathcal  R_{I;\epsilon^*;\balpha,\bbeta}\\
&\quad+O\pr{  \frac{q^{-1+A\eps} N_1^{A\eps}}{(N_1\cdots N_k)^{\frac12}{\displaystyle L^{\eps}}} (q^\vartheta N_1^{\frac {k+1}2}+q^{\frac k2-\frac13+\frac\vartheta 3} N_1^{\frac12}+q^{\frac16+\frac \vartheta3} N_2\cdots N_k)},
}
where $I_1:=I\cup\{1\}$ and
\est{
\mathcal R_{I;\epsilon^*;\balpha,\bbeta}&:=\sum_{d|q}\frac{\mu(\frac qd)}{\varphi(q)}\sum_\ell\sumstar_{h\mod \ell}\frac{P(\ell/L)}{(2\pi i)^{3+|I|}}\int_{(\substack{c_s,c_w,\\c_{u_1},}\,c_{v_i}\forall {i\in I})} \frac{q^{ks}d^{\frac12-\alpha_1-u_1+\frac w2}N_1^{u_1-s} }{\ell^{\sum_{i\in I_1}(1-\alpha_i+\beta_i)+w}} \\
&\quad \times  \bigg(\prod_{j\in J}\frac1{2\pi i}\int_{(c_{u_j})}D_{\alpha_j,\beta_j}\pr{\tfrac12+u_j,\tfrac{\pm_j h}{\ell}}\tilde P(u_j-s)N_j^{u_j-s}du_j\bigg)\\
&\quad\tilde P(u_1-s)\times F(\tfrac12+\alpha_1+u_1-\frac w2-\sum_{i\in I}v_i,\pm_*\tfrac{dh}\ell) \Psi_{I;\epsilon^*_I,0}'\pr{\tfrac12-\alpha_1-u_1+\tfrac w2,\bv_I} \\
&\quad \times \bigg(\prod_{i\in I}\tilde P(\tfrac 12-\alpha_i-v_i-s)N_i^{\frac 12-\alpha_i-v_i-s}d^{v_i}dv_i\bigg)\frac{H'_{I;\balpha,\bbeta}(w,s)}{ws}dwdu_1 ds\\
}
and lines of integrations given by~\eqref{ghw} and $c_{u_1}=\tfrac12+3\eps$.
\end{lemma} 
\begin{proof}
Using~\eqref{iml},~\eqref{hgfd},~\eqref{afml} and~\eqref{jha2} we obtain~\eqref{tya}, with $\mathcal R_{I;\epsilon^*;\balpha,\bbeta}$ replaced by
\est{
\mathcal R'_{I;\epsilon^*;\balpha,\bbeta}&:=\sum_{\nu}\frac{B!}{\nu_2!\cdots \nu_k!\nu_*!}\sumdagger_{N_*}\mathcal Q_{I;\nu;\balpha,\bbeta}\\
}
and $\sumdagger\mathcal Q_{I;\nu;\balpha,\bbeta}$ as in~\eqref{afwae}. Thus, the Lemma reduces to showing that $\mathcal R_{I;\epsilon^*;\balpha,\bbeta}=\mathcal R'_{I;\epsilon^*;\balpha,\bbeta}$. This is an immediate consequence of Lemma~\ref{rcmpv} below, which is applicable since the pole of $F$ is canceled by the sum over $d$.
\end{proof}

\subsection{Reassembling the sum over $\epsilon$}

Now, we can also get rid of the integral over $w$. To do this, first we move the lines of integration $c_{u_1}$ and $c_{u_j}$ for $j\in J$ (without passing through poles), so that we have
\es{\label{uav}
c_{u_1}=-\Re(\alpha_1)+7\eps, \quad c_{u_j}=\frac12-\Re(\alpha_j)-2 \eps  \ \forall j\in J \quad c_s:=\eps/k,\quad  c_{v_i}=\frac \eps k\  \forall i\in I\
}
and then we move $c_{w}$ to $c_{w}=-1+10\eps$ passing through a pole at $w=0$. The contribution of the new line of integration is trivially bounded by
\es{\label{ypa}
\ll q^{-1+A\eps}N_1^{-\frac12+A\eps} (N_1\cdots N_k)^{\frac12}L^{-\eps}.
}
since we have the convexity bound $D(1-2\eps+it,\alpha_j-\beta_j,\frac h\ell)\ll \ell^{3\eps}(1+|t|)^{3\eps}$ and $|I|\geq 2$ (since $|I|>|J|$ and $k\geq3$). Thus, we only need to consider the contribution from the residue at $w=0$. 

By the convexity bound $F(\frac12+7\eps-|I|\eps/k+it,\frac h\ell)\ll \ell^\frac12(1+|t|)^\frac12$, the contribution of the $d=1$ term is also bounded by~\eqref{ypa}. Thus, using also that $\varphi(q)^{-1}=q^{-1}+O(q^{-2})$ for $q$ prime, we have
\es{\label{tya2}
\mathcal R_{I;\epsilon^*;\balpha,\bbeta}=\mathcal S_{I;\epsilon^*;\balpha,\bbeta}+O\pr{q^{-1+A\eps}N_1^{-\frac12+A\eps} (N_1\cdots N_k)^{\frac12}L^{-\eps}}
}
with 
\est{
\mathcal S_{I;\epsilon^*;\balpha,\bbeta}&=\sum_\ell\sumstar_{h\mod \ell}\frac{P(\ell/L)}{(2\pi i)^{2+|I|}}\int_{(\substack{c_s,\\c_{u_1},}\,c_{v_i}\forall {i\in I})} \frac{q^{ks-\frac12-\alpha_1-u_1}}{\ell^{\sum_{i\in I\cup\{1\}}(1-\alpha_i+\beta_i)}}N_1^{u_1-s}\tilde P(u_1-s)  \\
&\quad \times  \bigg(\prod_{j\in J}\frac1{2\pi i}\int_{(c_{u_j})}D_{\alpha_j,\beta_j}\pr{\tfrac12+u_j,\pm_*\tfrac{\pm_j h}{\ell}}\tilde P(u_j-s)N_j^{u_j-s}du_j\bigg)\\
&\quad\times F(\tfrac12+\alpha_1+u_1-\sum_{i\in I}v_i,\tfrac{qh}\ell) \Psi_{I;\epsilon^*_I,0}'\pr{\tfrac12-\alpha_1-u_1,\bv_I} \\
&\quad \times \bigg(\prod_{i\in I}\tilde P(\tfrac 12-\alpha_i-v_i-s)N_i^{\frac 12-\alpha_i-v_i-s}q^{v_i}dv_i\bigg)\frac{H'_{I;\balpha,\bbeta}(0,s)}{s}du_1 ds\\
}
and lines of integration given by~\eqref{uav}. Notice that we made the change of variable $h\rightarrow \pm_*h$.

We are ready to reassemble the sum over $\epsilon$. To do this, first we split $\epsilon$ into $\epsilon_{I_1}$ and $\epsilon_{J}$, where $\epsilon_S:=(\pm_i1)_{i\in S}$; in particular, $\rho_\Upsilon(\epsilon)=\rho_\Upsilon(\epsilon_{I_1})\rho_\Upsilon(\epsilon_J)$ where, with a slight abuse of notation, we write $\rho_{\Upsilon}(\epsilon_S):=\prod_{i\in\Upsilon\cap S} (\pm_i1)$. Then, we observe that 
\est{
\sum_{\epsilon_J\in\{\pm1\}^{|J|}}\rho_{\Upsilon}(\epsilon_J)\prod_{j\in J}D_{\alpha_j,\beta_j}\pr{s_j,\pm_*\tfrac{\pm_j h}{\ell}}=2^{|J|}(\pm_*i)^{|\Upsilon\cap J|}\prod_{j\in J}D_{j;\alpha_j,\beta_j}\pr{s_j,\tfrac{ h}{\ell}}.
}
Thus, 
\es{\label{asdcae}
&\sum_{\substack{\epsilon\in\{\pm1\}^{k},\\ \pm_11=-1}}\rho_{\Upsilon}(\epsilon)\sum_{\pm_*1\in\{\pm1\}} \mathcal S_{I;\epsilon^*;\balpha,\bbeta}=2^{|J|}i^{|\Upsilon\cap J|} \sum_\ell\sumstar_{h\mod \ell}\frac{P(\ell/L)}{(2\pi i)^{2+|I|}}\frac{q^{ks-\frac12-\alpha_1-u_1}}{\ell^{\sum_{i\in I\cup\{1\}}(1-\alpha_i+\beta_i)}}  \\
&\hspace{1.5em}\quad \times \int_{(\substack{c_s,\\c_{u_1},}\,c_{v_i}\forall i \in I)}  \bigg(\prod_{j\in J}\frac1{2\pi i}\int_{(c_{u_j})}D_{j;\alpha_j,\beta_j}\pr{\tfrac12+u_j,\tfrac{ h}{\ell}}\tilde P(u_j-s)N_j^{u_j-s}du_j\bigg)\\
&\hspace{1.5em}\quad\times N_1^{u_1-s}\tilde P(u_1-s)F(\tfrac12+\alpha_1+u_1-\sum_{i\in I}v_i,\tfrac{qh}\ell) \mathcal{X}_I(\tfrac12-u_1-\alpha_1,\bv_I) \\
&\hspace{1.5em}\quad \times \bigg(\prod_{i\in I}\tilde P(\tfrac 12-\alpha_i-v_i-s)N_i^{\frac 12-\alpha_i-v_i-s}q^{v_i}dv_i\bigg)\frac{H'_{I;\balpha,\bbeta}(0,s)}{s}du_1 ds,\\
}
with
\est{
\mathcal{X}_I(v_1,\bv_I)&:=\sum_{\pm_*1\in\{\pm1\}}(\pm_*1)^{|\Upsilon\cap J|}\sum_{\substack{\epsilon_{I_1}\in\{\pm1\}^{|I_1|},\\\pm_11=-1}}\rho_{\Upsilon}(\epsilon_{I_1})  \Psi_{I;\epsilon^*_I,0}'\pr{v_1,\bv_I} ,\\
&=\Gamma\Big(1-\sum_{i\in I_1}v_i\Big)\prod_{i\in I_1}\Gamma(v_i)\sum_{\pm_*1\in\{\pm1\}}\hspace{-0.3em}(\pm_*1)^{|\Upsilon\cap I_1|}\\
&\quad\times\sum_{\substack{\epsilon_{I_1}\in\{\pm1\}^{|I_1|},\\\pm_11=-1}}\rho_{\Upsilon}(\epsilon_{I_1}) \prbigg{\Gamma\Big(\sum_{\substack{i\in I_1,\\\pm_i1=\mp_*1}}v_i\Big)\Gamma\Big(1-\sum_{\substack{i\in I_1,\\\pm_i1=\mp_*1}}v_i\Big) }^{-1},
}
by the definition~\eqref{gfddr} of $\Psi_{I;\epsilon^*_I,0}'$ and since  $(\pm_*1)^{|\Upsilon\cap J|}=(\pm_*1)^{|\Upsilon\cap( I\cup\{1\})|}$ for $|\Upsilon|$ even (as we have assumed).  Also, we remind that we defined $\epsilon_{I}^*:=(\pm_i1)_{i\in I\cup\{*\}}$ and  $I_1:=I\cup\{1\}$.

We will now give a $\Gamma$-function identity, which we will use to give a symmetric expression for $\mathcal{X}_I(v_1,\bv_I)$.

\begin{lemma}\label{gfar}
Let $r\geq1$, $\Theta\subseteq\{1,\dots,r\}$ and $(s_1,\dots,s_r)\in\C^{r}$. For $\epsilon_{r}=(\pm_1,\dots,\pm_r1)\in\{\pm1\}^{r}$ let $\rho_{\Theta}(\epsilon):=\prod_{i\in\Theta} (\pm_i1)$.
Then,\footnote{The identity has to be interpreted as an identity between meromorphic functions.}
\est{
&
\prod_{i=1}^r\Gamma(s_i)\sum_{\pm_*1\in\{\pm1\}}\hspace{-0.3em}(\pm_*1)^{| \Theta|} \hspace{-0.4em}\sum_{\substack{\epsilon\in\{\pm1\}^{r},\\ \pm_11=-1}}\rho_{\Theta}(\epsilon)\prbigg{\Gamma\Big(\sum_{\substack{\pm_i1=\mp_*1}}s_i\Big)\Gamma\Big(1-\sum_{\substack{\pm_i1=\mp_*1}}s_i\Big) }^{-1}
\\
&\hspace{1em}=\frac{2^{s_1+\cdots+s_r}}{\pi^{1-\frac r2}}\bigg(\prod_{\substack{ i\in\Theta}}\frac{\Gamma(\frac12+\frac {s_i}2)}{\Gamma(1-\frac {s_i}2)} \bigg)\bigg(\prod_{\substack{i\notin\Theta}}\frac{\Gamma(\frac {s_i}2)}{\Gamma(\frac12-\frac {s_i}2)}\bigg)\sin\prBig{\frac{\pi }2(s_1+\cdots+s_r)-\frac{\pi }{2}|\Theta|}.\\
}
\end{lemma}
\begin{proof}
First, we observe that, by analytic continuation, we can assume that $s_1,\dots,s_r\in\R\setminus\Z$. Thus, using the reflection formula for the Gamma function to have 
\est{
\prbigg{\Gamma\Big(\sum_{\substack{\pm_i1=\mp_*1}}s_i\Big)\Gamma\Big(1-\sum_{\substack{\pm_i1=\mp_*1}}s_i\Big) }^{-1}=\pi^{-1}\sin\Big(\pi\sum_{\substack{\pm_i1=\mp_*1}}s_i\Big)=\pi^{-1}\Im\bigg(\exp\Big(\pi i\sum_{\substack{\pm_i1=\mp_*1}}s_i\Big)\bigg).
}
It follows that
\est{
S&:=\sum_{\pm_*1\in\{\pm1\}}\hspace{-0.3em}\mathcal (\pm_*1)^{| \Theta|}  \hspace{-0.4em}\sum_{\substack{\epsilon\in\{\pm1\}^{r},\\ \pm_11=-1}}\rho_{\Theta}(\epsilon)\prbigg{\Gamma\Big(\sum_{\substack{\pm_i1=\mp_*1}}s_i\Big)\Gamma\Big(1-\sum_{\substack{\pm_i1=\mp_*1}}s_i\Big) }^{-1}\\
&=\pi^{-1}\Im \prBig{e^{\pi i s_1}A_{+}
+(-1)^{| \Theta|}  A_-},\\
}
where
\est{
A_{\pm}:=\sum_{\substack{\epsilon\in\{\pm1\}^{r},\\ \pm_11=-1}}\rho_{\Theta}(\epsilon)\exp\Big(\pi i\sum^r_{\substack{i=2,\\\pm_i1=\mp1}}s_i\Big).
}
Now, since $\rho_{\Theta}(\epsilon)=\prod_{i\in\Theta} (\pm_i1)=(-1)^{|\Theta\cap\{1\}|}\prod_{i\in\Theta\setminus\{1\}} (\pm_i1)$, we have
\est{
A_{\pm}&=(-1)^{|\Theta\cap\{1\}|}\bigg(\prod_{\substack{i=2,\\ i\in\Theta}}^r(\pm1\mp e^{\pi is_i})\bigg)\bigg(\prod_{\substack{i=2,\\ i\notin\Theta}}^r(1+e^{\pi is_i})\bigg)\\
&=(\pm 1)^{|\Theta\cap\{1\}|}(\mp1)^{|\Theta|}i^{|\Theta\setminus\{1\}|}2^{r-1}\exp\prbigg{\frac{\pi i}2\sum_{i=2}^rs_i}\bigg(\prod_{\substack{i=2,\\ i\in\Theta}}^r\sin\pr{\frac{\pi s_i}2}\bigg)\bigg(\prod_{\substack{i=2,\\ i\notin\Theta}}^r\cos\pr{\frac{\pi s_i}2}\bigg)
}
and thus
\est{
S&=\frac{2^{r-1}}\pi\bigg(\prod_{\substack{ i\in\Theta,\\ i\neq1}}\sin\pr{\frac{\pi s_i}2}\bigg)\bigg(\prod_{\substack{i\notin\Theta,\\ i\neq1}}\cos\pr{\frac{\pi s_i}2}\bigg)(-1)^{| \Theta|}\Im\pr{ \pr{ e^{\pi is_1}+  (-1)^{|\Theta\cap\{1\}|}}e^{\frac{\pi i}{2}|\Theta\setminus\{1\}|+\frac{\pi i}2\sum_{i=2}^rs_i}}\\
&=\frac{2^{r}}\pi\bigg(\prod_{\substack{ i\in\Theta}}\sin\pr{\frac{\pi s_i}2}\bigg)\bigg(\prod_{\substack{i\notin\Theta}}\cos\pr{\frac{\pi s_i}2}\bigg)(-1)^{| \Theta|}\Im\pr{i^{|\Theta\cap\{1\}|}e^{\frac{\pi i}{2}|\Theta\setminus\{1\}|+\frac{\pi i}2\sum_{i=1}^rs_i}}\\
&=\frac{2^{r}}\pi\bigg(\prod_{\substack{ i\in\Theta}}\sin\pr{\frac{\pi s_i}2}\bigg)\bigg(\prod_{\substack{i\notin\Theta}}\cos\pr{\frac{\pi s_i}2}\bigg)\sin\prBig{-\frac{\pi }{2}|\Theta|+\frac{\pi }2\sum_{i=1}^rs_i}.\\
}
By the duplication and the reflection formula for the $\Gamma$-function we have
\est{
\sin\pr{\frac{\pi s}2}\Gamma(s)=\pi^\frac122^{s-1}\frac{\Gamma(\frac12+\frac s2)}{\Gamma(1-\frac s2)},\qquad
\cos\pr{\frac{\pi s}2}\Gamma(s)=\pi^\frac122^{s-1}\frac{\Gamma(\frac s2)}{\Gamma(\frac12-\frac s2)}
}
and thus the Lemma follows.
\end{proof}
Applying Lemma~\ref{gfar} with $r=|I_1|$ and using the definition~\eqref{dfgi} of $\Gamma_i$, we obtain
\est{\mathcal{X}_I(v_1,\bv_I)
&=
\Gamma\Big(1-\sum_{i\in I_1}v_i\Big)
\frac{2^{\sum_{i\in I_1}v_i}}{\pi^{1-\frac {|I_1|}2}}\bigg(\prod_{\substack{ i\in I_1}}\frac{\Gamma_i(\frac {v_i}2)}{\Gamma_i(\frac12-\frac {v_i}2)}\bigg)\sin\prBig{-\frac{\pi }{2}|I_1\cap \Upsilon|+\frac{\pi }2\sum_{i\in I_1}v_i}.\\
}
Thus, plugging this expression into~\eqref{asdcae}, making the change of variables $ u_i=\frac12-\alpha_i-v_i-s$ for all $i\in I$ and $u_j\rightarrow u_j+s$ for all $j\in(J\cup\{1\})$, and moving slightly the lines of integrations, we obtain 
\es{\label{tya4}
\sum_{\substack{\epsilon\in\{\pm1\}^{k},\\ \pm_11=-1}}\rho_{\Upsilon}(\epsilon)\sum_{\pm_*1\in\{\pm1\}} \mathcal S_{I;\epsilon^*;\balpha,\bbeta}&=\mathcal U_{I_1;\balpha,\bbeta}
}
where for $\mathcal I\subseteq \{1,\dots,k\}$ (with $|\mathcal I|\geq2$), $\mathcal J:=\{1,\dots,k\}\setminus\mathcal I$, we define
\es{\label{wefcw}
\mathcal U_{\mathcal I;\balpha,\bbeta}&:=-2^{|\mathcal  J|}i^{|\Upsilon\cap\mathcal  J|}\sum_\ell\sumstar_{h\mod \ell}\frac{P(\ell/L)}{(2\pi i)^{1+k}}\int_{(c_s,\bcu)} \frac{q^{ks-1}}{\pi\ell^{\sum_{i\in \mathcal I}(1-\alpha_i+\beta_i)}} \frac{H''_{\mathcal I;\balpha,\bbeta}(s)}{s} \\
&\quad \times  \bigg(\prod_{j\in \mathcal J}D_{j;\alpha_j,\beta_j}\pr{\tfrac12+u_j+s,\tfrac{ h}{\ell}}\tilde P(u_j)N_j^{u_j}\bigg)\Gamma\Big(1+|\mathcal I|(s-\tfrac12)+\sum_{i\in \mathcal I}(u_i+\alpha_i)\Big)\\
&\quad\times\sin\prBig{\frac{\pi}{2}|\Upsilon\cap\mathcal I|+\frac{\pi }{2}|\mathcal I|(s-\tfrac12)+\frac{\pi }2\sum_{i\in \mathcal I}(u_i+\alpha_i)} \\
&\quad\times F(1+|\mathcal I|(s-\tfrac12)+\sum_{i\in \mathcal I}(u_i+\alpha_i),\tfrac{qh}\ell) \bigg(\prod_{i\in \mathcal I}\frac{\pi^\frac12\tilde P(u_i)N_i^{u_i}}{(2q)^{u_i+\alpha_i+s-\frac12}}\frac{\Gamma_i(\frac14-\frac {u_i+\alpha_i+s}2)}{\Gamma_i(\frac14+\frac {u_i+\alpha_i+s}2)}\bigg) ds\bdu,\\
}
with lines of integration
\es{\label{addwe}
 c_{u_i}=\frac12-\Re(\alpha_i)-3\frac \eps k-\eps \ \forall i\in \{1,\dots,k\} \quad c_s:=\eps/k,
}
and, recalling~\eqref{hgfc2},
\es{\label{raf}
H''_{\mathcal {I};\balpha,\bbeta}(s):=G_{\balpha,\bbeta}(s)g_{\balpha,\bbeta}(s)\prod_{i\in \mathcal I}\zeta(1-\alpha_i+\beta_i).\\
}
For future use we remark that if we move $c_{u_i'}$ to $c_{u_i'}=-\frac12-\Re(\alpha_1)-5\eps$ for some $i'\in \mathcal I$ we
 get
\es{\label{fms2}
\mathcal U_{\mathcal I;\balpha,\bbeta} \ll N_{i'}^{-1+A\eps}q^{A\eps}(N_1\cdots N_k)^{\frac12}L^{-\eps}.
}
Also, if $|\mathcal I|>|\mathcal J|$, then moving the line of integration to $c_{u_j}$ to $c_{u_j}= -\frac12+5\frac{\eps}{k}-\eps$ for all $j\in\mathcal J$ (leaving the other lines of integration as in~\eqref{addwe}), we obtain
\es{\label{fms}
\mathcal U_{\mathcal I;\balpha,\bbeta} \ll q^{-1+A\eps}(N_1\cdots N_k)^{\frac12}L^{-\eps}\prod_{j\in \mathcal J}N_j^{-1+A\eps}.
}
Finally, moving $c_s$ to $c_s=\frac12+B-3\frac \eps k$ and $c_{u_i}$ to $c_{u_i}=-B$ for all $i=1,\dots, k$ we obtain
\es{\label{tb3c}
\mathcal U_{\mathcal I;\balpha,\bbeta} \ll_B (N_1\cdots N_k/q^k)^{-B}q^{\frac k2-1+A\eps},
}
if $|\mathcal I|\geq2$. 
 
We are now ready to complete the proof of Lemma~\ref{ctd}. Using~\eqref{tya},~\eqref{tya2} and~\eqref{tya4} we obtain
\est{
\sum_{\epsilon\in\{\pm1\}^k}\hspace{-0.2em}\rho_{\Upsilon}(\epsilon)\mathcal O''_{\epsilon,\alpha,\beta}
&=\sumdagger_L\sum_{\substack{I\cup J=\{2,\dots,k\},\\  I\cap J=\emptyset,\  |I|>|J|}}\,\sum_{\substack{\{\alpha_i^\prime,\beta_i^\prime\}=\{\alpha_i,\beta_i\}\ \forall i\in I_1,\\ (\alpha_j^\prime,\beta_j^\prime)=(\alpha_j,\beta_j)\ \forall j\in J}}2\,\mathcal U_{I_1;\balpha,\bbeta}\\
&\quad+O\pr{  \frac{N_1^{A\eps}}{q^{1-A\eps} } \pr{\frac{q^\vartheta N_1^{\frac {k}2}+q^{\frac k2-\frac13+\frac\vartheta 3} }{(N_2\cdots N_k)^{\frac12}}+(q^{\frac16+\frac \vartheta3} N_1^{-\frac12}+1)(N_2\cdots N_k)^\frac12}}\\
&=\mathcal M_{\alpha,\beta}+O\pr{  \frac{N_1^{A\eps}}{q^{1-A\eps} } \pr{\frac{q^\vartheta N_1^{\frac {k}2}+q^{\frac k2-\frac13+\frac\vartheta 3} }{(N_2\cdots N_k)^{\frac12}}+(q^{\frac16+\frac \vartheta3} N_1^{-\frac12}+1)(N_2\cdots N_k)^\frac12}}\\
}
where
\es{\label{dfmm}
\mathcal M_{\alpha,\beta}(N_1,\cdots,N_k):=\sumdagger_L\sum_{\substack{\mathcal I\cup \mathcal J=\{1,\dots,k\},\\  \mathcal I\cap\mathcal J=\emptyset,\  |\mathcal I|>|J|+1}}\,\sum_{\substack{\{\alpha_i^\prime,\beta_i^\prime\}=\{\alpha_i,\beta_i\}\ \forall i\in\mathcal I,\\ (\alpha_j^\prime,\beta_j^\prime)=(\alpha_j,\beta_j)\ \forall j\in\mathcal J}}2\,\mathcal U_{\mathcal I;\balpha,\bbeta}.
}
and $\mathcal U_{\mathcal I;\balpha,\bbeta}$ as defined in~\eqref{wefcw}. Noticed that in the last step we used~\eqref{fms} to extend the sum over the subsets of $\{1\,\dots,k\}$ to include also the sets $\mathcal I$ that do not contain $1$. Moreover, by~\eqref{fms2} and~\eqref{fms} we also have
\est{
\mathcal M_{\alpha,\beta}(N_1,\cdots,N_k)\ll q^{A\eps}(N_1\cdots N_k)^{\frac12}N_1^{-1+A\eps}.
}
and thus the proof of Lemma~\ref{ctd} is complete. Also, by~\eqref{tb3c} for any $B>0$ we have 
\es{\label{tb3c2}
\mathcal M_{\alpha,\beta}(N_1,\cdots,N_k)\ll_B q^{\frac k2-1+A\eps}(N_1\cdots N_k/q^k)^{-B}
}

We remark that we reached a formula for $\mathcal M_{\alpha,\beta}$ which is completely symmetric in the $N_1,\dots,N_k$. This is important in order to remove the assumption that $N_1$ is the largest of $N_1,\dots, N_k$, so that we can sum over the partitions of unity. 

\section{Assembling the main terms}\label{mmtt}
In this section we prove Lemma~\ref{smtmt}.

We start by moving $c_{u_i}$ to $c_{u_i}=0$ for all $i\in \mathcal I$ and $c_{s}$ to $\frac12-3\frac\eps{k}$ (we can do this without passing through any pole nor having problem of convergence). Then, after extending the sum over the partitions of unity $L,N_1,\dots,N_k$ using~\eqref{tb3c2} and summing over them using Lemma~\ref{sopu} we obtain
\est{
 \sumdagger_{\substack{N_1,\dots,N_k,\\ N_1\cdots N_k\ll q^{k+\eps}}}\mathcal M_{\alpha,\beta}(N_1,\cdots,N_k)=\sum_{\substack{\mathcal I\cup \mathcal J=\{1,\dots,k\},\\  \mathcal I\cap\mathcal J=\emptyset,\  |\mathcal I|>|\mathcal J|+1}}\,\sum_{\substack{\{\alpha_i^\prime,\beta_i^\prime\}=\{\alpha_i,\beta_i\}\ \forall i\in\mathcal I,\\ (\alpha_j^\prime,\beta_j^\prime)=(\alpha_j,\beta_j)\ \forall j\in\mathcal J}}2\,\mathscr{V}_{\mathcal I;\balpha,\bbeta}+O(1)
}
with
\es{\label{dfys}
\mathscr V_{\mathcal I;\balpha,\bbeta}&:=-2^{|\mathcal J|}i^{|\Upsilon\cap \mathcal J|}\frac{1}{2\pi i}\int_{(c_s)}\sum_\ell\sumstar_{h\mod \ell} \frac{q^{\frac12|\mathcal I|+|\mathcal J|s-1}}{\ell^{\sum_{i\in \mathcal I}(1-\alpha_i+\beta_i)}} \Gamma\Big(1+|\mathcal I|(s-\tfrac12)+\sum_{i\in \mathcal I}\alpha_i\Big) \\
&\quad \times  \bigg(\prod_{j\in \mathcal J}D_{j;\alpha_j,\beta_j}\pr{\tfrac12+s,\tfrac{ h}{\ell}}\bigg)\bigg(\prod_{i\in \mathcal I}\frac{(2\pi)^\frac12}{2^{\alpha_i+s}q^{\alpha_i}}\frac{\Gamma_i(\frac14-\frac {\alpha_i+s}2)}{\Gamma_i(\frac14+\frac {\alpha_i+s}2)}\bigg)\frac{H''_{\mathcal I;\balpha,\bbeta}(s)}{\pi s} \\
&\quad \times
 F(1+|\mathcal I|(s-\tfrac12)+\sum_{i\in \mathcal I}\alpha_i,\tfrac{h}\ell)
 \sin\prBig{\frac{\pi}{2}|\Upsilon\cap\mathcal I|+\frac{\pi }{2}|\mathcal I|(s-\tfrac12)+\frac{\pi }2\sum_{i\in \mathcal I}\alpha_i}  ds,\\
}
and line of integration $c_s=\frac12-3\frac\eps{k}$. 

Now, for each integral we move the line of integration $c_s$ to 
\est{ 
c_s=\max\pr{0,-\frac12+\frac{|\mathcal J|+\frac32}{|\mathcal I|+|\mathcal J|-\frac12}}+9\frac \eps k
=\begin{cases}
\frac3{4k-2}+9\frac \eps k& \text{if $|\mathcal I|=|\mathcal J|+2=\frac k2+1$ with $k$ even,}\\
9\frac \eps k& \text{if $|\mathcal I|>|\mathcal J|+2$.}
\end{cases}
}
picking up the residue of the pole of the $\Gamma$-function at 
\est{
s'=s'(\balpha)=\frac12-\frac{1+\sum_{i\in \mathcal I}\alpha_i}{|\mathcal I|}
}
(unless $k=4$, $|\mathcal I|=3$ in which case we stay on the right of such pole).
Notice that Lemma~\ref{DIbound2} guarantees the convergence of the sum over $\ell$ on the new line of integration. Also, a quick computation shows that if $\mathcal I\neq I_k:=\{1,\dots, k\}$ (and $|\mathcal I|>|\mathcal J|+1$) then
\est{
\frac12|\mathcal I|+|\mathcal J|c_s-1\leq\frac k2-\frac 32 +\iota_k+9\frac \eps k 
}
where $\iota_k=\frac3{14}$ if $k=4$ and $\iota_k=0$ otherwise. In particular, if  $\mathcal I\neq I_k$, then by~\eqref{shjgk} the contribution of the integral on the new line of integration is $O(q^{\frac k2-\frac32+\iota_k+A\eps})$ and we obtain
\es{\label{aedw1}
\mathscr V_{\mathcal I;\balpha,\bbeta}=\mathscr X_{\mathcal I;\balpha,\bbeta}+O(q^{\frac k2-\frac32+\iota_k+A\eps}),
}
where 
\est{
\mathscr X_{\mathcal I;\balpha,\bbeta}&:=-\frac{2^{|\mathcal J|}i^{|\Upsilon\cap \mathcal J|}}{|\mathcal I|}\sum_\ell\sumstar_{h\mod \ell} \frac{q^{\frac12|\mathcal I|+|\mathcal J|s'-1}}{\ell^{\sum_{i\in \mathcal I}(1-\alpha_i+\beta_i)}}  F(0,\tfrac{h}\ell)
 \sin\prBig{\frac{\pi }{2}(|\Upsilon\cap\mathcal I|-1)}\\
&\quad \times  \bigg(\prod_{j\in \mathcal J}D_{j;\alpha_j,\beta_j}\pr{\tfrac12+s',\tfrac{ h}{\ell}}\bigg)\bigg(\prod_{i\in \mathcal I}\frac{(2\pi)^\frac12}{2^{s'+\alpha_i}q^{\alpha_i}}\frac{\Gamma_i(\frac14-\frac {\alpha_i+s'}2)}{\Gamma_i(\frac14+\frac {\alpha_i+s'}2)}\bigg)   \frac{H''_{\mathcal I;\balpha,\bbeta}(s')}{\pi s'}. \\
}
(Notice that~\eqref{aedw1} holds also in the case $k=4$, $|\mathcal I|=3$, since from~\eqref{fef12} and a trivial bound it follows that $\mathscr X_{\mathcal I;\balpha,\bbeta}$ is convergent and $O(q^{\frac23+A\eps})=O(q^{\frac k2-\frac32+\iota_k+A\eps})$).

If $\mathcal I=I_k$, then 
\est{
\mathscr V_{I_k;\balpha,\bbeta}=\mathscr X_{I_k;\balpha,\bbeta}+\mathscr V'_{I_k;\balpha,\bbeta},
}
where $\mathscr V'_{I_k;\balpha,\bbeta}$ is as in~\eqref{dfys}, but with the line of integration $c_s=9\eps/k$.

Now, notice that if $|\Upsilon\cap \mathcal J|$ is odd  (and thus so is $|\Upsilon\cap \mathcal I|$ since $|\Upsilon|$ is even), then the sine in the expression defining $\mathscr X_{\mathcal I;\balpha,\bbeta}$ is equal to $0$ and thus so is $\mathscr X_{\mathcal I;\balpha,\bbeta}$. If $|\Upsilon\cap \mathcal J|$ is even, then changing $h$ into $-h$, $\epsilon_\mathcal J$ into $-\epsilon_\mathcal J$ and using the identity $F(0,h/\ell)+F(0,-h/\ell)=-1$ (which follows immediately from~\eqref{fef12}), we obtain $\mathscr X_{\mathcal I;\balpha,\bbeta}=-\mathscr X_{\mathcal I;\balpha,\bbeta}+\mathscr K_{\mathcal I;\balpha,\bbeta}$ and so
$\mathscr X_{\mathcal I;\balpha,\bbeta}=\frac12\mathscr K_{\mathcal I;\balpha,\bbeta}$,
where
\est{
\mathscr K_{\mathcal I;\balpha,\bbeta}&:=-\frac{2^{|\mathcal J|}}{|\mathcal I|}\sum_\ell\sumstar_{h\mod \ell} \frac{q^{\frac12|\mathcal I|+|\mathcal J|s'-1}}{\ell^{\sum_{i\in \mathcal I}(1-\alpha_i+\beta_i)}}  \bigg(\prod_{j\in \mathcal J}D_{j;\alpha_j,\beta_j}\pr{\tfrac12+s',\tfrac{ h}{\ell}}\bigg)
 \\
&\quad \times  \bigg(\prod_{i\in \mathcal I}\frac{(2\pi)^\frac12}{2^{s+\alpha_i}q^{\alpha_i}}\frac{\Gamma_i(\frac14-\frac {\alpha_i+s'}2)}{\Gamma_i(\frac14+\frac {\alpha_i+s'}2)}\bigg)   \frac{H''_{\mathcal I;\balpha,\bbeta}(s')}{\pi s'} \\
&\,=-\mathscr D'_{\mathcal I;\balpha,\bbeta},
}
where $\mathscr D'_{\mathcal I;\balpha,\bbeta}$ is as in~\eqref{mtmt},
since $\frac12|\mathcal I|+|\mathcal J|s'-1=ks'+(\frac{1}{2}-s')|\mathcal I|-1=ks'+\sum_{i\in\mathcal I}\alpha_i$ and by the definition~\eqref{raf} of $H''_{\mathcal I;\balpha,\bbeta}(s')$.
It follows that
\est{
\sum_{\substack{\mathcal I\cup \mathcal J=\{1,\dots,k\},\\  \mathcal I\cap\mathcal J=\emptyset,\\  \frac{k}2+\frac34<|\mathcal I|}}\,\sum_{\substack{\{\alpha_i^\prime,\beta_i^\prime\}=\{\alpha_i,\beta_i\}\ \forall i\in\mathcal I,\\ (\alpha_j^\prime,\beta_j^\prime)=(\alpha_j,\beta_j)\ \forall j\in\mathcal J}}2\mathscr{V}_{\mathcal I;\balpha,\bbeta}
&=\sum_{\substack{\{\alpha_i^\prime,\beta_i^\prime\}=\{\alpha_i,\beta_i\}\ \forall i\in I_k}} 2\mathscr V'_{I_k;\balpha,\bbeta},-\mathscr D_{\balpha,\bbeta}+O(q^{\frac k2-\frac12+\iota_k+A\eps})
}
and thus to conclude the proof of Lemma~\ref{smtmt}, we just need to show $\mathscr V'_{I_k;\balpha,\bbeta}+X_{\bbeta,\balpha}\mathscr V'_{I_k;-\bbeta,-\balpha}=\mathscr M_{\balpha,\bbeta}$ with $\mathscr M_{\balpha,\bbeta}$ as in~\eqref{gabcd}. First, we need the following lemma.

\begin{lemma}
For $\Re(s_1+s_1)>2$ and $\Re(s_2)>1$ we have
\es{\label{aq4}
\sum_\ell\frac1{\ell^{s_2}}\sumstar_{h\mod \ell} F(s_1,\tfrac{h}\ell)=\frac{\zeta(s_1)\zeta(s_1+s_2-1)}{\zeta(s_2)}.
}
\end{lemma}
\begin{proof}
From the functional equation for $F(s,x)$ and the Phragm\'en-Lindel\"of principle one sees that if $|s_1-1|>\eps'>0$, then if 
\est{
\sumstar_{h\mod \ell} |F(s_1,\tfrac{h}\ell)|\ll_{s,\eps,\eps'} 1+\ell^{1-\Re(s_1)+\eps}
} 
for all $\eps>0$. It follows that the left hand side of~\eqref{aq4} defines a meromorphic function in $s_1,s_2$ on the region $\Re(s_2)>1$, $\Re(s_1+s_2)>2$. Now, assume $\Re(s_1),\Re(s_2)>1$. 
Expanding $F$ into its Dirichlet expansion~\eqref{dff}, executing the sum over $h$, and using~\eqref{hga}, we obtain
\est{
\sum_\ell\frac1{\ell^{s_2}}\sumstar_{h\mod \ell} F(s_1,\tfrac{h}\ell)&=\sum_\ell\frac1{\ell^{s_2}}\sum_{n}\frac{c_{\ell}(n)}{n^{s_1}}=\frac1{\zeta(s_2)}\sum_{n}\frac{\sigma_{1-s_2}(n)}{n^{s_1}}=\frac{\zeta(s_1)\zeta(s_1+s_2-1)}{\zeta(s_2)}.
}
The Lemma then follows by analytic continuation.
\end{proof}

Applying this Lemma,  we see that
\est{
\mathscr V'_{I_k;\balpha,\bbeta}&=-\frac{q^{\frac k2-1}}{2\pi i}\int_{(c_s)} \Gamma\Big(1+ks-\frac k2+\sum_{i=1}^k\alpha_i\Big)  
 \frac{\zeta(1+ks-\frac k2+\sum_{i=1}^k\alpha_i)\zeta(\frac k2+ks+\sum_{i=1}^k\beta_i)}{\zeta(k-\sum_{i=1}^k(\alpha_i-\beta_i))}\\
 &\quad \times  
 \sin\prBig{\frac{\pi }{2}\Big(|\Upsilon|-\frac k2+ks+\sum_{i=1}^k\alpha_i\Big)} \bigg(\prod_{i=1}^k\frac{\Gamma_i(\frac14-\frac {\alpha_i+s}2)}{\Gamma_i(\frac14+\frac {\alpha_i+s}2)}\frac{\pi^{\frac12}q^{-\alpha_i}}{2^{s-\frac12+\alpha_i}}\bigg) \frac{H''_{I_k;\balpha,\bbeta}(s)}{\pi s} ds,\\
}
so that by the functional equation (using that $|\Upsilon|$ is even) and the definition~\eqref{raf} of $H''$ we obtain
\est{
\mathscr V'_{I_k;\balpha,\bbeta}
&=(-1)^{\frac{|\Upsilon|}2}\frac{q^{\frac k2-1}}{2\pi i}\int_{(c_s)}  
 \frac{\zeta(\frac k2-ks-\sum_{i=1}^k\alpha_i)\zeta(\frac k2+ks+\sum_{i=1}^k\beta_i)}{\zeta(k-\sum_{i=1}^k(\alpha_i-\beta_i))}\\
 &\quad \times  G_{\balpha,\bbeta}(s)\bigg(\prod_{i=1}^k\zeta(1-\alpha_i+\beta_i)\frac{\Gamma_i(\frac14-\frac {\alpha_i+s}2)\Gamma_i(\frac14+\frac {\beta_i+s}2)}{\Gamma_i\prst{\tfrac {\frac12+\alpha_i}2}\Gamma_i\pr{\tfrac {\frac12+\beta_i}2}}\pr{\frac{q}\pi}^{-\alpha_i}\bigg) \frac{ds}s.\\
 }
Notice that changing $s$ into $-s$ we obtain exactly $-X_{\bbeta,\balpha}$ times the analogous term coming from $\mathscr V'_{I_k;-\bbeta,-\balpha,}$ but with line of integration $c_s=-9\frac \eps k$. Thus, $\mathscr V'_{I_k;\balpha,\bbeta}+X_{\bbeta,\balpha}\mathscr V'_{I_k;-\bbeta,-\balpha,}$ coincides with the residue of the above integral at $s=0$, that is
\est{
\mathscr V'_{I_k;\balpha,\bbeta}+X_{\bbeta,\balpha}\mathscr V'_{I_k;-\bbeta,-\balpha,}&=(-1)^{\frac{|\Upsilon|}2}  q^{\frac k2-1}
 \frac{\zeta(\frac k2-\sum_{i=1}^k\alpha_i)\zeta(\frac k2+\sum_{i=1}^k\beta_i)}{\zeta(k-\sum_{i=1}^k(\alpha_i-\beta_i))}\\
 &\quad \times  \prod_{i=1}^k\zeta(1-\alpha_i+\beta_i)\frac{\Gamma_i(\frac14-\frac {\alpha_i}2)}{\Gamma_i\prst{\frac14+\tfrac {\alpha_i}2}}\pr{\frac{q}\pi}^{-\alpha_i}.\\
}
Thus, Lemma~\ref{smtmt} follows.  
\comment{
(we had a factor of $\frac{i^{|\Upsilon|}}{2^{k}}$ and a $2$ from~\eqref{tac}).
}


\section{The terms far from the diagonal}\label{ttftd}
We will use the following result of Young, to prove Lemma~\ref{pfofa}.
\begin{lemma}\label{YL}
Let $L,K\ll q^{1+\eps}$ and let $W$ be a smooth function with compact support on $\R_{>0}$. Then,
\est{
\sum_{0< \ell < L}\bigg|\sum_{\substack{(k,q)=1}} \e{\frac{\ell \overline k}q}W\pr{\frac{k}{K}}\bigg|\ll 
L^\frac12q^{\frac34+\eps}+q^{\eps}K^\frac12L.
}
\end{lemma}
\begin{proof}
This is Proposition~4.3 of~\cite{You11b} (notice that we removed the condition $(q,\ell)=1$ from the first sum; one can easily check the bound holds also in this case).
\end{proof}
\begin{proof}[Proof of Lemma~\ref{pfofa}]
For simplicity we shall take $\balpha=\bbeta=\mathbf 0:=(0,\dots,0)$, as the shifts don't play any role in this Lemma and the same argument with obvious modifications works also when $\alpha_i,\beta_i\ll 1/{\log q}$.

By symmetry, we can assume that $N_1$ is the maximum of the $N_i$ and that $N_2$ is the second largest. Also, we assume $N_1\cdots N_k\ll q^{ k+\eps}$ and $N_1\gg q^{1+3\eps}$, since otherwise the result is trivial.

Now, we start by observing that we can remove the condition $\pm_1n_1+\cdots +\pm_kn_k\neq 0$ in $\mathcal{O}''_{\epsilon}:=\mathcal{O}''_{\epsilon,\mathbf 0,\mathbf 0}$ at a cost of an admissible error:
\est{
\mathcal{O}''_{\epsilon}&=\sum_{d|q}d\frac{\mu(\tfrac qd)}{\varphi(q)}\sum_{\substack{d|(  \pm_1n_1\pm_2\cdots \pm_kn_k)}} 
\hspace{-0.5em}\frac{d(n_1)\cdots d(n_k)}{(n_1\cdots n_k)^{\frac12}}V_{\balpha,\bbeta}\pr{\frac{n_1\cdots n_k}{q^k}}P\prBig{\frac {n_{1}}{N_1}}\cdots P\prBig{\frac {n_{k}}{N_k}}\\
&\quad+O\prbig{q^{A\eps}(N_1\cdots N_k)^{\frac12}/N_1}.
}
Next, we decompose $n_1$ into $n_1=f_1g_1$ and attach to the new variables two partitions of unity so that  $f_1\asymp F_1, g_1\asymp G_1$ with $F_1\geq G_1$ and $F_1G_1\asymp N_1$.  Writing $V$ and $P(\frac {n_{1}}{N_1})$ in terms of their Mellin transform, we obtain
\es{\label{rv2a}
\mathcal{O}''_{\epsilon}&=\!\sumdagger_{\substack{F_1G_1\asymp N_1,\\ G_1\leq F_1}}\int'_{(c_s,c_{u_1})}q^{ks}N_1^{u_1}\hspace{-.7em}\sum_{\substack{0\leq |m|<M}}\hspace{-0.7em}w_m(s)\mathscr{A}_m(s+u_1)G_{\balpha,\bbeta}(s) g_{\balpha,\bbeta}\tilde P(u_1)\frac{ds}sdu_1\\
&\quad+O\prbig{q^{A\eps}(N_1\cdots N_k)^{\frac12}/N_1},\\
}
where $M:= \min(2kN_2,q)$, the $\int'$ indicates that the integrals are truncated at $|u_1|,|s|\leq q^{\eps}$, the lines of integrations are $c_{u_1}=0$, $c_s=\eps/k$, and
\est{
\mathscr{A}_m(s)&:=\sum_{d|q}d\,\frac{\mu(\tfrac qd)}{\varphi(q)}\sum_{g_1}\sum_{ \substack{f_1g_1\equiv m  \mod d}}\frac1{(f_1g_1)^{\frac12+s}}P\prBig{\frac {f_1}{F_1}}P\prBig{\frac {g_1}{G_1}},\\
w_m(s)&:=\sum_{\pm_2n_2\pm_3\cdots\pm_k n_k\equiv -m\mod q}\frac{d(n_2)\cdots d(n_k) }{(n_2\cdots n_k)^{\frac12+s}}P\prBig{\frac{n_2}{N_2}}\cdots P\prBig{\frac{n_k}{N_k}}.\notag
}
Now, we apply Poisson's summation formula to the sum over $f_1$ and we see that for $\Re(s)= \eps/k$
\est{
\mathscr{A}_m(s)&=\sum_{d|q}\,\frac{\mu(\tfrac qd)}{\varphi(q)} \sum_{g_1}\frac{P\prst{ {g_1}/{G_1}}}{g_1^{\frac12+s}}\hspace{-0.6em}\sum_{0\leq |\ell|\leq \frac{dq^{A\eps}}{(d,g_1)F_1}}\hspace{-0.6em}\e{\frac{\ell m\overline{g_1/(d,g_1)}}{d/(d,g_1)}}\hspace{-0.2em}\int_{0}^{\infty}\frac{P\prst{ {x}/{F_1}}}{x^{\frac12+s}}\e{-\frac{\ell x}{d}}\,dx+O(q^{-1})\\
&=\mathscr{A}^*_m(s)+O(q^{-1}),
}
where 
\est{
\mathscr{A}^*_m(s)&=\frac{F_1^{\frac12-s}}{\varphi(q)} \sum_{(g_1,q)=1}\frac{P\prst{ {g_1}/{G_1}}}{g_1^{\frac12+s}}\sum_{0< |\ell|\leq L}\e{\frac{\ell m\overline{g_1}}{q}}\int_{0}^{\infty}\frac{P\prst{ {x}}}{x^{\frac12+s}}\e{-\frac{\ell F_1 x}{q}}\,dx
}
and $L= q^{1+A\eps}/F_1$. Indeed, the sum over $\ell$ in the $d=1$ summands contains only the term $\ell=0$ which cancel with the $\ell=0$ term from $d=q$ (notice also that $(g_1,q)>1$ implies $\ell=0$).
Thus,~\eqref{rv2a} becomes
\est{
\mathcal{O}''_{\epsilon}&=\sumdagger_{\substack{F_1G_1\asymp N_1,\\ G_1\leq F_1}}\int'_{(c_s,c_{u_1})}q^{ks}N_1^{u_1}\sum_{\substack{0\neq |m|< M}}w_m(s)\mathscr{A}_m^*(s+u_1)G_{\balpha,\bbeta}(s) g_{\balpha,\bbeta}\tilde P(u_1)\frac{ds}sdu_1\\
&\quad+O(q^{\frac k2-\frac32+A\eps}+q^{A\eps}(N_1\cdots N_k)^{\frac12}/N_1),
}
since the contribution of the terms with $m=0$ can be bounded trivially by
\est{
&\ll \frac{L(F_1G_1)^\frac12}q(1+N_2/q)N_2^{-\frac12}(N_3\cdots N_k)^{\frac12}q^{A\eps}\\
&\ll q^{A\eps}N_2^{-\frac12}(N_3\cdots N_k)^{\frac12}+q^{-1+A\eps}(N_2\cdots N_k)^\frac12\ll q^{\frac k2-\frac32+A\eps}+q^{A\eps}(N_1\cdots N_k)^{\frac12}/N_1.\\
}
since $N_2^{-1}N_3\cdots N_k\leq N_4\cdots N_k\ll q^{k-3+A\eps}$. 
Also, we assume $N_1\leq F_1^2\leq q^{2+2A\eps}$, since otherwise $\mathscr{A}^*_m(s)$ is identically zero.

Changing the order of summation and integration and bounding trivially $G_{\balpha,\bbeta}(s) g_{\balpha,\bbeta}\tilde P(u_1)$, we see that 
\es{\label{rv3}
\mathcal{O}_{\epsilon}''&\ll \hspace{-0em}\sumdagger_{\substack{F_1G_1\asymp N_1\\ G_1\leq F_1}}\int'_{(c_s,c_{u_1})}q^{\eps}|E(s,u_1)|
\,|dsdu_1|+q^{\frac k2-\frac32+A\eps}+q^{A\eps}\frac{(N_1\cdots N_k)^{\frac12}}{N_1},
}
where
\est{
E(s,u_1)&:=\frac{F_1^{\frac12}}{\varphi(q)}\sum_{0< |\ell|\leq L} \sum_{0< |m|< M}|w_m(s)|\left|\sum_{(g_1,q)=1}\frac{P\prst{ {g_1}/{G_1}}}{g_1^{\frac12+s+u_1}}\e{\frac{\ell m\overline{g_1}}{q}}
\right|\\
&\ll \frac{F_1^{\frac12}}{qG_1^{\frac12}}\max_{0< |r|\leq R}c_r\sum_{0\neq |r|\leq R}\Bigg|\sum_{(g_1,q)=1}P\pr{\frac {g_1}{G_1}}\pr{\frac{G_1 }{g_1}}^{\frac12+s+u_1}\e{\frac{r\overline{g_1}}{q}}
\Bigg|,\\
}
with $R:= \min(2kLN_2,q)\leq 2k q^{1+A\eps}\min(N_2/F_1,1)$ and
\est{
c_r&:=\sum_{\substack{\ell m\equiv r\mod q ,\\  0<|m|\leq M,\ 0< |\ell|\leq L}}|w_{m}(s)|\ll  \sum_{\substack{0<|\ell|\leq L,n_2\asymp N_2,\dots,n_k\asymp N_k,\\(\pm_2n_2\pm_3\cdots\pm_k n_k)\ell\equiv -r\mod q}}q^{A\eps}d(n_2)\cdots d(n_2)(N_2\cdots N_k)^{-\frac12}\\
&\ll \sum_{\substack{0<|\ell|\leq L, |n|\ll k N_2,\\n\ell\equiv -r\mod q}}q^{A\eps} N_2^{-\frac12}(N_3\cdots N_k)^{\frac12}
\ll \sum_{\substack{|a|\ll k LN_2,\\ a \equiv -r\mod q}}d(a)q^{A\eps} N_2^{-\frac12}(N_3\cdots N_k)^{\frac12}\\
&\ll q^{A\eps} N_2^{-\frac12}(N_3\cdots N_k)^{\frac12}\pr{1+LN_2/q}.\\
}
Thus, by Lemma~\ref{YL}, for $|s|,|u_1|\ll q^{A\eps}$, $\Re(s)=\eps/k,$ $\Re(u_1)=0$ we have
\est{
E(s,u_1)&\ll q^{-1+A\eps} F_1^{\frac12}(G_1N_2)^{-\frac12}(N_3\cdots N_k)^{\frac12}\pr{1+LN_2/q}\pr{R^\frac12q^{\frac34}+G_1^\frac12R}\\
&\ll q^{A\eps} (F_1/G_1N_2)^{\frac12}(N_3\cdots N_k)^{\frac12} \pr{1+N_2/F_1} \min \pr{F_1^{-\frac12}N_2^\frac12q^{\frac14}+G_1^\frac12N_2/F_1,q^{\frac14}+G_1^\frac12}\\
&= q^{A\eps} (N_3\cdots N_k)^{\frac12} \pr{1+N_2/F_1} \min \pr{\frac{q^{\frac14}}{G_1^\frac12}+\frac{N_2^\frac12}{F_1^{\frac12}},\frac{F_1^{\frac12}}{N_2^{\frac12}}\frac{q^{\frac14}}{G_1^{\frac12}}+\frac{F_1^{\frac12}}{N_2^{\frac12}}}.\\
}
For $x,y>0$  we have $(1+{x}^2)\min(y+x,y/x+1/x)\leq(x+1)(y+1)$, whence
\es{\label{tresc}
E(s,u_1)
&\ll q^{A\eps} (N_3\cdots N_k)^{\frac12} \prbig{N_2^{\frac12}q^{\frac14}N_1^{-\frac12}+N_2^{\frac12}F_1^{-\frac12}+q^{\frac14+\eps}G_1^{-\frac12}+1},\\
&\ll q^{A\eps} (N_3\cdots N_k)^{\frac12} \prbig{N_2^{\frac12}N_1^{-\frac14}+q^{\frac34}N_1^{-\frac12}+1},\\
&\ll q^{A\eps}\prBig{(N_2\cdots N_k)^{\frac12} N_1^{-\frac14}+ (N_2\cdots N_k)^{\frac12-\frac{1}{2(k-1)}}\prbig{q^{\frac34}N_1^{-\frac12}+1} },\\
}
where in the second inequality we used that $N_1\gg q$, $F_1^{-\frac12}\leq N_1^{-\frac14}$, $G_1^{-\frac12}\asymp (F_1/N_1)^\frac12\leq N_1^{-\frac12}q^{\frac12+A\eps}$ and $N_1\gg q$, and in the third one that $N_3\cdots N_k\leq N_2^{k-2}$ (so that $N_3\cdots N_k\leq (N_2\cdots N_k)^{\frac{k-2}{k-1}}$).
The lemma then follows by inserting~\eqref{tresc} in~\eqref{rv3}.
\end{proof}
\comment{ 
------------------
CHECK FROM HERE
------------------

\begin{remark}
Using only Weil's bound: (not quite, because it's not complete, but more or less)
\est{
E(s,u_1)&\ll q^{k\eps} F_1^{\frac12}q^{-1}(G_1N_2)^{-\frac12}(N_3\cdots N_k)^{\frac12}\pr{1+kLN_2/q}q^{\frac12}R\\
&\ll q^{k\eps} F_1^{\frac12}q^{-1}(G_1N_2)^{-\frac12}(N_3\cdots N_k)^{\frac12}\pr{1+N_2/F_1}q^{\frac12}\min(qN_2/F_1,q)\\
&\ll q^{\frac12+k\eps} F_1^{\frac12}(G_1N_2)^{-\frac12}(N_3\cdots N_k)^{\frac12}N_2/F_1\\
&\ll q^{\frac12+k\eps} N_1^{-1}(N_1\cdots N_k)^{\frac12},\\
}
which is non-trivial for $N_1\gg q^{\frac32+\eps}$.

Using only Weil's, in Lemma~\ref{ctd} we would get the error term
\est{
O\pr{k^{Ak}\frac{q^{-1+k\eps}N_{\tn{max}}^{\frac34+\frac k2+\eps}}{(N_1\cdots N_k)^\frac12}+\cdots},
}
which is non-trivial for $N_1\ll q^{\frac{k}{\frac34+\frac k2}}$. Thus, we get the result if $q^{\frac{k}{\frac34+\frac k2}}>q^{\frac32}$, that is if $\frac{k}4>\frac98$, i.e. $k>\frac92$. Using Kloosterman's bound for Kloosterman sums, we would need $q^{\frac{k}{\frac78+\frac k2}}>q^{\frac74}$. Thus, $\frac k8>\frac{49}{32}$, i.e. $k>\frac{49}{4}$. Similarly, with a saving of $q^\delta$: $q^{\frac{k}{1-\frac{\delta}2+\frac k2}}>q^{2-\delta}$, i.e. $q^{\frac{k\delta}2}>q^{(2-\delta)(1-\delta/2)}$, i.e. $k>\frac {(2-\delta)^2}{\delta}$
\end{remark}

} 

\section{A Mellin formula}\label{amell}
In this section we prove a formula to separate the variables in expressions of the form $(\pm_1x_1\pm_2\cdots\pm_\kappa x_\kappa )^{-s}$ which generalizes the Mellin transforms given in the following Lemma.
\begin{lemma}
Let $x,y>0$. Then
\es{\label{ffor}
(x+y)^{-b}=\frac1{2\pi i}\int_{(c_v)}\frac{\Gamma(v)\Gamma(b-v)}{\Gamma(b)}x^{v-b}y^{-v}dv,
}
for $0<c_{v}<\Re(b)$. Moreover, for $\Re(b)<0<c_{w}$, we have
\es{\label{ffor2}
(x-y)^{-b}\chi_{\R_{>0}}(x-y)=\frac1{2\pi i}\int_{(c_w)}\frac{\Gamma(w)\Gamma(1-b)}{\Gamma(1-b+w)}x^{w-b}y^{-w}dw,
}
where $\chi_X(x)$ is the indicator function of the set $X$.

\end{lemma}
Equation~\eqref{ffor} can be used repeatedly to give a formula for $(x_1+\cdots+x_\kappa )^{-s}$ valid for $\Re(s)>0$. However, it is not straightforward to obtain a satisfactory formula valid in the case when there are some minus signs, as the integrals obtained by repeatedly applying~\eqref{ffor} and~\eqref{ffor2} are not absolutely convergent. The following Lemma overcomes this problem by introducing an extra integration.

\begin{lemma}\label{sml}
Let $\kappa \geq2$ and  $x_1,\dots x_\kappa >0$. Let  $\epsilon=(\pm_1,\cdots,\pm_\kappa 1)\in \{\pm1\}^\kappa $, with $\pm_11=-1$. Let $B\in\Z_{\geq0}$ be such that $\frac \kappa 2+\frac12<\Re(v_1)< B+1$. Moreover, let $c_{v_2},\dots,c_{v_\kappa },c'_{v_2},\dots,c'_{v_\kappa }>0$ be such that
\es{\label{cfml} 
\Re(v_1)+c_{v_2}+\cdots +c_{v_\kappa }<B+1<\Re(v_1)+c'_{v_2}+\cdots+c'_{v_\kappa }.
}
Then
\es{\label{tmf}
&(\pm_2x_2\pm_3\cdots\pm_\kappa x_{\kappa })^{v_1-1}\chi_{\R_>0}(\pm_2x_2\pm_3\cdots\pm_\kappa x_{\kappa })\\
&=\sum_{\substack {\nu=(\nu_2,\dots,\nu_\kappa )\in\Z_{\geq0}^{\kappa -1},\\\nu_2+\cdots+\nu_\kappa =B,\\\nu_{i}=0\text{ if $\pm_i=-1$}}}\frac{B!}{\nu_2!\cdots \nu_\kappa !}\frac1{(2\pi i)^{\kappa -1}}\Bigg(\int_{(c_{v_2},\dots,c_{v_\kappa })}-\int_{(c'_{v_2},\dots,c'_{v_\kappa })}\Bigg)\frac{\Psi_{\epsilon,B}(v_1,\dots,v_\kappa )}{x_2^{v_2-\nu_2}\cdots x_\kappa ^{v_\kappa -\nu_\kappa }}dv_2\cdots dv_\kappa ,\\[-1.5em]
}
where
\es{\label{dfgfs}
\Psi_{\epsilon,B}(s_1,\dots,s_\kappa )
&:=
\frac{\Gamma(s_1)\cdots\Gamma(s_\kappa )}{\Gamma(V_{+;\epsilon}(s_1,\dots,s_\kappa ))\Gamma(V_{-;\epsilon}(s_1,\dots,s_\kappa ))}\frac{G(B+1-s_1-\cdots-s_\kappa )}{B+1-s_1-\cdots-s_\kappa },\\
V_{\pm;\epsilon}(v_1,\dots,v_\kappa )&:=\sum_{\substack {1\leq i\leq \kappa ,\\\pm_i1=\pm1}}v_{i}
}
and $G(s)$ is any entire function such that $G(0)=1$ and $ G(\sigma+it)\ll e^{-C_1|t|}(1+|\sigma|)^{C_2|\sigma|}$ for some fixed $C_1,C_2>0$.

\end{lemma}
\begin{remark}
If $\epsilon=(-1,\dots,-1)$, then $\Psi_{\epsilon}$ has to be interpreted as being identically zero.
%
\end{remark}
\begin{remark}
If $\xi(s)$ is Riemann $\xi$-function, then $G(s)=\xi(s)/\xi(0)$ satisfies the hypothesis of the Lemma.
%
\end{remark}

Before giving a proof for Lemma~\ref{sml}, we give the following Lemma which implies that the integrals in~\eqref{tmf} are absolutely convergent.

\begin{lemma}\label{bfp}
Let $s_i=\sigma_i+it_i$ for $i=1,\dots,\kappa $. Then, for some $A>0$ we have
\es{\label{feg}
\Psi_{\epsilon,B}(s_1,\dots,s_\kappa )&\ll \frac1{\delta^\kappa }\frac{(1+B+|\sigma_1|+\cdots+|\sigma_\kappa |)^{A(1+B+|\sigma_1|+\cdots+|\sigma_\kappa |)}}{(1+|t_1|)^{\frac12-\sigma_1}\cdots (1+|t_\kappa |)^{\frac12-\sigma_\kappa }(1+|t_1|+\cdots+|t_\kappa |)^{\sigma_1+\dots+\sigma_\kappa -1}},\\
}
provided that the $s_i$ are located at a distance greater than $\delta>0$ from the poles of $\Psi_{\epsilon}$.
\end{lemma}
\begin{proof}
 By Stirling's formula (and the reflection's formula for the Gamma function), if the distance of $s=\sigma+it$ from the poles of $\Gamma(s)$ is greater than $\delta$, then we have
%
\est{
\Gamma(s)&\ll \frac1{\delta}(1+A_1|\sigma|)^{|\sigma|}(1+|t|)^{\sigma-\frac12}e^{-\frac{\pi}{2}|t|}, \\
\Gamma(s)^{-1}&\ll (1+A_1|\sigma|)^{|\sigma|}(1+|t|)^{-\sigma+\frac12}e^{\frac{\pi}{2}|t|}, 
}
for some $A_1>0$. It follows that 
\est{
\Psi_{\epsilon,B}(s_1,\dots,s_\kappa )&\ll \frac1{\delta^\kappa }\frac{(1+B+|\sigma_1|+\cdots+|\sigma_\kappa |)^{A_2(1+B+|\sigma_1|+\cdots+|\sigma_\kappa |)}}{(1+|t_1|)^{\frac12-\sigma_1}\cdots (1+|t_\kappa |)^{\frac12-\sigma_\kappa }}\times\\
&\quad\times \frac{e^{-\frac{\pi}2(|t_1|+\cdots+|t_\kappa |-|V_{+;\epsilon}(t_1,\dots,t_\kappa )|-|V_{-;\epsilon}(t_1,\dots,t_\kappa )|)-C_1|t_1+\cdots + t_\kappa |}}{(1+|V_{+;\epsilon}(t_1,\dots,t_\kappa )|)^{V_{+;\epsilon}(\sigma_1,\dots,\sigma_\kappa )-\frac12}(1+|V_{-;\epsilon}(t_1,\dots,t_\kappa )|)^{V_{-;\epsilon}(\sigma_1,\dots,\sigma_\kappa )-\frac12}},\\
}
for some $A_2>0$. 
Now, we have
\est{
\frac{e^{-C_1|x+y|}}{(1+|x|)^{\eta_1}(1+|y|)^{\eta_2}}\ll \frac{(1+|\eta_1|+|\eta_2|)^{A_3(|\eta_1|+|\eta_2|)}}{(1+|x|+|y|)^{\eta_1+\eta_2}},
}
for some $A_3>0$ (depending on $C_1$). 
Thus, 
\est{
& \frac{e^{-\frac{\pi}2(|t_1|+\cdots+|t_\kappa |-|V_{+;\epsilon}(t_1,\dots,t_\kappa )|-|V_{-;\epsilon}(t_1,\dots,t_\kappa )|)-C_1|t_1+\cdots + t_\kappa |}}{(1+|V_{+;\epsilon}(t_1,\dots,t_\kappa )|)^{V_{+;\epsilon}(\sigma_1,\dots,\sigma_\kappa )-\frac12}(1+|V_{-;\epsilon}(t_1,\dots,t_\kappa )|)^{V_{-;\epsilon}(\sigma_1,\dots,\sigma_\kappa )-\frac12}}\\
&\qquad\ll(1+|\sigma_1|+\dots+|\sigma_\kappa |)^{A_4(|\sigma_1|+\dots+|\sigma_\kappa |)} \frac{e^{-\frac{\pi}2(|t_1|+\cdots+|t_\kappa |-|V_{+;\epsilon}(t_1,\dots,t_\kappa )|-|V_{-;\epsilon}(t_1,\dots,t_\kappa )|)}}{(1+|V_{+;\epsilon}(t_1,\dots,t_\kappa )|+|V_{-;\epsilon}(t_1,\dots,t_\kappa )|)^{\sigma_1+\cdots+\sigma_\kappa -1}}\\
&\qquad\ll\frac{(1+|\sigma_1|+\dots+|\sigma_\kappa |)^{A_5(|\sigma_1|+\dots+|\sigma_\kappa |)}}{(1+|t_1|+\cdots+|t_\kappa |)^{\sigma_1+\cdots+\sigma_\kappa -1}},\\
}
and~\eqref{feg} follows.
\end{proof}

\begin{proof}[Proof of Lemma~\ref{sml}]
First, we remark  that the estimate~\eqref{feg} implies the absolute convergence of the integrals on the right hand side of~\eqref{tmf} and justifies the following computations.

We prove the Lemma by induction. First we consider the case $\kappa =2$. From~\eqref{ffor} we have 
\es{\label{fffd}
(x_2+x_3)^{v_1-1}&=(x_2+x_3)^B(x_2+x_3)^{v_1-1-B}\\
&=\sum_{\substack {\nu_2,\nu_3\in\Z_{\geq0},\\ \nu_2+\nu_3=B}}\frac{B!}{\nu_2! \nu_3!}x^{\nu_2}x_3^{\nu_3}\frac1{2\pi i}\int_{(c_{v_3})}\frac{\Gamma(v_3)\Gamma(1+B-v_1-v_3)}{\Gamma(1+B-v_1)x_2^{B+1-v_1-v_3}x_3^{v_3}}dv_3,
}
for $0<c_{v_3}<1+B-\Re(v_1)$. 
Now, by contour integration,
\est{
\frac{\Gamma(1+B-v_1-v_3)}{\Gamma(1+B-v_1)}x_2^{v_1+v_3-B-1}=\frac1{2\pi i}\bigg(\int_{(c_{v_2})}-\int_{(c'_{v_2})}\bigg)\frac{\Gamma(v_2)x_2^{-v_2}}{\Gamma(v_2+v_3)}\frac{G(B+1-v_1-v_2-v_3)}{B+1-v_1-v_2-v_3}dv_2,
}
where $c_{v_2},c'_{v_2}>0$ and $c_{v_2}<-\Re(v_1+{v_3})+B+1<c'_{v_2}$. Inserting this into~\eqref{fffd} we obtain~\eqref{tmf} in the case $\epsilon=(-1,1,1)$.

The case $\epsilon=(-1,1,-1)$ (and thus its permutation $\epsilon=(-1,-1,1)$) follows in the same way from~\eqref{ffor2}.

Now, let $\epsilon=(-1,\pm_2,\dots,\pm_{\kappa +1})\in\{\pm1\}^{\kappa +1}$ with $\kappa \geq2$ and suppose~\eqref{tmf} holds for all $\epsilon'\in\{\pm1\}^{\kappa }$ with $\pm_1'1=-1$. Since $\kappa +1\geq3$ there are two indexes $2\leq i<j\leq \kappa +1$ such that $\pm_i1=\pm_j1$ and without loss of generality we can assume $i=\kappa ,j=\kappa +1$. Then, letting $\epsilon'=(-1,\pm_2,\dots,\pm_\kappa )$, we have
\est{
&(\pm_2x_2\pm_3\cdots\pm_{\kappa +1}x_{\kappa +1}))^{v_1-1}\chi_{\R_{>0}}(\pm_2x_2\pm_3\cdots\pm_{\kappa +1}x_{\kappa +1})=\sum_{\substack {\nu=(\nu_2,\dots,\nu_\kappa )\in\Z_{\geq0}^\kappa ,\\\nu_2+\cdots+\nu_{\kappa +1}=B,\\\nu_{i}=0\text{ if $\pm_i=-1$}}}\frac{B!}{\nu_2!\cdots \nu_\kappa !}\frac1{(2\pi i)^{\kappa -1}}\times\\
&\hspace{2.5cm}\times\Bigg(\int_{(c_{v_2},\dots,c_{v_\kappa })}-\int_{(c'_{v_2},\dots,c'_{v_\kappa })}\Bigg)\frac{\Psi_{\epsilon',B}(v_1,\dots,v_\kappa )}{x_2^{v_2-\nu_2}\cdots x_{\kappa -1}^{v_{\kappa -1}-\nu_{\kappa -1}} }(x_\kappa +x_{\kappa +1})^{-v_\kappa +\nu_\kappa }\,dv_2\cdots dv_\kappa \\
}
where $c_{v_2},\dots,c_{v_\kappa },c'_{v_2},\dots,c'_{v_\kappa }>0$ satisfy~\eqref{cfml}. Then, we expand the binomial $(x_{\kappa }+x_{\kappa +1})^{\nu_\kappa }$ and apply~\eqref{ffor} to $(x_{\kappa }+x_{\kappa +1})^{-v_\kappa }$. We obtain 
\est{&(\pm_2x_2\pm_3\cdots\pm_{\kappa +1}x_{\kappa +1}))^{v_1-1}\chi_{\R_{>0}}(\pm_2x_2\pm_3\cdots\pm_{\kappa +1}x_{\kappa +1})=\\
&\hspace{2cm}=\sum_{\substack {\nu=(\nu_2,\dots,\nu_{\kappa +1})\in\Z_{\geq0}^\kappa ,\\\nu_2+\cdots+\nu_{\kappa +1}=B,\\\nu_{i}=0\text{ if $\pm_i=-1$}}}\frac{B!}{\nu_2!\cdots \nu_{\kappa +1}!}\frac1{(2\pi i)^{\kappa }}\Bigg(\int_{(c_{v_2},\dots,c_{v_{\kappa +1}})}-\int_{(c'_{v_2},\dots,c'_{v_{\kappa +1}})}\Bigg)\times\\
&\hspace{2cm}\quad\times\frac{\Psi_{\epsilon',B}(v_1,\dots,v_\kappa )}{x_2^{v_2-\nu_2}\cdots x_{\kappa -1}^{v_{\kappa -1}-\nu_{\kappa -1}}x_\kappa ^{v_{\kappa }-v_{\kappa +1}-\nu_{\kappa }}x_{\kappa +1}^{v_{\kappa +1}-\nu_{\kappa +1}}}\frac{\Gamma(v_{\kappa +1})\Gamma(v_\kappa -v_{\kappa +1})}{\Gamma(v_{\kappa })}dv_2\cdots dv_\kappa ,\\
}
where $c_{v_{\kappa +1}},c'_{v_{\kappa +1}}>0$ are such that $0<c_{v_{\kappa +1}}<c_{v_{\kappa }}$. We make the change of variables $v_{\kappa }\rightarrow v_{\kappa }+v_{\kappa +1}$ and the lemma follows.
\end{proof}

\begin{lemma}\label{rcmpv}
Let $\kappa \geq2$ and let  $\epsilon^*=(\pm_11,\cdots,\pm_\kappa 1,\pm_*1)\in \{\pm1\}^{\kappa +1}$, with $\pm_11=-1$.
For $B\geq0$, let
\est{
\Psi_{\epsilon^*,B}'(v_1,\dots,v_{\kappa })&:=\frac{\Gamma(B+1-v_1-\cdots-v_\kappa )\Gamma(v_1)\cdots\Gamma(v_\kappa )}{\Gamma(V_{\mp_*;\epsilon}(v_1,\dots,v_{\kappa }))\Gamma(B+1-V_{\mp_*;\epsilon}(v_1,\dots,v_{\kappa }))},\\
}
where $V_{\pm;\epsilon}$ is defined in~\eqref{dfgfs}.

Let $\mathcal F(v_0,\dots,v_\kappa )$ be analytic on $\pg{(v_0,\dots,v_\kappa )\in\C^{\kappa +1}\mid 0<\Re(v_0)<B+1}$ and assume that for $0<\Re(v_0)<B+1$ and any $A>0$ one has that $\mathcal F$ satisfies
\est{
\mathcal F(v_0,\dots,v_\kappa )\ll \prod_{i=2}^\kappa (1+|v_{i}|)^{-A}
}
where the implicit constant may depend on $A$, $v_1$ and $\Re(v_0)$.
Then for any $v_1\in\C$ and $c_{v_2},\dots,c_{v_\kappa }>0$ satisfying $0<\Re(v_1)+c_{v_2}+\cdots+c_{v_\kappa }<1$ we have
\es{\label{hgs}
&\sum_{\substack {\nu}}\frac{B!}{\nu_*!\nu_2!\cdots \nu_\kappa !}
\int_{(c_{v_2},\dots,c_{v_\kappa })}\Psi_{\epsilon,B}'(v_1,\dots,v_\kappa )\times{}\\
&\hspace{3cm}\times\mathcal F(B+1-\nu_*-v_1-\cdots-v_\kappa ,v_1,v_2-\nu_2.\dots, v_\kappa -\nu_\kappa )\,dv_2\cdots dv_\kappa \\
&\hspace{1.5cm}=\int_{(c_{v_2},\dots,c_{v_\kappa })}\Psi_{\epsilon,0}'(v_1,\dots,v_{\kappa })\mathcal F(1-v_1-\cdots-v_\kappa ,v_1,\dots,v_\kappa )\,dv_2\cdots dv_\kappa,
}
where the sum on the left is taken over $\nu=(\nu_2,\dots,\nu_\kappa ,\nu_*)\in\Z_{\geq0}^\kappa $ satisfying
\est{
\nu_2+\cdots+\nu_\kappa +\nu_*=B,\qquad \nu_{i}=0\text{ if $\pm_i=-1$ or $i\in J$},\qquad \nu_{*}=0\text{ if $\pm_*=-1$}.
}\end{lemma}
\begin{proof}
Making the change of variables $v_i\rightarrow v_i+\nu_i$, for $i=2,\dots,\kappa $, moving back  the lines of integration to $c_{v_i}$ (as we can do without crossing any pole), and switching the order of summation and integration, we see that the left hand side of~\eqref{hgs} is equal to
\est{
&\int_{(c_{v_2},\dots,c_{v_\kappa })}\sum_{\substack {\nu}}\frac{B!}{\nu_*!\nu_2!\cdots \nu_\kappa !} \Psi_{\epsilon}'(v_1,v_2+\nu_2,\dots,v_\kappa +\nu_\kappa )\\
&\hspace{3cm}\times\mathcal F(1-v_1-\dots-v_\kappa ,v_1,\dots, v_\kappa )\,dv_2\cdots dv_\kappa .\\
}

Now, the identity $\B(s_1+1,s_2)+\B(s_1,s_2+1)=\B(s_1,s_2)$, satisfied by the Beta function $\B(s_1,s_2):=\Gamma(s_1)\Gamma(s_2)\Gamma(s_1+s_2)^{-1}$, can be generalized to 
\est{
\sum_{\substack {(r_1,\dots,r_m)\in\Z_{\geq0}^m,\\ r_1+\cdots+r_m=r }}\frac{r!}{r_1!\cdots r_m!}\frac{\Gamma(s_1+r_1)\cdots\Gamma(s_m+r_m)}{\Gamma(r+s_1+\cdots+s_m)}=\frac{\Gamma(s_1)\cdots\Gamma(s_m)}{\Gamma(s_1+\cdots+s_m)},
}
for $m,r\geq1$, $s_1,\dots,s_m\in\C$. 
Thus, we have
\est{
&\sum_{\substack {\nu=(\nu_2,\dots,\nu_\kappa ,\nu_*)\in\Z_{\geq0}^\kappa ,\\\nu_2+\cdots+\nu_\kappa +\nu_*=B,\\\nu_{i}=0\text{ if $\pm_i=-1$},\\\nu_{*}=0\text{ if $\pm_*=-1$}}}
\frac{B!}{\nu_*!\nu_2!\cdots \nu_\kappa !}\Psi_{\epsilon,B}(v_1,v_2+\nu_2\dots,v_\kappa +\nu_\kappa )
=\Psi_{\epsilon,0}'(v_1,\dots,v_{\kappa })
}
and the Lemma follows.
\end{proof}

\addresses{Dipartimento di Matematica, Universit\`a di Genova; via Dodecaneso 35; 16146 Genova; Italy. }

\end{document}